\documentclass[10pt]{amsart}
\usepackage[foot]{amsaddr}

\usepackage{amsmath, amssymb, amsbsy, bbm, mathrsfs}
\usepackage{array}
\usepackage[pdftex]{graphicx}
\usepackage[pdftex,colorlinks,backref=page,citecolor=blue]{hyperref}
\usepackage{tcolorbox}
\usepackage{tikz}
\usetikzlibrary{calc}
\tikzset{sdot/.style = {fill, circle, inner sep = 1pt}}
\usepackage{mathtools}
\usepackage{cleveref}
\usepackage{float}
\usepackage{crossreftools}

\newcounter{bullet}

\usepackage{setspace}
\setdisplayskipstretch{1}

\usepackage{geometry}
\usepackage{comment}
\usepackage[font=scriptsize]{caption}
\newtheorem{theorem}{Theorem}[section]
\newtheorem{corollary}[theorem]{Corollary}
\newtheorem{lemma}[theorem]{Lemma}

\newtheorem{definition}[theorem]{Definition}
\newtheorem{remark}[theorem]{Remark}
\definecolor{blue-shortcut}{rgb}{0.2, 0.2, 0.6}

\begin{document}

\title{Local Limits of Small World Networks}

\author[Y. Alimohammadi]{Yeganeh Alimohammadi}
\address{Marshall School of Business, University of Southern California}
\email{yalimoha@usc.edu}

\author[S. I\c{s}\i k]{Senem I\c{s}\i k}
\address{Management Science and Engineering, Stanford University}
\email{senemi@stanford.edu}

\author[A. Saberi]{Amin Saberi}
\address{Management Science and Engineering, Stanford University}
\email{saberi@stanford.edu}

\begin{abstract}
Small-world networks, known for high local clustering and short path lengths, are a fundamental structure in many real-world systems, including social, biological, and technological networks. We apply the theory of (marked) local convergence (also known as Benjamini-Schramm convergence) to derive the limiting behavior of the local structures for two commonly studied small-world network models: the Watts-Strogatz and the Kleinberg models. Establishing local convergence enables us to show that key network measures---clustering coefficient, PageRank, greedy maximal independent set, number of spanning trees and tree entropy---converge as network size increases, with their limits determined by the graph's local structure. Additionally, this framework facilitates the estimation of global phenomena---such as the size of the giant component under bond percolation and the closely related properties, the size of the epidemic and information cascades---using local information from small neighborhoods. Furthermore, we observe a critical change in the behavior of the limit exactly when the parameter governing long-range connections in the Kleinberg model crosses the threshold where decentralized search remains efficient, offering a new perspective on why decentralized algorithms fail in certain regimes.
\end{abstract}

\maketitle

\section{Introduction}
Many real-world networks, social, biological, and technological, exhibit small-world properties: high local clustering along with short average distances \cite{bork2004protein, humphries2006brainstem,telesford2011ubiquity}. Two canonical models capture these phenomena from complementary perspectives: the Watts-Strogatz (WS) \cite{watts} and Kleinberg \cite{kleinberg} models. The WS model starts from a ring lattice where each node is connected to its $k$ nearest neighbors on each side and rewires each edge independently with probability $\phi$ to a random node (see Figure~\ref{ws-model}).
The Kleinberg model embeds nodes on an $n\times n$ grid and adds $q$ long‑range shortcuts from each node, where the probability of a shortcut decays as a power-law distribution of lattice distance with exponent $\ell$ (see Figure~\ref{kleinberg-model}).

While both models have been extensively studied from structural and algorithmic viewpoints  \cite{barrat2000properties,fraigniaud2010searchability,giakkoupis2011optimal,kundu2014network, nguyen2005analyzing, uzzi2007small}, a unifying limit‑object description of their local geometry has been missing. Establishing such a limit through the framework of local weak convergence \cite{oded,Lyons2005,remco} 
immediately implies that any statistic or algorithm whose output depends on a bounded‑radius neighborhood converges to its counterpart on the limit. Moreover,   the scope of this perspective is broader still: a large class of dynamics and graph statistics are continuous with respect to the local topology and hence inherit convergence \cite{alimohammadi2023epidemic,banerjee2024local,lacker2019local,remco}. 
% This also enables sampling-based algorithms of some global network behavior using only local observations, without requiring access to the full graph\cite{alimohammadi2022algorithms, banerjee2024local}.
Although the theory of local convergence is well‑developed for many random graph models (e.g., inhomogeneous random graphs, configuration models and related families \cite{BolJanRio07,KomLod20,van2021local}), it has not previously been established for small‑world models.

\subsection*{Our contributions.} 
In this paper, we establish the local limits of the Watts--Strogatz and Kleinberg models. 
% For WS, the local limit of $WS(n,\phi,k)$ is a recursive \emph{$k$-Fuzz} process (Theorem~\ref{main-thm-ws}; see also Figure~\ref{full-k-fuzz}).
For the Watts--Strogatz model, we show that the local limit is a recursive object built on the infinite $k$-path: a rooted graph in which each node is connected to its $k$ nearest neighbors on both sides. Each outgoing ring edge is then independently either kept or rewired into a shortcut. These shortcuts then lead to new, independent copies of a similar local structure. We formalize this recursive limit as the \emph{$k$-Fuzz} process (Theorem~\ref{main-thm-ws}; see also Figure~\ref{full-k-fuzz}).

For the Kleinberg model, we uncover a phase transition at $\ell=2$.
When $\ell<2$, the limit is a locally tree-like \emph{$(q,k)$-patch}: informally, shortcuts leave every fixed neighborhood and do not create local cycles, so successive shortcut traversals lead to asymptotically independent regions of the lattice (Theorem~\ref{main-thm-kl-1}; see Figure~\ref{q-k patch}).
When $\ell>2$, the limit is a bounded-range kernel perturbation of the $k$-lattice (Theorem~\ref{main-thm-kl-2}).
The critical case $\ell=2$ combines features of both regimes: the shortcut structure remains tree-like as in the case $\ell<2$, but the mark dynamics collapse, so every rooted neighborhood has the same asymptotic behavior as in the case $\ell>2$ (Theorem~\ref{main-thm-kl-critical}; see Figure~\ref{q-k patch}).

Establishing these limits immediately yields convergence of a wide range of local statistics such as clustering coefficients, PageRank, number of spanning trees, greedy maximal independent sets, and allows us to derive sampling approximations for dynamics on graphs such as epidemics on networks and the size of the giant after percolation in random graphs,  discussed in Section~\ref{sec-algorithmic-app}.  
Next, we review the definition of local convergence of graphs and describe the limits of small-world models; we then discuss their applications.

\section{Local Convergence} \label{theory}

To describe the limiting objects we use the theory of local convergence \cite{AldSte04,oded}. At a high level, a sequence of graphs $(G_n)_{n\in\mathbb N}$ is said to converge locally if the empirical distribution of the neighborhood of randomly chosen nodes converges to a limit distribution. %The chosen node can be seen as a root. 

Formally, we work with \emph{rooted, locally finite, marked graphs}.  A marked rooted graph is a finite or infinite graph $G=(V,E)$ together with:
(1) a distinguished node as root $o\in V$; and
(2) a \emph{mark} attached to each node and/or edge, taking values in a separable metric space $\Xi$.
Marks encode constant-size information (direction of an edge, a vector of real numbers, etc). For $r\in\mathbb{N}$, we will let $B_r(G,o)$ denote the rooted radius-$r$ neighborhood of $o$ in $G$, including all marks and the root.

% The theory of local limits equips the space of rooted graphs with a metric which is Polish and separable (see Appendix~\ref{metric}). 
One can equip the space of rooted, marked graphs with a Polish, separable metric (see Appendix~\ref{metric}).
This metric allows us to define probability measures, and convergence of such measures. 
% To define it formally, we need the following notation: for $r\in\mathbb{N}$, let $B_r(G,o)$ denote the rooted radius-$r$ neighborhood of $o$ in $G$, including all marks and the root.

\begin{definition}[Local Convergence in Probability] \label{local-convergence-def}
Let $(G_n, M_n)_{n\geq 1}$ be a sequence of random graphs. Then $(G_n, M_n)_{n\geq 1}$  converges locally in probability to $(G,o, M)\sim \mu$, when for any $\epsilon>0$, all integers $r>0$,  and any marked rooted graph $H_* = (H_*, o_*, M_*)$, 
\begin{equation*}
    \frac{1}{|V(G_n)|}\sum_{o_n\in V(G_n)}   \mathbbm{1}_r^\epsilon((G_n,o_n)\simeq H_*)\overset{\mathbb P}{\to}\mu(  \mathbbm{1}_r^\epsilon((G,o)\simeq H_*))
\end{equation*}
where $ \mathbbm{1}_r^\epsilon$ is defined as the indicator that $B_r(G_n,o_n)$ and $B_r(H_*,o_*)$ are isomorphic, and the distance of the corresponding marks is bounded by $\epsilon$. Formally, 
\begin{align*} 
\mathbbm{1}_r^\epsilon((G_n,o_n)\simeq (H_*, o_*, M_*))=  \mathbbm{1}\Big( B_r(G_n,o_n)\simeq B_r(H_*,o_*) \text{ and}&\\ \exists \pi \text{  such that,}
\max_{u\in V(B_r(H_*,o_*))} d_\Xi(M_n(u),M_*(\pi(u)))< \epsilon
&\\
\max_{uv\in E(B_r(H_*,o_*))} d_\Xi(M_n(u,v),M_*(\pi(u,v)))< \epsilon \Big)&
\end{align*}
where $\pi$ runs over all isomorphisms between $B_r(G_n,o_n)$ and $B_r(H_*,o_*)$.
\end{definition}

This definition says that a sequence of graphs converges if the empirical frequencies of subgraphs, with an $\epsilon$-tolerance for error in the marks, converge to a limiting distribution. For more on local convergence, see the beautiful book of \cite[Chapter 2]{remco}.

\section{Main Results}
We are now ready to state the main results. For each of the small-world networks, we start by defining the models first, and then we describe the local limits.

\subsection{Watts-Strogatz (WS) Model} \label{ws-process}

The WS model \cite{watts} can be formally described as follows: Start with a cycle of \(n\) nodes. Connect each node to its \(k\) nearest neighbors on either side, forming a regular ring lattice. This structure is called a \textit{\(k\)-ring} (where \(k\) is a constant independent of \(n\)). Next, direct all edges in a clockwise direction. Rewire each edge's endpoint (the outgoing end) with probability $\phi$ to a uniformly selected node. Finally, the graph is made undirected by treating all edges as bidirectional.\footnote{Self-loops and multi-edges are allowed, but this does not affect the local limit as we will see later.}

The resulting graph, denoted $WS(n,\phi,k)$, has two kinds of edges: \emph{ring edges} (non-rewired) and \emph{shortcuts} (rewired edges). Also, we refer to the direction of edges as \textit{incoming} or \textit{outgoing} based on their direction in the initial clockwise orientation of the construction above.  Next, we describe the limit.

\begin{figure}
\begin{center}
    \begin{tikzpicture}[scale=0.6]
        \begin{scope}[]
            \foreach \i in {1, ..., 12} {
            \node [sdot] (\i) at (-30*\i:2) {};
            }
            \foreach \j [count = \i] in {2, ..., 12, 1} {
                \draw[->] (\i) -- (\j);
            }
            \foreach \j [count = \i] in {3, ..., 12, 1, 2} {
                \draw[->] [out = {-60  -30*\i}, in = {+60 -30*\j}] (\i) to (\j);
            }
        \end{scope}
        \begin{scope}[xshift = 6cm]
            \foreach \i in {1, ..., 12} {
            \node [sdot] (\i) at (-30*\i:2) {};
            }
            \draw[->] (1) -- (2);
            \draw[->] [out = {-60  -30*1}, in = {+60 -30*3}] (1) to (3);
            \draw[->] (2) -- (3);
            \draw[->] [out = {-60  -30*2}, in = {+60 -30*4}] (2) to (4);
            \draw[->, blue-shortcut] (3) -- (10);
            \draw[->, blue-shortcut] (3) -- (5);
            \draw[->, blue-shortcut] (4) -- (12);
            \draw[->, blue-shortcut] (4) -- (1);
            \draw[->] (5) -- (6);
            \draw[->] [out = {-60  -30*5}, in = {+60 -30*7}] (5) to (7);
            \draw[->, blue-shortcut] (6) -- (9);
            \draw[->] (6) -- (7);
            \draw[->, blue-shortcut] (7) -- (2);
            \draw[->, blue-shortcut] (7) -- (5);
            % \draw[->, blue-shortcut] (8) -- (8);
            \draw[->, blue-shortcut] (8) -- (6);
            \draw[->] (9) -- (10);
            \draw [out = {-60  -30*9}, in = {+60 -30*11}] (9) to (11);
            \draw[->] (10) -- (11);
            \draw[->, blue-shortcut] (10) -- (2);
            \draw[->, blue-shortcut] (11) -- (7);
            \draw[->, blue-shortcut] (11) -- (9);
            \draw[->, blue-shortcut] (12) -- (9);
            \draw[->, blue-shortcut] (12) -- (7);
        \end{scope}
        \begin{scope}[xshift = 12cm]
            \foreach \i in {1, ..., 12} {
            \node [sdot] (\i) at (30*\i:2) {};
            }
            \draw[->, blue-shortcut] (1) -- (2);
            \draw[->, blue-shortcut] (1) -- (5);
            \draw[->, blue-shortcut] (2) -- (8);
            \draw[->, blue-shortcut] (2) -- (6);
            \draw[->, blue-shortcut] (3) -- (10);
            \draw[->, blue-shortcut] (3) -- (5);
            \draw[->, blue-shortcut] (4) -- (12);
            \draw[->, blue-shortcut] (4) -- (1);
            \draw[->, blue-shortcut] (5) -- (3);
            \draw[->, blue-shortcut] (5) -- (7);
            \draw[->, blue-shortcut] (6) -- (9);
            \draw[->, blue-shortcut] (6) -- (7);
            \draw[->, blue-shortcut] (7) -- (2);
            \draw[->, blue-shortcut] (7) -- (5);
            \draw[->, blue-shortcut] (8) -- (9);
            \draw[->, blue-shortcut] (8) -- (6);
            \draw[->, blue-shortcut] (9) -- (11);
            \draw[->, blue-shortcut] (9) -- (5);
            \draw[->, blue-shortcut] (10) -- (11);
            \draw[->, blue-shortcut] (10) -- (2);
            \draw[->, blue-shortcut] (11) -- (7);
            \draw[->, blue-shortcut] (11) -- (9);
            \draw[->, blue-shortcut] (12) -- (9);
            \draw[->, blue-shortcut] (12) -- (7);
        \end{scope}
    \end{tikzpicture}
\end{center}
\caption{$WS(12, \phi, 2)$ increasing randomness from left to right: for $\phi = 0$, $0 < \phi < 1$, $\phi = 1$ respectively. Colored edges correspond to shortcuts while black edges correspond to ring edges. Arrows indicate the directions of edges in the Watts-Strogatz process but the final graphs are undirected.}
\label{ws-model}
\end{figure}
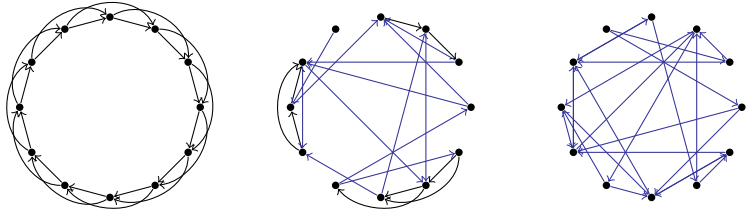

\subsection{The Structure of the Local Limits of Watts-Strogatz Model} \label{local-limit-ws}

We begin by analyzing the limit of the initial $k$-ring before introducing rewiring. 
The local limit of the $k$-ring is an infinite rooted path where each node is connected to its $k$ nearest neighbors on either side, resulting in each node having degree $2k$.
We call this the \emph{Full $k$-Path}. 
Fix an orientation of the path. For each node, we define the $k$ edges to the right as its outgoing edges, and the $k$ edges to the left as its incoming edges.
We also define a variant of the Full $k$-Path, denoted the \emph{Reduced $k$-Path}, which is identical except that the root has only $k-1$ outgoing edges, with one outgoing edge removed uniformly at random.
These serve as the building blocks for the $k$-Fuzz structures.

% We approach the problem by first examining the limit structure of the initial network (the $k$-ring) and then introducing the rewiring process in the limit. The local limit of the $k$-ring is an infinite rooted path where each node is connected to its \( k \)-nearest neighbors. This results in a graph where every node has exactly \( 2k \) neighbors, a configuration we call the  \textit{Full \( k \)-Path}. We define the \textit{outgoing edges} as the connections from a node to its \(k\) nearest neighbors on its right-hand side.  As an extension of the  Full \(k\)-Path, we define a reduced version of the Full \( k \)-Path (called \textit{Reduced $k$-Path}) with the key difference that the root node has only \( k-1 \) outgoing edges (one of the original $k$ outgoing edges is removed uniformly at random). We now define the \( k \)-Fuzz structures.

\subsubsection*{\( k \)-Fuzz structures} 
A \textbf{Full $k$-Fuzz} is obtained by starting from a Full $k$-Path and, for each node, independently keeping each outgoing ring edge with probability $1-\phi$ or rewiring it with probability $\phi$. 
Each rewired edge becomes a shortcut to the root of a new Full $k$-Fuzz, producing an independent recursive copy. 
In addition, each node receives $\text{Poi}(\phi k)$ incoming shortcuts, each connecting to the root of a \emph{Reduced $k$-Fuzz}. 
The \textbf{Reduced $k$-Fuzz} is defined analogously but starts from a Reduced $k$-Path (see Figure~\ref{full-k-fuzz}).

\begin{figure}
    \begin{center}
    \begin{tikzpicture}[scale=0.6]
        \foreach \i in {1, ..., 9} {
            \node (\i) [sdot] at (\i, 0) {};
        }
        \node (5) [sdot, red] at (5, 0) {};
        \node at (5,.6){\tiny Full $k$-Fuzz};
        \foreach \i\j in {1/2, 2/3, 3/4, 4/5, 5/6, 6/7, 7/8, 8/9} {
            \draw (\i) -- (\j);
        }
        \foreach \i\j in {1/3, 2/4, 3/5, 4/6, 6/8, 7/9} {
            \draw [out = 30, in = 150] (\i) to (\j);
        }

        \foreach \i in {1, ..., 7} {
            \node (a\i) [sdot] at (\i - 2.5, -2) {};
        }
        \foreach \i\j in {1/2, 2/3, 3/4, 5/6, 6/7} {
            \draw (a\i) -- (a\j);
        }
        \foreach \i\j in {1/3, 2/4, 3/5, 4/6, 5/7} {
            \draw [out = 30, in = 150] (a\i) to (a\j);
        }

        \foreach \i in {1, ..., 7} {
            \node (b\i) [sdot] at (\i + 4.5, -2) {};
        }
        \foreach \i\j in {1/2, 2/3, 3/4, 4/5, 5/6, 6/7} {
            \draw (b\i) -- (b\j);
        }
        \foreach \i\j in {1/3, 2/4, 3/5, 4/6, 5/7} {
            \draw [out = 30, in = 150] (b\i) to (b\j);
        }
        \draw [->, blue-shortcut, thick, >=latex] (a4) -- (5);
        \draw [->, blue-shortcut, thick, >=latex] (5) -- (b4);
        \node at (b4) [below] {\tiny Full $k$-Fuzz};
        \node at (a4) [below] {\tiny Reduced $k$-Fuzz};
    \end{tikzpicture}
\end{center}
    \caption{Full $k$-Fuzz (for $k=2$). The outgoing shortcut from the root (shown in red) connects to another full $k$-Fuzz, while the incoming shortcut connects to a reduced $k$-Fuzz.}
    \label{full-k-fuzz}
\end{figure}
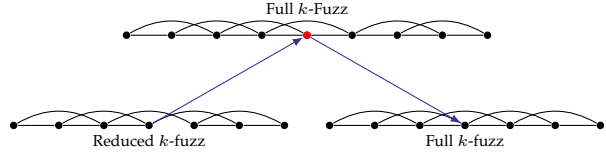

Let \( WS(\phi,k) \) denote the measure in the space of rooted graphs that describes the Full \( k \)-Fuzz process above. Our first main result formally establishes $WS(\phi,k)$ as the local limit of $WS(n,\phi,k)$.

\begin{theorem}[Local Limit of the Watts-Strogatz model] \label{main-thm-ws}
For any integer $k\geq 1$, Watts-Strogatz model converges locally in probability to the Full $k$-Fuzz described above, i.e., $WS(n,\phi, k)\overset{\mathbb P}{\to}WS(\phi, k)$ as $n\to\infty$. 
\end{theorem}

To build intuition for the local limit, consider the neighborhood of a uniformly chosen root.   Traversing a shortcut leads, with high probability, to a node that is sufficiently far from the root so that its local neighborhood is `asymptotically' independent of that of the original node.  This results in an independent $k$-Fuzz process rooted at the target node. The degree of the target node within the $k$-Fuzz depends on the type of shortcut:
(1) If we traverse an incoming shortcut to reach the target node, then one of the  \( k \) outgoing ring edges of the target node has been rewired. This leaves \( k-1 \) remaining edges at the target node that could potentially be rewired, hence the rest of the process is described by a Reduced \( k \)-Fuzz. 
(2) If we traverse an outgoing shortcut, we encounter a Full \( k \)-Fuzz, as all \( k \) outgoing edges are available for potential rewiring.

A key structural fact is that shortcuts do not form cycles in any fixed neighborhood (Lemma~\ref{no cycle}); this independence drives the recursive $k$-Fuzz limit.  For this purpose, we carefully track whether each edge is a shortcut or a ring edge, and how far the endpoints are from each other in the original $k$-ring neighbors  (in total, this requires defining $4k$ edge marks). We prove convergence with respect to these edge marks, using a second-moment argument in Section~\ref{sec:proof-WS}. In the special case of $k=1$, the limit simplifies to a multi-type branching process:

\begin{corollary}\label{cor-ws} 
    When $k = 1$, the Watts-Strogatz model converges locally in probability to a multi-type branching process with $4$ types.
\end{corollary}

\subsection{Kleinberg Model} \label{kleinberg-process} 
In the Kleinberg model \cite{kleinberg}, nodes are arranged on an \(n \times n\) lattice, so there are $n^2$ nodes.
Let node \(u\) have coordinates \((u_x, u_y)\), with the top-left node at \((0,0)\) and first and second coordinates increasing by \(1\) when moving right and down respectively.\footnote{In Section~\ref{proof-main-thm-kl-1}, we will use these coordinates normalized by \(n\) to define a node’s mark, which will be crucial in our analysis.} Then, the lattice distance between nodes \(u\) and \(v\) is
$d(u, v) = |u_x - v_x| + |u_y - v_y|$.
 Each node connects to all nodes within lattice distance \(k\), forming a directed \(k\)-lattice.
In addition, each node $v$ creates $q$ i.i.d. outgoing \emph{shortcuts}, each landing at $u$ with independent probability proportional to $d(u,v)^{-\ell}$, where $\ell$ is the power-law exponent.\footnote{Multi-edges between distinct nodes are allowed, but self-loops are not (the normalization $\sum_{u\neq v} d(u,v)^{-\ell}$ excludes $u=v$). That said, multi-edges do not affect the local limit as we will see next.} The resulting random graph is denoted $K(n,q,k,\ell)$ (see Figure~\ref{kleinberg-model}). Finally, the graph is made undirected by treating all edges as bidirectional.

A key insight from \cite{kleinberg} highlights how the distribution of long-range shortcuts affects the efficiency of decentralized search. Specifically, for \(\ell = 2\), short paths can be found using a simple greedy algorithm, whereas for other values of \(\ell\), no local algorithm can efficiently discover the shortest path (see Section~\ref{sec:kl-app}).  We likewise identify $\ell=2$ as the critical point in the local limit: $K(n,q,k,\ell)$ exhibits a \emph{phase transition} at $\ell=2$, with different limits for $\ell < 2$, $\ell > 2$, and $\ell = 2$.

\begin{figure}
\centering
\begin{tikzpicture}[scale=0.75]
\usetikzlibrary{calc}
\definecolor{blue-shortcut}{rgb}{0.2, 0.2, 0.6}

% --- styles ---
\tikzset{
  sdot/.style={circle,fill=black,inner sep=1pt},
  incoming/.style={thick,teal},
  blue-shortcut/.style={thick,draw=blue-shortcut} % ensure custom color is used
}
% ---------------

\def\step{0.5} % lattice spacing

% grid macro with variable size: \gridN{Prefix}{N}{xshift}{yshift}
\newcommand{\gridN}[4]{%
  \foreach \i in {1,...,#2}{
    \foreach \j in {1,...,#2}{
      \node[sdot] (#1-\i-\j) at (#3+\step*\i, #4-\step*\j) {};
    }
  }
  \foreach \i in {1,...,#2}{
    \foreach \j [count=\k] in {2,...,#2}{
      \draw (#1-\i-\k) -- (#1-\i-\j); % vertical
      \draw (#1-\k-\i) -- (#1-\j-\i); % horizontal
    }
  }
}

% ---------- draw the three grids ----------
% TOP: 7x7
\def\nT{7}
\gridN{T}{\nT}{-1}{0}
\pgfmathtruncatemacro{\cT}{(\nT+1)/2} % = 4 for 7x7

% BOTTOM-LEFT & BOTTOM-RIGHT: 5x5 (unchanged)
\def\nB{5}
\gridN{L}{\nB}{-5}{-3.5}
\gridN{R}{\nB}{ 4}{-3.5}
\pgfmathtruncatemacro{\cB}{(\nB+1)/2} % = 3 for 5x5

% ---------- root ----------
\node[sdot,red] (root) at (T-\cT-\cT) {};

% ---------- labels ----------
% top label ABOVE the top grid (above the middle of the top row)
\node at ($(T-\cT-1)+(0,0.35)$) {\tiny $(q,k)$-patch};

% bottom labels BELOW the grids
\coordinate (Lcenter) at (L-\cB-\cB);
\coordinate (Rcenter) at (R-\cB-\cB);
\node at ($(Lcenter)+(0,-1.35)$) {\tiny Reduced $(q,k)$-patch};
\node at ($(Rcenter)+(0,-1.35)$) {\tiny $(q,k)$-patch};

% ---------- shortcuts ----------
\draw[->, incoming,      >=latex] (Lcenter) -- (root);      % teal incoming
\draw[->, blue-shortcut, >=latex] (root)    -- (Rcenter);   % blue outgoing

% ---------- local outgoing (blue-shortcut) edges around the TOP root ----------
% Curved cardinals to the border (top/bottom/left/right rows/cols)
\draw[<->, violet, out=110, in=250] (T-\cT-\cT) to (T-\cT-2);     % up (row 1)
\draw[<->, violet, out= 20, in=160] (T-\cT-\cT) to (T-6-\cT);   % right (col nT)
\draw[<->, violet, out=-70, in= 70] (T-\cT-\cT) to (T-\cT-6);   % down (row nT)
\draw[<->, violet, out=-160,in= -20] (T-\cT-\cT) to (T-2-\cT);    % left (col 1)

% Straight orthogonals one step from center
\draw[<->, violet] (T-\cT-\cT) -- (T-\cT-\the\numexpr\cT-1\relax);
\draw[<->, violet] (T-\cT-\cT) -- (T-\cT-\the\numexpr\cT+1\relax);
\draw[<->, violet] (T-\cT-\cT) -- (T-\the\numexpr\cT-1\relax-\cT);
\draw[<->, violet] (T-\cT-\cT) -- (T-\the\numexpr\cT+1\relax-\cT);

% Diagonals one step from center
\draw[<->, violet] (T-\cT-\cT) -- (T-\the\numexpr\cT-1\relax-\the\numexpr\cT-1\relax);
\draw[<->, violet] (T-\cT-\cT) -- (T-\the\numexpr\cT+1\relax-\the\numexpr\cT+1\relax);
\draw[<->, violet] (T-\cT-\cT) -- (T-\the\numexpr\cT+1\relax-\the\numexpr\cT-1\relax);
\draw[<->, violet] (T-\cT-\cT) -- (T-\the\numexpr\cT-1\relax-\the\numexpr\cT+1\relax);

\end{tikzpicture}
 \caption{The local limit structure of $K(n, q, k, \ell)$ for $\ell \leq 2$: $(q, k)$-patch (for $k=2$). The outgoing shortcut from the root (shown in red) connects to another $(q, k)$-patch, while the incoming shortcut connects to a reduced $(q, k)$-patch. (For simplicity, we plot the directed edges that only belong to the root and colored edges based on the type.)}
\label{q-k patch}
\end{figure}

\subsection{Subcritical regime: Local Limit of the Kleinberg Model ($\ell < 2$)} \label{local-limit-kl-1}

When $\ell < 2$, shortcuts almost never remain in a node’s local neighborhood. To capture this behavior, we view $K(n,q,k,\ell)$ as a marked graph: each node is marked with its normalized lattice coordinates in $[0,1]^2$ (with the $L_1$ metric). 
The limit object is a recursive structure we call a \emph{$(q,k)$-patch} (see Figure~\ref{q-k patch}), defined as follows:
\begin{enumerate}
    \item \textit{The Base Lattice:} Start from a rooted infinite $k$-lattice.  All nodes in this lattice share the same mark \(m\) sampled uniformly from \([0,1] \times [0,1]\).    
    \item \textit{Outgoing Shortcuts with Marks:} Each node has $q$ outgoing shortcuts. Each shortcut from a node with mark \(m\) connects to the root of a new independent $(q,k)$-patch with a new mark $m_{\mathrm{out}}$ sampled from a distribution on \([0,1] \times [0,1]\) with density
    \[
      p_{\mathrm{out},m}(m_{\mathrm{out}}) = \frac{\|m-m_{\mathrm{out}}\|_1^{-\ell}}{\int_{x \in [0,1]^2} \|m-x\|_1^{-\ell}\, dx}.
    \]
    \item \textit{Incoming Shortcuts with Marks:}  Each node with node mark \(m\) receives $\text{Poi}(\Lambda_m)$ incoming shortcuts,\footnote{When the patch was reached from its parent via an incoming shortcut, that parent edge is part of the parent's $q$ outgoing shortcuts and is not resampled here; the $\mathrm{Poi}(\Lambda_m)$ counts only the additional incoming shortcuts.} with
    \[
      \Lambda_m = q \int_{w \in [0,1]^2} p_{\mathrm{out},w}(m)\, dw.
    \]
    Each such edge comes from the root of a \emph{reduced} $(q,k)$-patch, identical to a $(q,k)$-patch except that the root has only $q-1$ outgoing shortcuts with the missing shortcut chosen uniformly at random from \(q\) possible shortcuts. 
    The incoming mark $m_{\text{in}}$ is sampled from a distribution on \([0,1] \times [0,1]\) with density
    \[
      p_{\text{in},m}(m_{\text{in}}) = \tfrac{q}{\Lambda_m}\,p_{\mathrm{out},m_{\text{in}}}(m).
    \]
    \item \textit{Recursion:} The process repeats independently for all new patches.
\end{enumerate}
We refer to the graph sampled from this $(q,k)$-patch process as $K_{<}(q,k,\ell)$, and the next theorem establishes $K_{<}(q,k,\ell)$ as the local limit of $K(n,q,k,\ell)$ for $\ell < 2$.

\begin{theorem}[Local limit of the Kleinberg model ($\ell < 2$)] \label{main-thm-kl-1} For any integers $k$ and $q>0$, $\ell < 2$, the Kleinberg model $K(n,q,k,\ell)$ converges locally in  probability to  $K_{<}(q,k,\ell)$.
\end{theorem}

The proof is based on showing that shortcuts originating from a node connect, with high probability, to distant and independent regions of the lattice.  Equivalently, we show that the graph induced by shortcuts does not form a cycle (see Lemma~\ref{lem:no-local-shortcut-cycle}).  In the proof, we also rely on a coupling argument that aligns the continuous marks of the shortcuts in the limiting graph with the discrete marks in the finite case, while carefully tracking the first and second moments of subgraph frequencies. See details in Sections~\ref{proof-main-thm-kl-1} and~\ref{sec:proof-main-thm-kl-1}.

\subsection{Supercritical regime: Local Limit of the Kleinberg Model ($\ell > 2$)} \label{local-limit-kl-2}

For \(\ell > 2\), the local limit of the Kleinberg model exhibits a fundamentally different behavior. In this regime, shortcuts are local, connecting nodes within a bounded lattice distance. To formalize this, introduce the normalization constant
\[
\zeta(\ell-1) := \sum_{d=1}^{\infty} d^{\,1-\ell},
\]
which is finite for \(\ell > 2\). The probability of a shortcut connecting a node \(u\) to another node \(v\) at lattice distance \(d(u,v)\) is given by
\[
\kappa(u,v) := \frac{d(u,v)^{-\ell}}{4\,\zeta(\ell-1)}.
\]
Here, $\kappa:\mathbb{Z}^2\times\mathbb{Z}^2\to[0,1]$ is the kernel defining the probability distribution of connections.

The limiting graph \(K_>(q, k, \ell)\) is constructed as follows: Starting with an infinite \(k\)-lattice, each node is assigned \(q\) outgoing shortcuts, where the endpoint of each shortcut is chosen independently according to the kernel \(\kappa\). The following theorem establishes the convergence of the finite Kleinberg model to this limiting structure.

\begin{theorem} [Local limit of the Kleinberg model ($\ell > 2$)] \label{main-thm-kl-2}
For any integers \(k > 0\), \(q > 0\), and \(\ell > 2\), the Kleinberg model \(K(n, q, k, \ell)\) converges locally in probability to the graph \(K_>(q, k, \ell)\). 
\end{theorem}

The first step in the proof is to show that shortcuts in both the finite and limiting graphs remain confined to a bounded lattice distance. 
This ensures that the limit is well-defined and local. 
To establish this, we relate graph neighborhoods to lattice neighborhoods:  for any graph distance \(r\), we prove the existence of a lattice distance \(R\) such that the graph neighborhood of radius \(r\) is contained within the lattice neighborhood of radius \(R\). The key tool is a path-counting argument showing that any path of length $r$, even with shortcuts, cannot escape beyond a bounded lattice distance (Lemma~\ref{path counting}). 
Once we can bound local neighborhoods by lattice neighborhoods, convergence of lattice neighborhoods follows more directly, since no shortcut connections to infinity need to be tracked. 
A second-moment argument shows that neighborhoods of two uniformly chosen nodes become asymptotically independent (Lemmas~\ref{lattice-cv-1} and \ref{lattice-cv-2}).  This is formalized in Section~\ref{proof-main-thm-kl-2}.

The path-counting argument is similar to the techniques developed for spatial inhomogeneous random graphs (SIRGs) in the beautiful paper of \cite{van2021local}. However, some modifications are required to handle the unique features of the Kleinberg model. In the SIRG framework, nodes are embedded in a spatial structure, and edges between any two nodes arise independently with probabilities governed by a kernel function that depends on their spatial distance. By contrast, in the Kleinberg model, the number of shortcuts is a constant  $q$ per node, which changes the way randomness operates in long-range connections.

\subsection{Critical threshold: Local Limit of the Kleinberg Model ($\ell = 2$)} \label{local-limit-kl-critical}

The critical case $\ell=2$ has features of both neighboring regimes.
As in the case $\ell<2$, shortcuts leave every fixed neighborhood with
high probability, so the shortcut structure in the local limit is still
tree-like. But, as in the case $\ell>2$ marks collapse and play no further role in the limit. 
% The reason is that the kernel $\|m-m'\|_1^{-2}$ is not integrable on $[0,1]^2$, so the marked patch construction used for $\ell<2$ does not extend to the boundary case.
As a result, we lose a meaningful distribution for where shortcuts land, and the limit is a graph in which every node has incoming shortcut degree distributed as $\text{Poi}(q)$. This substantially simplifies the local limit at criticality compared to the $\ell<2$ case and leads to simpler limiting properties, such as a simpler expression for the clustering coefficient (see Corollary~\ref{cor2}). 

To formalize, we define the critical limit $K_{=}(q,k)$ as follows:
\begin{enumerate}
    \item \textit{Base lattice:} Start from a rooted infinite
    $k$-lattice.
    \item \textit{Outgoing shortcuts:} Each node has exactly $q$
    outgoing shortcuts, each leading to the root of an independent copy
    of $K_{=}(q,k)$.
    \item \textit{Incoming shortcuts:} Each node receives
    $\mathrm{Poi}(q)$ incoming shortcuts. Each such shortcut
    comes from the root of an independent \textit{reduced} $K_{=}(q,k)$, defined
    in the same way except that the root has only $q-1$ outgoing
    shortcuts, with the missing shortcut chosen uniformly from the $q$
    possible outgoing shortcuts.

    \item \textit{Recursion:} The process repeats independently for all newly created (reduced) $K_{=}(q,k)$'s.

\end{enumerate}
One can also interpret $K_{=}(q,k)$ as an \textit{unmarked} and \textit{symmetric} $(q, k)$-\textit{patch}. Our next result establishes that this is indeed the local limit when $\ell=2$.

\begin{theorem}[Local limit of the Kleinberg model at criticality ($\ell = 2$)] \label{main-thm-kl-critical} For any integers $k$ and $q>0$, $\ell = 2$, the Kleinberg model $K(n,q,k,\ell)$ converges locally in  probability to  $K_{=}(q,k)$.
\end{theorem}

The main ingredient for the proof of Theorem~\ref{main-thm-kl-critical} is showing that asymptotically the local neighborhood of any node looks the same almost everywhere in the grid (see Lemmas~\ref{lem:uniform-normalization}, \ref{lem:boundary-avoidance}, and \ref{lem:critical-poisson-input}), making the marks have no intrinsic meaning and distribution for where shortcuts land irrelevant. Having established this, the rest of the argument follows from similar steps to the proofs of convergence for $\ell < 2$ and $\ell > 2$ cases. See details in Section~\ref{proof-main-thm-kl-critical}. 
A direct corollary of the proof of Theorem~\ref{main-thm-kl-1} and~\ref{main-thm-kl-critical} is the following.
\begin{corollary}[Hidden random tree] \label{hidden-tree}
For $\ell \leq 2$, the subgraphs induced by the shortcuts in
$K(n,q,k,\ell)$ converge locally in probability to a multitype branching
process.\footnote{At $\ell=2$, the four types collapse to a single one and, hence, the limit is a single-type branching process.}
\end{corollary}

Taken together, Theorems ~\ref{main-thm-kl-1}-~\ref{main-thm-kl-critical} provide a unified characterization of the local geometry of small-world networks.  Across all regimes, the local neighborhood of a typical node converges to a well-defined limiting object. While the structure differs, the key common feature is that dependencies vanish at finite scales. This locality enables analysis: any graph functional or dynamical process depending only on bounded neighborhoods can be studied directly on the limit. We now leverage this perspective to analyze several such quantities.

\section{Applications}\label{sec-algorithmic-app}
By proving local convergence for small-world networks, we gain powerful tools to analyze a wide range of network properties. Below, we outline several key applications of Theorems~\ref{main-thm-ws}-\ref{main-thm-kl-critical}. 

% \subsection*{Local Functionals and Sampling-Based Estimation.}
% We first state a general consequence of local weak convergence and then consider different applications.
% For this purpose, call a function $f$ \emph{$r$-local} if $f(G,o)$ depends only on $B_r(G,o)$ (including marks).

% \begin{corollary} [Local Limits and Constant-time Estimation]\label{prop:meta}
% Let $\{G_n\}_{n\geq 0}$ converge locally in probability to $\mathcal G_* \sim \mu$ and let $f$ be a $r$-local and bounded function.
% In particular, for $\theta(\mu):=\mathbb{E}_{G\sim \mu}[f(B_r(G,o))]$,
% the estimator $\hat\theta_m:=\frac1m\sum_{i=1}^m f(B_r(G,o_i))$ from $m$ i.i.d.\ uniform roots satisfies
% $|\hat\theta_m-\theta(\mu)|\le \varepsilon$ with probability $1-\delta$ for $m=O(\varepsilon^{-2}\log(1/\delta))$.
% \end{corollary}
% This is a direct consequence of local convergence (and a restatement of \cite[Theorem~2.17]{remco}).
% This proposition is our main ``engine" for applications: any local parameter or algorithm whose outcome can be approximated by constant-radius neighborhoods
% automatically inherits convergence and admits constant-time estimators. Conversely, whenever a graph functional converges along a locally convergent sequence, its value can be approximated by constant-radius samples from the underlying graph.
% In the remainder of this section, we instantiate this template for several graph properties.

\subsection{Dynamic Processes on Small-World Networks}
Local convergence is a powerful tool for understanding dynamic processes on networks, such as the spread of information or diseases. One key application is understanding the spread of epidemics on networks.

In the context of the SIR (Susceptible-Infected-Recovered) model, individuals are represented as nodes of a graph, where edges indicate possible transmission \cite{bernoulli1766essai, kermack1927contribution, ross1917pathometry3}. Each node is initially infected independently with probability $\rho>0$ and is otherwise susceptible.
An infected node remains infectious for a random amount of time and then recovers with rate $\mu$. During its infectious period, it attempts to infect each neighbor at rate $\beta$. If such an attempt occurs when the neighbor is still susceptible, the neighbor becomes infected and the process continues from there.

% A special case with a constant time infectious period is the \textbf{Independent Cascade (IC)} model, in which each infected node infects each of its neighbors $(u\!\to\! v)$ independently with probability $p$; see \cite{kempe2003kdd,chen2009kdd,borgs2014soda}.

It is shown in the literature that when a sequence of  graphs converges locally in probability  then 
the time evolution of the epidemic (the fractions of nodes in each state) converges as the network size increases \cite{alimohammadi2023epidemic, milewska2025sirlocallyconvergingdynamic}. Moreover, it enables the use of local network samples of a constant number of nodes to predict the spread of epidemics \cite{alimohammadi2022algorithms}. Using our result, these insights can be applied to small-world networks as well.  

\begin{corollary}[Epidemic Spread]\label{cor:sir-finite}
Fix $\rho,\beta,\mu$. If $\{G_n\}_{n\geq0}$  converges locally in probability to $G\sim \mu$, then for any $t\in[0,\infty]$
\[
\mathbb{E} \left[X_t^{\mathrm{SIR}}\!\big(G_n,o_n\big) \right] \longrightarrow\ \mathbb{E}  \left[X_t^{\mathrm{SIR}}(G,o)
\right] \]
where $X_t^{\mathrm{SIR}}(G,o)$ is the indicator that $o$ is infected by time $t$.
Moreover, for every $\varepsilon,\delta\in(0,1)$ there exist constants $R=R(\beta, \varepsilon)$ and $m=O(\varepsilon^{-2}\log(1/\delta))$ such that
\[
\hat\theta_t \;=\; \frac1m\sum_{i=1}^m X_t^{\mathrm{SIR}}\!\big(B_R(G_n,o_i)\big)
\]
(using $m$ i.i.d.\ uniform roots $o_i$ and only radius–$R$ neighborhoods) is an $(\varepsilon,\delta)$–accurate estimator of the fraction infected by time $t$.
\end{corollary}

The key message is that to understand epidemic behavior on large finite networks it suffices to analyze the corresponding local limit, which is a simpler object to work with.  In $WS(\phi,k)$, increasing $\phi$ creates tree–like shortcut structure and accelerates early–time spread (smaller return times, larger finite–horizon prevalence), while the final size is a function of $(\rho,\beta,\mu,k,\phi)$ computed on the limit object. For $K(n,q,k,\ell)$, when $\ell\le 2$ the $(q,k)$–patch induces branching–like expansion (larger early growth and, typically, larger final size), whereas for $\ell>2$ the $\kappa$–kernel lattice behaves closer to a bounded–range lattice and spread resembles lattice–like propagation.

\subsection{Edge Percolation and the Size of the Giant}
\label{sec:app-giant-percolation-kleinberg} 
Another natural question is the robustness of small-world structure in preserving \emph{global} connectivity under random edge failures. A widely studied way to probe this is bond percolation: retain each edge of graph $G$ independently with probability $p\in(0,1)$ and study the resulting connected components in resulting graph $G^{(p)}$.  Writing $\mathcal C^{(p)}(v)$ for the connected component of $v$ in $G^{(p)}$, the main object of interest is the largest component $\mathcal C_{\max}(G^{(p)})$ (and, for comparison, the second largest component $\mathcal C_{(2)}(G^{(p)})$).

When a graph sequence converges locally, the percolated local limit describes the law of $\mathcal C^{(p)}(o_n)$ for a uniform root $o_n$. However, the density of the largest component is a global graph property, and does not usually follow from local convergence alone. One therefore typically requires additional structural assumptions; for example, the giant component is known to be local for random regular graphs, Erd\H{o}s--R\'enyi graphs, and, more generally, for expander sequences \cite{yeganeh-digraph, vanderhofstad2023giant, van_der_Hofstad_2025}. Recently, van der Hofstad~\cite{vanderhofstad2023giant} identified a necessary and sufficient additional condition under which the giant becomes local.  Informally, the condition states that large percolation clusters must typically \emph{coalesce}: two nodes that both lie in percolation clusters of size at least $k$ should be connected to each other with high probability once $k$ is large. In the next theorem we verify this condition for the Kleinberg model, and hence conclude that the percolated giant is determined by the local limit.

\begin{theorem}[The Giant Component under Bond Percolation in the Kleinberg Model]
\label{thm:kleinberg-perc-giant}
Fix $q > 0$, and $k\ge 1$. Let $G_n = K(n,q,k,\ell)$ be the Kleinberg graph, and let $G_n^{(p)}$ denote its bond percolation. Assume either $\ell < 2$ and $p\in[0,1]$, or $\ell>2$ and $p\in(1/2,1]$. Let $(G,o)$ be the local limit of $(G_n,o_n)$ as defined in Theorems~\ref{main-thm-kl-1}-\ref{main-thm-kl-critical}, and write $G^{(p)}$ for its bond percolation. Define $\theta(p) := \mathbb P\!\left[|\mathcal C_{G^{(p)}}(o)|=\infty\right]$. Then,
\[
\frac{|\mathcal C_{\max}(G_n^{(p)})|}{n^2} \xrightarrow{\mathbb P} \theta(p),
\qquad
\frac{|\mathcal C_{(2)}(G_n^{(p)})|}{n^2} \xrightarrow{\mathbb P} 0 .
\]
\end{theorem}

Theorem~\ref{thm:kleinberg-perc-giant} shows that, despite being a global connectivity property, the size of the giant component after percolation is determined by the local structure of the network.  When $\ell < 2$, the Kleinberg model exhibits expansion properties with high probability \cite{flaxman}; in this case, the conclusion follows directly by combining expansion with existing results showing that the giant component under bond percolation is local for expander sequences~\cite{yeganeh-digraph}. For $\ell>2$, the proof proceeds by verifying the condition of van der Hofstad~\cite{vanderhofstad2023giant}. We leverage the fact that, on $\mathbb Z^2$, the size of the giant component is local \cite{borgs2000lattice}.  Consequently, for the supercritical regime $p>1/2$, the condition of~\cite{vanderhofstad2023giant} holds for the Kleinberg graph as well. See the proof details in Section~\ref{proof-applications-epidemic}. 

The result also admits a natural interpretation in terms of information cascades: in the Independent Cascade model, where each edge transmits influence with probability $p$ \cite{borgs2014soda, chen2009kdd, kempe2003kdd}, the set of nodes eventually influenced from a single initial seed coincides with the connected component of that seed in $G^{(p)}$. As a result, in the regimes covered by Theorem~\ref{thm:kleinberg-perc-giant}, the probability and typical size of a large cascade are determined by the percolated local limit, and can be inferred from local neighborhoods of the finite graph.

% A special case with a constant time infectious period is the \textbf{Independent Cascade (IC)} model, in which each infected node infects each of its neighbors $(u\!\to\! v)$ independently with probability $p$; see \cite{kempe2003kdd,chen2009kdd,borgs2014soda}.

\subsection{Clustering Coefficient}
Next, we turn our attention to estimating common graph properties that have received attention in the literature.
High clustering coefficient is one of the defining signatures of the small-world phenomenon: the original Watts-Strogatz model was introduced precisely to capture the coexistence of short path lengths and substantial local triangle structure \cite{watts}, and clustering coefficients remain the standard way to quantify this local transitivity in network models. It is therefore natural to begin our applications to estimating graph structures with clustering coefficient.

The local clustering coefficient $C(v)$ of a node $v$ measures the fraction of ordered pairs of neighbors of $v$ that are adjacent, while the global clustering coefficient $C_{\mathrm{global}}$ aggregates the same quantity over the whole graph:
\[
C(v)=\frac{\Delta_G(v)}{d_v(d_v-1)}, \qquad
C_{\mathrm{global}}=\frac{\sum_{v\in V}\Delta_G(v)}{\sum_{v\in V} d_v(d_v-1)}.
\]
Here, $d_v$ denotes the degree of node $v$, and
\[
\Delta_G(v)=\sum_{u,w\in \partial B_1(v)} 1_{\{\{u,w\}\in E(G)\}}
\]
denotes twice the number of triangles in $G$ that contain $v$. We also define the average local clustering coefficient by
\[
C_{\mathrm{local}}:=\frac{1}{|V|}\sum_{v\in V} C(v).
\]

These quantities are especially well-suited to the local convergence framework. Indeed, $C(v)$ is determined by the radius-$1$ neighborhood of $v$, while $C_{\mathrm{global}}$ is obtained from the corresponding empirical averages. Hence, \cite[Theorems 2.22 and 2.23]{remco} identify the asymptotic behavior of $C_{\mathrm{global}}$ and $C_{\mathrm{local}}$ along locally convergent graph sequences. In the present models, the computation simplifies further because shortcuts are asymptotically locally tree-like: in the Watts--Strogatz model, and in the Kleinberg model for $\ell \leq 2$, shortcut edges do not belong to triangles in the limit (see  Lemmas~\ref{no cycle} and ~\ref{lem:no-local-shortcut-cycle}). Thus the limiting clustering is determined by the base ring or lattice structure together with the limiting degree distribution.

\begin{corollary} [Clustering Coefficient of Watts-Strogatz Model] \label{cor1}
Let \(v\) be a uniformly selected node of \(WS(n, \phi, k)\). Then, as \(n \to \infty\), 
\[
C_{\mathrm{local}} \xrightarrow{\mathbb{P}} \mathbb{E}_{\mu}\left[\frac{\Delta_G(o)}{d_o(d_o - 1)}\right],
\qquad C_{\mathrm{global}} \xrightarrow{\mathbb{P}} \frac{3(k-1)}{2(2k-1) + \phi (2 - \phi)} (1-\phi)^3,\]
where \(\Delta_G(o)\) is twice the number of triangles that contain the root \(o\) and \(\mu\) is the distribution specified in Theorem~\ref{main-thm-ws}.
\end{corollary}

For the Kleinberg model in the regimes $\ell \leq 2$, let $\Delta_{L_k}$ denote twice the number of triangles that include the root in the infinite $k$-lattice. Since the infinite $k$-lattice is deterministic, $\Delta_{L_k}$ is also deterministic. The next corollary identifies the limiting clustering coefficients in these regimes.

% For the analog in the Kleinberg model, let \(\Delta_{L_k}\) denote twice the number of triangles that include the root in the infinite \(k\)-lattice. Since the infinite \(k\)-lattice is deterministic, \(\Delta_{L_k}\) is also deterministic. Then the following result formalizes the behavior of the clustering coefficient in the limit.
\begin{corollary}[Clustering Coefficient of the Kleinberg Model] \label{cor2}
Let \(v\) be a uniformly selected node of \(K(n,q,k,\ell)\), and set \(A:=2k(k+1)+q\).

If \(\ell<2\), then as \(n\to\infty\),
\[C_{\mathrm{local}} \xrightarrow{\mathbb P}
\mathbb{E}\!\left[\frac{\Delta_{L_k}}{(A+X_m)(A+X_m-1)}\right], \qquad C_{\mathrm{global}}\xrightarrow{\mathbb P}\ \frac{\Delta_{L_k}}{A^2 - A + 2A\,\mathbb E[\Lambda_m] + \mathbb E[\Lambda_m^2]},
\]
where, conditional on the mark \(m\) of \(v\), the random variable \(X_m\) is Poisson with mean \(\Lambda_m\), and the expectations are over the uniformly sampled mark \(m\in [0,1]\times[0,1]\).

For \(\ell=2\), since \(\Lambda_m=q\) for all \(m\), if \(X\sim \mathrm{Poi}(q)\), then
\[
C_{\mathrm{local}} \xrightarrow{\mathbb P}
\mathbb{E}\!\left[\frac{\Delta_{L_k}}{(A+X)(A+X-1)}\right], \qquad C_{\mathrm{global}}\xrightarrow{\mathbb P}\ \frac{\Delta_{L_k}}{(A+q)^2-A}.
\]
\end{corollary}

These results provide the first rigorous characterization of clustering in small-world networks, previously understood only through \emph{mean-field} analyzes~\cite{barrat2000properties}. Proofs of both corollaries appear in Appendix~\ref{proof-applications-cluster}. In particular, \cite{barrat2000properties} approximate the global clustering coefficient of the WS model by
$\frac{3(k-1)}{2(2k-1)}(1-\phi)^3,$
interpreting it as the ratio between the mean number of edges among the neighbors of a node and the expected number of possible edges among those neighbors. They justify this approximation numerically, showing agreement with simulations for large \(n\). In contrast, our result identifies the exact asymptotic limit, explaining rigorously why the mean-field prediction matches observed behavior.

\subsection{PageRank}
Another property of interest is PageRank, which is an algorithm used to rank the importance or centrality of nodes in a graph, originally designed to rank World Wide Web pages  \cite{page1999pagerank}. Its applications extend to community detection, citation analysis, and beyond \cite{page1999pagerank}. Garavaglia et al. \cite{garavaglia2018local} proved that if a sequence of graphs $G_n$ with $n$ nodes converges in probability, then the PageRank of a uniform random node converges in distribution. As an application of this and our main theorems, we get the following corollary.

\begin{corollary}
    The normalized PageRank of $WS(n, \phi, k)$ and $K(n, q, k, \ell)$ converges in distribution to a limit. 
\end{corollary} 

Importantly, as a consequence of this result, the PageRank of a typical node
can be approximated from a sufficiently large finite neighborhood of the root
in the local limit.
In fact, we can use this result to interpret pagerank behavior asymptotically.
In $WS(\phi,k)$, increasing $\phi$ lowers short-horizon return probabilities and reduces PageRank at typical nodes.
In $K(n,q,k,\ell)$, the (marked) $(q,k)$-patch for $\ell\le 2$ yields smaller return probabilities (lower PageRank), while $\kappa$-kernel with short-range in expectation at $\ell>2$ yields PageRank values closer to the base lattice.

\subsection{Spanning Trees and Tree Entropy}
The number of spanning trees $\tau(G)$ is a classical property of graphs, linking the structure to random walks and electrical flows \cite{lyons2017probability}. Lyons \cite{Lyons2005} showed that under local convergence, the normalized logarithm $\tfrac{1}{n}\log\tau(G)$ converges to the \emph{tree entropy}, a quantity expressed in terms of random-walk return probabilities on the limit object, i.e.,  \cite[Theorem.~3.3]{Lyons2005} shows
\[
\frac{1}{|V(G_n)|}\log \tau(G_n)\ \longrightarrow\ h(G)
=\mathbb{E}[\log \deg_{G}(o)]-\sum_{t\ge 1}\frac{1}{t}\,\mathbb{E}[p_t(o,o)],
\]
where $p_t(o,o)$ is the $t$-step return probability of simple random walk on $G$. Each \( p_t(o,o) \) is \( t \)-local, so the series \( \sum_{t \ge 1} \frac{1}{t}\,\mathbb{E}[p_t(o,o)] \) consists of local functions; in particular, it can be approximated to arbitrary precision by truncating the sum at a finite horizon. We can apply our main theorems to this, and get the following.

\begin{corollary}[Tree Entropy for Small-World Networks]\label{cor:tree}
Let $G_n\in\{WS(n,\phi,k),\,K(n,q,k,\ell)\}$ with local limit $G\in\{WS(\phi,k),\,K_{<}(q,k,\ell),\,K_{=}(q,k),\,K_{>}(q,k,\ell)\} $.
Then $\tfrac{1}{|V(G_n)|}\log \tau(G_n)\to h(G)$ as above.
\end{corollary}

We can also use this result to compare the behavior of different small-world models. In $WS(\phi,k)$ (the $k$-Fuzz limit), shortcuts make neighborhoods more tree-like as $\phi$ increases, lowering return probabilities and increasing tree entropy. 
In $K(n,q,k,\ell)$, for $\ell\le 2$, the $(q,k)$-patch induces expansion (higher tree entropy); for $\ell>2$ the $\kappa$-kernel lattice behaves closer to the base lattice, yielding larger return probabilities and lower tree entropy.

\subsection{Greedy Maximal Independent Set}

Maximal independent sets are another fundamental  substructure in graphs, and the random greedy maximal independent set (MIS) process (also known as random sequential adsorption (RSA)  \cite{flory}) is a natural way to construct one: scan the nodes in a uniformly random order and add each node if none of its neighbors were added earlier. It is a fundamental local algorithm, studied in combinatorics  \cite{Nicholas}, probability \cite{Rahman}, computer science \cite{Fischer} and chemistry \cite{flory}. 

This process is especially natural in the present framework. Although the membership of a given node in the final MIS is not determined by a fixed-radius neighborhood, its asymptotic density is governed by the local limit. In particular, \cite{krivelevich2024rsa} proved that, for locally convergent bounded-degree graph sequences, the indicator that the root belongs to the random greedy MIS converges in distribution. Applying this perspective to the local limits established in this paper yields the following corollary.

\begin{corollary}[Random Greedy MIS for Small-World Networks]
Let $G_n\!\in\!\{WS(n,\phi,k),K(n,q,k,\ell)\}$ with local limit $(G, o)$.  
Then,
\[
\frac{|\mathrm{MIS}(G_n,\pi_n)|}{|V(G_n)|}\xrightarrow{d}\mathbb{E}\!\left[\mathbbm{1}\{o \in\mathrm{MIS}(G,\pi)\}\right],
\]
where $\xrightarrow{d}$ denotes convergence in distribution.
\end{corollary}
% \noindent\textbf{Algorithmic interpretation.}
This result has direct implications for randomized sublinear-time approximation algorithms on arbitrarily large finite graphs. One can estimate the greedy MIS density by exploring neighborhoods up to a fixed radius $r$, applying the MIS local oracle defined in  \cite{yoshida}---treating nodes beyond distance $r$ for the oracle as included in the MIS---and averaging the outcome over sufficiently many uniformly chosen roots and random permutations.

\subsection{Phase Transition and Kleinberg's Local Algorithm}\label{sec:kl-app}

Kleinberg considered a simple greedy algorithm to find a path from the source to the target. This algorithm starts from the source node and, at each step, it moves to the neighbor that is closest to the target destination until the target is reached \cite{kleinberg}. In particular, Kleinberg showed that this local algorithm is most efficient when \( \ell = 2 \): while we reach the target in \( O\left((\log n)^2\right) \) when $\ell = 2$, the expected time is significantly larger ($n^\alpha$ for some $\alpha>0$) for other values of $\ell$.

Our results in Theorems~\ref{main-thm-kl-1},~\ref{main-thm-kl-2}, and \ref{main-thm-kl-critical} provide an alternative explanation for this phenomenon. 
As the power-law exponent \(\ell\) crosses the threshold of $2$, the network's local structure undergoes significant changes. When \(\ell > 2\), the local limit reveals that all node marks are equal within a neighborhood, indicating that shortcuts only connect nodes within very short distances. This structural change explains why Kleinberg’s decentralized search algorithm becomes inefficient in this regime: the shortcuts fail to support long-range navigation.

Conversely, when $\ell \leq 2$, the graph induced by shortcuts converges to a tree-like structure, ensuring strong global connectivity. However, the regime of $\ell < 2$ has another challenge: the distance between the two ends of a shortcut is bounded below by $\Omega(n)$ (Lemma~\ref{shortcut-length-distribution}). Thus, once the decentralized search algorithm gets suitably close to the target node, all shortcuts will connect further away from the target, making the shortcuts ineffective.

The case of $\ell = 2$ is unique. Here, shortcuts can connect nodes with both different and identical marks. As established in Lemma~\ref{shortcut-length-distribution}, the distance $d$ between the two ends of a shortcut satisfies $\log(d) = \Omega(\log(n))$,  allowing for sublinear distances as well. Note that the marks of two nodes with $o(n)$ distance converge to the same value in the limit. Consequently, at $\ell = 2$, the shortcuts are effective not only when the target is far (i.e. when the mark of the target is different from the source) but also when it is near (i.e. when the mark of the target is the same as the source), leading to an optimal scenario for a decentralized search algorithm. 

\section{Other Related Work}

Small-world networks have been extensively studied due to their unique properties. \cite{fraigniaud2010searchability,nguyen2005analyzing} extended Kleinberg’s model to arbitrary base graphs, showing that the structure of the base graph, combined with long-range connections, influences the performance of greedy algorithms in decentralized search. 
Other studies confirm how random shortcuts or clustering reshape network distances and connectivity \cite{benjamini2001diameter,bollobas2011sparse, coppersmith2002diameter}. 
Other lines of work study random walks, search algorithms, and routing in small-world networks \cite{dietzfelbinger2009tight, dyer2020random, fraigniaud2010searchability, giakkoupis2011optimal, mehrabian2016randomized, nguyen2005analyzing}. 
Our results complement prior work by establishing the local convergence of small-world networks, allowing for a unified analysis of both local structures and dynamic properties in this class of networks. Note that, although we treat rings and regular lattices, we expect analogous local-limit descriptions to hold for other base graphs (e.g., higher-dimensional lattices). By proving the local limits of small-world models, we open new avenues for rigorous analysis of network properties and dynamics using local convergence techniques.

\subsection*{Local Convergence of Random Graph Models}
Notably, there is a large body of work proving that random graph models converge locally in probability. The class of random network models with local limit in probability includes sparse inhomogeneous random graphs (including stochastic block models) \cite[Chapter 3]{BolJanRio07, remco}),  configuration models, \cite[Chapter 4]{DemMon10b, remco}, 
preferential attachment models \cite[Chapter 5]{BerBorChaSab14, GarHazHofRay22, remco}), random intersection graphs \cite{Kura15, HofKomVad18}, and spatial inhomogeneous random graphs  \cite{van2021local}. The latter includes hyperbolic random graphs \cite{KomLod20, KriPapKitVahBog10}. 
Another related work is \cite{bollobas2011sparse}, where they introduce a class of sparse random graphs with high local clustering that converge locally in probability. However, as noted in their paper, this model does not apply to small-world networks, as the mathematical intractability arises from the dependence between edges.

To the best of our knowledge, we are the first to establish the local convergence of small-world networks.  The closest related work is by \cite{barbour2006discrete}, who studied the distance of typical paths in the Watts-Strogatz model and described a branching process governing the structure of shortcuts. While they did not prove local convergence in probability, their work inspired our development of the  $k$-Fuzz structure, which fully characterizes the limit graph of the Watts-Strogatz model and enables us to rigorously prove local convergence in probability. For the Kleinberg model, we are not aware of any prior work that addresses local convergence or provides a comprehensive analysis of its local structure. We additionally note that our results can easily be extended to analyze the Kleinberg model where the base lattice is higher-dimensional ($d >2$) or has toroidal boundary conditions.

Finally, we note that the study of local graph properties and the convergence of local functionals—such as those discussed in this paper, including spectral distributions, the size of independent sets, and other network dynamics \cite{gamarnik2006maximum, ganguly2022hydrodynamic, garavaglia2018local, lacker2019local}—remains an active area of research. Any future results demonstrating that a property is local for convergent graph sequences can be directly applied to small-world network models using the results established in this work.

\section{Preliminaries} \label{prelims}

We begin by introducing some standard asymptotic notation. For a function \(f(n)\) depending on \(n\), we write \(f(n) = \omega(g(n))\) if \(\lim_{n \to \infty} \frac{f(n)}{g(n)} = \infty\);  \(f(n) = o(g(n))\) if \(\lim_{n \to \infty} \frac{f(n)}{g(n)} = 0\); \(f(n) = \Omega(g(n))\) if \(\lim_{n \to \infty} \frac{f(n)}{g(n)} \geq c > 0\); \(f(n) = O(g(n))\) if \(\lim_{n \to \infty} \frac{f(n)}{g(n)} \leq C > 0\). 
In addition, whenever we write \(O_{a_1,\dots,a_m}(1)\), we mean a quantity bounded in absolute value by a constant that may depend on the fixed parameters \(a_1,\dots,a_m\), but is independent of \(n\). These notations will allow us to describe the growth behavior of functions on graphs as the number of nodes grows.

The Benjamini--Schramm notion of local convergence remains meaningful even when degrees are unbounded, provided the expected degree of each node is uniformly bounded, as observed, for example, by Lyons~\cite{Lyons2005} and others.
Therefore, we first show that, for $WS(n,\phi,k)$, $WS(\phi,k)$, $K_<(n,q,k,\ell)$, $K_{=}(q,k)$, and $K_{>}(q,k,\ell)$, the expected degree of every node is bounded. This establishes \textit{local finiteness} and allows us to apply the theory of local convergence. We begin with the Watts--Strogatz model.

\begin{lemma}[Local finiteness of the WS model and its limit object] \label{prelim-1}
   For any integers, $n\geq k$, and any $\phi\in[0,1]$,  $WS(n, \phi, k)$ and $WS(\phi, k)$ are  locally finite.
\end{lemma}

\begin{proof}
The expected degree of any node in both models is bounded uniformly by $3k$. 
\end{proof}

Next, we will consider the Kleinberg model. It will be useful to establish two notions of distance: we will use $d(u, v)$ to denote the lattice distance between nodes, i.e. $d(u, v) = |x_u - x_v| + |y_u - y_v|$, and $d_G(u, v)$ to denote the graph distance between the nodes in graph $G$.

\begin{lemma}[Local finiteness of the Kleinberg model] \label{prelim-2} Let $K(n,q,k,\ell)$  be defined as in Section~\ref{kleinberg-process}. Then
$K(n,q,k,\ell)$ is locally finite. 
\end{lemma}
\begin{proof}

Let $G_n$ be $K(n,q,k,\ell)$. The number of lattice edges of any node $u$ of $G_n$ is $2k(k+1)$. We know that its outgoing shortcut degree is $q$. For the expected incoming shortcut degree, notice that the sum $\sum_{v \neq u} [d(u, v)]^{-\ell}$ will have $n^2 -1$ terms for all $u \in V(G_n)$. This sum is maximized when $u$ is the center node and minimized when $u$ is a corner node. Therefore, we have the following two inequalities:
\begin{equation} \label{corner}
    \sum_{v \neq u} [d(u, v)]^{-\ell} \geq \sum_{j=1}^{n-1} (j+1) j^{-\ell} + \sum_{j=1}^{n-1} (n-j) (n-1+j)^{-\ell} > \sum_{j=1}^{n-1} j^{1-\ell},
\end{equation}
In fact, the first inequality above is tight when $u$ is the corner node. Further, for the upper bound (when $u$ is a center node), we get
\begin{equation} \label{center}
   \sum_{v \neq u} [d(u, v)]^{-\ell} < 4\sum_{j=1}^{n-1} j^{1-\ell}.
\end{equation}
Then, the expected number of incoming shortcuts of node $u$ is given by
$$\Lambda_{u}^{(n)} = q \sum_{w\neq u} \frac{[d(w, u)]^{-\ell}}{\sum_{v \neq w} [d(w, v)]^{-\ell}} \leq \frac{4q \sum_{j=1}^{n-1} j^{1-\ell}}{\sum_{j=1}^{n-1} j^{1-\ell}} \leq 4q.$$
\end{proof}

Next, we will establish that the limit of the Kleinberg model is locally finite. 
To prove this, we need to ensure that the number of incoming edges originating from infinitely many nodes is controlled.

\begin{lemma}[Local finiteness of the Kleinberg limit object]\label{prelim-3} Let $K(q,k,\ell)$ be the local limit of $K(n, q,k,\ell)$ given by Theorems~\ref{main-thm-kl-1},~\ref{main-thm-kl-2},~\ref{main-thm-kl-critical} for $\ell<2, \ell=2,$ and $\ell >2$ respectively. Then
$K(q,k,\ell)$ is locally finite. 
\end{lemma}

\begin{proof}
    Let the incoming degree of node $u$ in $K(q, k, \ell)$ be $D_{u}$. Since the number of outgoing shortcuts and lattice connections is $2k(k+1) + q$, it suffices to show that $\mathbbm{E}[D_u] < \infty$. 
    If $\ell < 2$, then $\mathbbm{E}[D_u] =: \Lambda_m$ defined in Section~\ref{local-limit-kl-1}, where $m$ is the mark of node $u$, converges and is finite.

    For $\ell = 2$, each node $u$ has $\mathrm{Poi}(q)$ incoming shortcuts, giving $\mathbbm{E}[D_u] = q$.

    Now, let's consider $\ell > 2$. In this case, by symmetry, $D_{u}$ will be the same for all $u$. The probability that a shortcut is going to connect to it from node $v$ such that $d(u, v) = d$ is $d^{-\ell}/(4\zeta(\ell-1))$.
    Then, the moment-generating function of $D_u$ is
%     \begin{align*}
%     \mathbbm{E}[e^{tD_u}]&=
%     \prod_{d=1}^\infty \left(1-\frac{d^{-\ell}}{4\zeta(\ell-1)}+\frac{d^{-\ell}e^t}{4\zeta(\ell-1)}\right)^{4qd} \\
%     &\leq  \prod_{d=1}^\infty e^{4qd\log\big(\exp\left((e^t-1)\frac{d^{-\ell}}{4\zeta(\ell-1)}\right)\big)} =  \prod_{d=1}^\infty e^{q(e^t-1)\frac{d^{1-\ell}}{\zeta(\ell-1)}} \\
%     &=  e^{q(e^t-1)\sum_{d=1}^\infty\frac{d^{1-\ell}}{\zeta(\ell-1)}} = e^{q(e^t-1)}.
% \end{align*}
\begin{align*}
        \mathbbm{E}[e^{tD_u}]
    &= \prod_{d=1}^\infty \left(1+(e^t-1)\,\frac{d^{-\ell}}{4\zeta(\ell-1)}\right)^{4qd} \\
    &\leq \prod_{d=1}^\infty \exp\!\left(4qd\cdot(e^t-1)\,\frac{d^{-\ell}}{4\zeta(\ell-1)}\right)
       \quad\text{(using }1+x\le e^x\text{)}\\
    &= \prod_{d=1}^\infty \exp\!\left(q(e^t-1)\,\frac{d^{1-\ell}}{\zeta(\ell-1)}\right)
     = \exp\!\left(q(e^t-1)\sum_{d=1}^\infty\frac{d^{1-\ell}}{\zeta(\ell-1)}\right)
     = e^{q(e^t-1)}.
\end{align*}

Since the moment-generating function exists in a neighborhood of $0$, all moments of $D_u$ are finite. In particular, for $\ell > 2$, we have $\mathbbm{E}[D_u] < \infty$.
\end{proof}

Finally, we will state a useful lemma that we will frequently refer to in our proofs. For a rooted graph $(G,o)$ and $r\ge 0$, write
\(
\partial B_r(G,o):=\{v\in V(G): d_G(o,v)=r\}.
\)

\begin{lemma} \label{prelim-4}
Let $(G,o)$ be a rooted random graph. Assume that there exists a constant $d\ge 1$ such that, for every $t\ge 0$,
\[
\mathbb E\!\left[\,|\partial B_{t+1}(G,o)| \,\middle|\, B_t(G,o)\right]
\le d\,|\partial B_t(G,o)| \qquad \text{a.s.}
\]
Then, for every integer $r\ge 0$,
\[
\mathbb E[|\partial B_r(G,o)|]\le d^r,
\qquad
\mathbb E[|B_r(G,o)|]\le \sum_{j=0}^r d^j.
\]
 Furthermore, if $\{(G_n,o_n)\}_{n\ge 1}$ is a sequence of rooted random
  graphs each satisfying the hypothesis with the same constant $d$, then
  for any $f(n)\to\infty$ and any fixed $\epsilon>0$,
  \[
  \lim_{n\to\infty}
  \mathbb P\!\left[\frac{|B_r(G_n,o_n)|}{f(n)}\ge \epsilon\right]=0.
  \]
% Further, for any function $f(n)$ with $f(n)\to\infty$ and any fixed $\epsilon>0$,
% \[
% \lim_{n\to\infty}
% \mathbb P\!\left[\frac{|B_r(G_n,o_n)|}{f(n)}\ge \epsilon\right]=0
% \]
% for every sequence $(G_n,o_n)$ satisfying the same bound with the same constant $d$.
\end{lemma}

\begin{proof}
We first prove the bound on $\mathbb E[|\partial B_r(G,o)|]$ by induction on $r$. Since
\(
|\partial B_0(G,o)|=1,
\)
we have $\mathbb E[|\partial B_0(G,o)|]=1=d^0$. Now assume that
\(
\mathbb E[|\partial B_r(G,o)|]\le d^r.
\)
Then
\begin{align*}
\mathbb E[|\partial B_{r+1}(G,o)|]
=
\mathbb E\!\left[\mathbb E\!\left[|\partial B_{r+1}(G,o)|\,\middle|\, B_r(G,o)\right]\right] \le
\mathbb E\!\left[d\,|\partial B_r(G,o)|\right]
\le d^{r+1}.
\end{align*}
Thus $\mathbb E[|\partial B_r(G,o)|]\le d^r$ for all $r\ge 0$. Since
\(
B_r(G,o)=\bigsqcup_{j=0}^r \partial B_j(G,o),
\)
we obtain
\(
|B_r(G,o)|=\sum_{j=0}^r |\partial B_j(G,o)|,
\)
and therefore
\(
\mathbb E[|B_r(G,o)|]
\le \sum_{j=0}^r \mathbb E[|\partial B_j(G,o)|]
\le \sum_{j=0}^r d^j.
\) Finally, Markov's inequality gives
\[
\mathbb P\!\left[\frac{|B_r(G_n,o_n)|}{f(n)}\ge \epsilon\right]
\le
\frac{\mathbb E[|B_r(G_n,o_n)|]}{\epsilon f(n)}
\le
\frac{\sum_{j=0}^r d^j}{\epsilon f(n)},
\]
which tends to $0$ as $n\to\infty$ since $f(n)\to\infty$.
\end{proof}

Note that all graph models considered in this paper satisfy the hypothesis of Lemma~\ref{prelim-4}.\footnote{To see this, condition on $B_t(G,o)$, and let $F:=\partial B_t(G,o)$. Every node in $\partial B_{t+1}(G,o)$ is discovered from some node of $F$, so $|\partial B_{t+1}(G,o)|\le \sum_{u\in F} N_u,$
where $N_u$ is the number of previously unexplored neighbors of $u$. Hence it suffices to show that $\mathbb E[N_u\mid B_t(G,o)]\le d$
uniformly in $u$. For the Watts--Strogatz models, each node has expected incoming shortcut degree at most $\phi k$. Thus, one may take $d=2k+\phi k\le 3k$. For the finite Kleinberg model $K(n,q,k,\ell)$, each node has expected incoming shortcut degree at most $4q$ by Lemma~\ref{prelim-2}. Thus, one may take $d=2k(k+1)+5q.$
For the limiting Kleinberg models, in the critical and supercritical cases, the expected incoming shortcut degree is $q$, so one can take $d=2k(k+1)+2q.$
In the subcritical case, the incoming shortcut degree is $\mathrm{Poi}(\Lambda_m)$, where $\Lambda_m$ is finite for every mark $m$ by construction. Since $[0,1]^2$ is compact, $\Lambda_m$ is a continuous function of $m$, and $\sup_m \Lambda_m<\infty$, one can take $d=2k(k+1)+q+\sup_m \Lambda_m.$} From now on, whenever we need a control the radius-$r$ neighborhood of a node, we will use Lemma~\ref{prelim-4} in the following form: for every $\eta>0$, there exists $M=M(r,\eta)$ such that
\[
\sup_n \mathbb P\big[|B_r(G_n,o_n)|>M\big]<\eta.
\]
Accordingly, we may work on the event $\{|B_r(G_n,o_n)|\le M\}$ and let $\eta\downarrow 0$ at the end. 

\section{Local Limit of WS Model: Proof of Theorem~\ref{main-thm-ws}} \label{sec:proof-WS}

Let $G_n$ be the random graph sampled according to $WS(n, \phi, k)$, $o_n$ be the root selected uniformly at random from $G_n$, and $(G, o)$ be the random graph sampled according to $WS(\phi, k)$ with measure $\mu$ defined as in Theorem \ref{main-thm-ws}. 
Define $\mu_n$ to be the measure of $(G_n, o_n) = (WS(n, \phi, k), o_n)$. 
We prove a stronger version of Theorem~\ref{main-thm-ws}, where edges have marks defined as follows.

\subsection{Edge Marks and Ordering} \label{edge-marks}

Recall the construction of the Watts-Strogatz model from Section~\ref{ws-process}. At a high level, the edges of  $(G_n, o_n)$ and $(G, o)$ are divided into two categories: Edges that remain in their original positions in the $k$-ring  are called \emph{ring edges}, while those rewired during the construction are referred to as \emph{shortcuts}. Each edge is assigned two marks based on the following:
\begin{itemize}
    \item  \textit{Ring Distance:} This is the minimum number of steps along the original cycle ($1$-ring) subgraph (the cycle) needed to traverse between the two endpoints of an edge. There are $k$ possible values for this mark. This mark remains unchanged after rewiring: for shortcut edges, the ring distance corresponds to the original edge they replaced during the rewiring process.
    \item \textit{Direction of the Edge:} During the Watts-Strogatz construction, edges are initially directed before being made bidirectional in the final step. Consider a Breadth-First-Search (BFS) process starting from an arbitrary node. An edge is marked as \emph{outgoing} if it is explored from the node it stems from. An edge is marked as \emph{incoming} if it is explored from the node it connects to.
\end{itemize}

As a result of this process, we have $2k$ different marks for ring edges and $2k$ different marks for shortcuts. Hence, we have $4k$ different marks for edges in total.  We use the metric $d_\Xi(m_1, m_2) = 1-\mathbbm{1}(m_1 = m_2)$ for any edge marks $m_1$ and $m_2$, which equips the finite mark set with the discrete metric, hence a separable metric space as required by Definition \ref{metric}. 

Now, we will order the edges of $(G_n, o_n), (G, o)$ (and the subgraph $B_r(G,o)$, denoted by $H_*$) according to the BFS process that starts from the respective root and prioritizes the edges in the following order: outgoing shortcuts, outgoing ring edges, incoming ring edges, and incoming shortcuts, and mark them based on the mark classification above simultaneously. The BFS will prioritize ring edges that connect closer nodes to each other.  Let $\overline{G_n}$, $\overline{G}$, and $\overline{H}_*$ be the version of $G_n, G$ and $H_*$ with ordered and marked edges respectively. See Figure \ref{tree} for an example with $n=12$ and $k=1$.

\begin{figure}
\begin{center}
  \begin{tikzpicture}[scale=0.6]
        \begin{scope}
            \foreach \i\col\k in {1/teal/9, 2/violet/11, 3/blue-shortcut/6, 4/red/3, 5/teal/5, 6/teal/7, 7/blue-shortcut/4, 8/blue-shortcut/2, 9/black/0, 10/teal/1, 11/blue-shortcut/10, 12/red/8} {
            \node [sdot, \col] (\i) at (-30*\i:2) {};
            \node   at (-30*\i:2.5) {\tiny $\k$};
        }
            \draw (1) -- (2);
            \draw (2) -- (12);
            \draw (3) -- (4);
            \draw (4) -- (5);
            \draw (5) -- (6);
            \draw (6) -- (12);
            \draw (7) -- (8);
            \draw (8) -- (9);
            \draw (9) -- (10);
            \draw (10) -- (4);
            \draw (11) -- (12);
            \draw (12) -- (1);
        \end{scope}
        \begin{scope}[xshift=4cm, rotate=90]
            \node (1) [sdot, black] {}
            child {node (3) [sdot, blue-shortcut] {} child {node (5) [sdot, blue-shortcut] {
            } }}
            child {node (c) [sdot, teal] {} child {node (4) [sdot, red] {} child {node (7) [sdot, blue-shortcut] {
            } } child {node (6) [sdot, teal] {} child {node (8) [sdot, teal] {} child {node (9) [sdot, red] {} child {node (11) [sdot, blue-shortcut] {} } child {node (10) [sdot, teal] {} } child {node (12) [sdot, violet] {} }}
            }}}}
        ;
        \end{scope}
    \end{tikzpicture}
\end{center} 
   \caption{(left) Ordered and marked $WS(12, \phi, 1)$ (right) viewing same graph as a tree. Since $WS(12, \phi, 1)$ is a tree, there is a $1$-to-$1$ correspondence between the edges and the non-root nodes. Thus, we instead ordered and marked the child of the edge in the BFS with the corresponding edge mark. The root has a special mark. \\ Marks: \(\bullet\) (root), {\color{teal}{\(\bullet\)}} (outgoing ring edge), {\color{blue-shortcut}{\(\bullet\)} }(incoming ring edge), {\color{violet}{\(\bullet\)}} (incoming shortcut), {\color{red}{\(\bullet\)}} (outgoing shortcut).}
    \label{tree}
\end{figure}
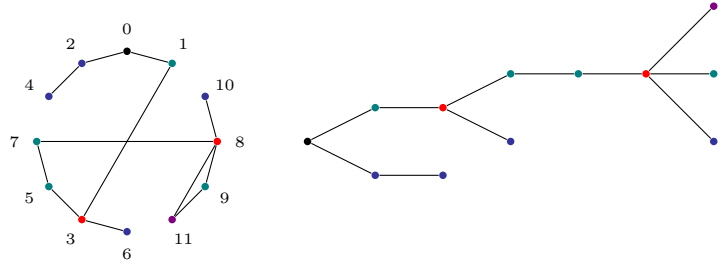

\subsection{Main Lemmas}

The main idea of the proof is a second-moment argument showing the convergence of the marked Watts-Strogatz model to its limit. To set this up, recall the definition of the indicator \(\mathbbm{1}_r^\epsilon\) from Definition~\ref{local-convergence-def}. Given a marked graph \((H_*, o_*, M_*)\), define
\begin{equation} \label{exp1} \mu_{r,n}^\epsilon((G_n,o_n)\simeq H_*) :=\mathbb{E}_{\mu_n}[\mathbbm{1}_r^\epsilon((G_n,o_n)\simeq H_*)],
\end{equation} 
and
\begin{equation} \label{exp2} \mu_{r}^\epsilon((G,o)\simeq (H_*,o_*)) :=\mathbb{E}_{\mu}[\mathbbm{1}_r^\epsilon((G, o)\simeq (H_*,o_*))].\end{equation} 
Without a loss, we can assume $\epsilon = 0$ because of our metric, and, for simplicity, omit the use of $\epsilon$ from now on.

Using the criterion given by Definition \ref{local-convergence-def}, it suffices to prove the convergence  for $(H_*, o_*, M_*)$  where  $\mu_r(H_*, o_*, M_*)>0$, i.e., the rooted marked graph $B_r(H_*, o_*, M_*)$ appears with a positive probability. Fix graph $H_*$ such that $\mu_r(H_*, o_*, M_*) > 0$. Given that $r$ is finite, by Lemma~\ref{prelim-1} and Lemma~\ref{prelim-4}, any $H_*$ considered from now on will be a finite graph. Let \(B_r(\overline{G}, o) = \overline{H}_*\) denote order equivalence; that is, the \(r\)-neighborhoods of the corresponding graphs are identical when comparing only nodes with the same order in the given orderings of \(\overline{G}\) and \(\overline{H}_*\). Notice that this ordering is not unique. However, in the event that $B_r(\overline{G}, o) = \overline{H}_*$, any other ordering of $G$ would be isomorphic to $H_*$ since the underlying structure is not changed. Thus, we get the following identity:
\begin{equation} \label{order-equivalence}
    \mu_r((G, o) \simeq\ H_*) =  \#(H_*) \cdot \mu_r((\overline{G}, o) = \overline{H}_*)
\end{equation}
where $\#(H_*)$ is the number of different ways we could order this marked graph based on the BFS rules outlined. Similarly, we have $\mu_{r,n}((G_n, o_n) \simeq\ H_*) =  \#(H_*) \cdot \mu_{r,n}((\overline{G_n}, o_n) = \overline{H}_*)$. Then, it suffices to prove Lemma~\ref{main-lemma-ws} below to show the convergence in probability.

\begin{lemma} [Local Convergence of Watts-Strogatz Model]  \label{main-lemma-ws} 

Fix any $r > 0$ and $H_* = B_r(H_*, o_*, M_*)$ such that $\mu_r(H_*, o_*, M_*)>0$ where $\mu$ is the measure of $WS(\phi, k)$ defined in section \ref{local-limit-ws}. Order and mark $G_n, G$, and $H_*$ according to the process outlined in section \ref{edge-marks} to get $\overline{G_n}, \overline{G}$ and $\overline{H}_*$. Then,
\begin{equation} \label{main-conv-ws}
    \frac{1}{n}\sum_{o_n\in V(G_n)}  \mathbbm{1}_r((\overline{G_n},o_n) = \overline{H}_* )\overset{\mathbb P}{\to}\mu_r((\overline{G},o) =  \overline{H}_*) \quad \text{ as $n \to \infty$.}
\end{equation}  
\end{lemma}

At a high level, we will first show that the expectation of the left-hand side in \eqref{main-conv-ws} converges to $\mu_r((\overline{G},o) =  \overline{H}_*)$ and then show that its variance goes to $0$ with $n$. This is formalized in the next two Lemmas. 

\begin{lemma}[First Moment: Watts-Strogatz Model] \label{ws-1moment} Fix any $r>0$ and $H_*$, and define $\overline{H_*}, (\overline{G_n}, o_n)$, $(\overline{G}, o), \mu_n$ and $\mu$ as in Lemma~\ref{main-lemma-ws}. Then,
\begin{equation*} 
\mu_{r,n}((\overline{G_n},o_n) = \overline{H}_*) \to \mu_{r}((\overline{G},o) = \overline{H}_*) \quad \text{ as $n \to \infty$.}
\end{equation*}
\end{lemma}

\begin{lemma}[Second Moment: Watts-Strogatz Model] \label{ws-2moment} Fix any $r>0$ and $H_*$ as in Lemma~\ref{main-lemma-ws}. Then,
\begin{equation*} 
\text{Var}\left(\frac{1}{n}\sum_{o_n\in V(G_n)}  \mathbbm{1}_r((G_n,o_n)\simeq H_*) \right) \to 0 \quad \text{ as $n \to \infty$.}
\end{equation*}
\end{lemma}

Now, assuming these two lemmas, one can easily prove Lemma~\ref{main-lemma-ws}.

\subsubsection*{Proof of Lemma~\ref{main-lemma-ws} and Theorem~\ref{main-thm-ws}}
Take any $\epsilon>0$, and let
\[
\textbf{X}_n:=\frac{1}{n}\sum_{o_n\in V(G_n)} \mathbbm{1}_r\big((G_n,o_n)\simeq H_*\big).
\]
By Lemma~\ref{ws-1moment}, for all sufficiently large $n$,
\[
\left|\mu_{n,r}\big((G_n,o_n)\simeq H_*\big)-\mu_r\big((G,o)\simeq H_*\big)\right|<\frac{\epsilon}{2}.
\]
Therefore, for all sufficiently large $n$,
\[
\mathbb P\!\left[\left|\textbf{X}_n-\mu_r((G,o)\simeq H_*)\right|\ge \epsilon\right]
\le
\mathbb P\!\left[\left|\textbf{X}_n-\mu_{n,r}((G_n,o_n)\simeq H_*)\right|\ge \epsilon/2\right].
\]
Applying Chebyshev's inequality,
\[
\mathbb P\!\left[\left|\textbf{X}_n-\mu_{n,r}((G_n,o_n)\simeq H_*)\right|\ge \epsilon/2\right]
\le
\frac{4\,\mathrm{Var}(\textbf{X}_n)}{\epsilon^2}.
\]
By Lemma~\ref{ws-2moment}, $\mathrm{Var}(\textbf{X}_n)\to 0$, so
\[
\frac{1}{n}\sum_{o_n\in V(G_n)} \mathbbm{1}_r\big((G_n,o_n)\simeq H_*\big) \xrightarrow{\mathbb P} \mu_r\big((G,o)\simeq H_*\big).
\]
This proves Lemma~\ref{main-lemma-ws}, and hence Theorem~\ref{main-thm-ws}. \qed 

The rest of this section is dedicated to the proof of Lemmas \ref{ws-1moment} and \ref{ws-2moment}.

\subsection*{Proof of Lemma~\ref{ws-1moment}}

By symmetry, it suffices to show that for a fixed node $o_n \in V(G_n)$,
\begin{equation}\label{goal1}
\mu_n\big(B_r(\overline{G}_n,o_n)=\overline{H}_*\big)\to
\mu\big(B_r(\overline{G},o)=\overline{H}_*\big)
\qquad\text{as }n\to\infty.
\end{equation}

Fix $r>0$ and a finite ordered marked graph $\overline{H}_*$ with $\mu(B_r(\overline{G},o)=\overline{H}_*)>0$. Let $P_k$ denote the infinite path in which each node is connected to its $k$ nearest neighbors on either side (the Full $k$-Path). If $n>2kr+1$, then before rewiring,
\begin{equation}\label{pre-ring}
B_r(G_n,o_n)\cong B_r(P_k,o).
\end{equation}
Indeed, every path of graph length at most $r$ moves by at most $k$ positions along the ring at each step, so every node in $B_r(G_n,o_n)$ lies within cyclic distance at most $kr$ of $o_n$. This neighborhood cannot wrap around the cycle. Thus, for all sufficiently large $n$, the pre-rewiring radius-$r$ neighborhood is deterministic. 

After rewiring, the neighborhood is determined by the following local data revealed by the BFS:
\begin{itemize}
    \item for each explored node $u$ and each outgoing edge $m\in[k]$, whether
    the outgoing ring edge of type $m$ is rewired;
    \item for each explored node $u$, the ordered list of marks of the incoming
    shortcuts incident to $u$.
\end{itemize}
Indeed, if an outgoing edge with mark $m$ is not rewired, it contributes the corresponding outgoing ring-edge mark; if it is rewired, it contributes an outgoing shortcut with mark $m$. Likewise, the ordered list of incoming shortcut marks determines the incoming shortcut edges. Since the outgoing rewiring decisions already have the correct law in both models, the only asymptotic differences could be:
\begin{enumerate}
    \item the number of incoming shortcuts at each explored node;
    \item the possibility that a shortcut creates a local cycle;
    \item the marks of the incoming shortcuts. \label{shortcut-mark}
\end{enumerate}

Let $v_{0,n},v_{1,n},\dots,v_{|H_*|-1,n}$ be the nodes of
$B_r(\overline{G}_n,o_n)$ in BFS order, and let
$v_0,v_1,\dots,v_{|H_*|-1}$ be the corresponding nodes of
$B_r(\overline{G},o)$. Let $x_i$ be the number of incoming shortcuts of the
$i$-th node of $\overline{H}_*$. Define $T_i$ to be the event that the first
$i+1$ explored nodes in $\overline{G}_n$ and $\overline{G}$ agree with
$\overline{H}_*$ in BFS order, together with all edge marks seen so far.

For each $i$, let $\mathbf{Good}_i$ be the event that every shortcut endpoint
revealed at time $t$ while exploring $v_{i,n}$ lies far enough from the already exposed set in the BFS up until time $t$ so that it neither creates a cycle containing a shortcut edge nor causes a
new $k$-Fuzz component to wrap around and intersect a previously exposed fuzz.
Hence, it suffices to prove that
\begin{align}
\prod_{i=0}^{|H_*|-1}
\mu_n\big(v_{i,n}\text{ has }x_i\text{ incoming shortcuts and }\mathbf{Good}_i
\mid T_{i-1}\big)
\label{ws-goal}
\end{align}
converges to
\[
\prod_{i=0}^{|H_*|-1}\mu\big(v_i\text{ has }x_i\text{ incoming shortcuts}\mid
T_{i-1}\big)
=
\prod_{i=0}^{|H_*|-1}\frac{e^{-\phi k}(\phi k)^{x_i}}{x_i!},
\]
and that, conditional on these events, the ordered shortcut marks converge to the corresponding law in the limit. We begin by identifying the asymptotic distribution of the number of incoming shortcuts, which turns out to be Poisson. Intuitively, this comes from a law-of-rare-events phenomenon: each node may receive a shortcut from many possible source nodes, and these contributions are nearly independent, with each individual probability becoming very small. 

\begin{lemma}\label{lm-degdist-WS}
Let $x_i$ be the number of incoming shortcuts of the $i$-th node of
$\overline{H}_*$. Then
\begin{equation}\label{product}
\prod_{i=0}^{|H_*|-1}
\mu_n\big(v_{i,n}\text{ has }x_i\text{ incoming shortcuts}\mid T_{i-1}\big)
\to
\prod_{i=0}^{|H_*|-1}\frac{e^{-\phi k}(\phi k)^{x_i}}{x_i!}
\qquad\text{as }n\to\infty.
\end{equation}
\end{lemma}

\begin{proof}
Fix $i$. Conditional on $T_{i-1}$, only a bounded number of directed ring edges
have already had their rewiring status revealed, or are otherwise excluded from
contributing an incoming shortcut to $v_{i,n}$. Since $r$ and $\overline{H}_*$
are fixed, there exists a constant $C=C(r,k,\overline{H}_*)<\infty$ and an
integer $c_{i,n}\in[0,C]$ such that
\[
N_{i,n}:=\#\{\text{incoming shortcuts of }v_{i,n}\}\ \Big|\ T_{i-1}
\sim \mathrm{Bin}\!\left(nk-c_{i,n},\frac{\phi}{n}\right).
\]
Indeed, each eligible directed ring edge independently rewires to $v_{i,n}$ with
probability $\phi/n$. Therefore, for fixed $x_i$,
\[
\mu_n\big(v_{i,n}\text{ has }x_i\text{ incoming shortcuts}\mid T_{i-1}\big)
=
\binom{nk-c_{i,n}}{x_i}
\left(\frac{\phi}{n}\right)^{x_i}
\left(1-\frac{\phi}{n}\right)^{nk-c_{i,n}-x_i}.
\]
Since $c_{i,n}=O(1)$ and $x_i$ is fixed, the standard binomial-to-Poisson limit
gives
\[
\mu_n\big(v_{i,n}\text{ has }x_i\text{ incoming shortcuts}\mid T_{i-1}\big)
\to
\frac{e^{-\phi k}(\phi k)^{x_i}}{x_i!}.
\]
Finally, $|H_*|<\infty$, so multiplying over $i=0,\dots,|H_*|-1$ yields
\eqref{product}.
\end{proof}

Next we show that there is no \textit{local} cycle that contains a shortcut with high probability. This is the key simplification in the limit.

\begin{lemma}\label{no cycle}
For each $i$, condition on
\(
T_{i-1}\cap\{v_{i,n}\text{ has }x_i\text{ incoming shortcuts}\}.
\)
Order all shortcut endpoints revealed while exploring $v_{i,n}$. Let $S_{i,j}$ be the set of nodes already
exposed by the BFS just before the $j$-th such shortcut endpoint is revealed,
and define the forbidden set
\[
F_{i,j}:=\{u\in V(G_n): d_{\mathrm{ring}}(u,S_{i,j})\le 2kr\},
\]
where $d_{\mathrm{ring}}$ denotes distance in the original $1$-ring (the cycle).
Let $\mathbf{Good}_i$ be the event that every shortcut endpoint revealed while
exploring $v_{i,n}$ lies outside the corresponding forbidden set $F_{i,j}$.
Then
\[
\prod_{i=0}^{|H_*|-1}
\mu_n\big(\mathbf{Good}_i\mid
v_{i,n}\text{ has }x_i\text{ incoming shortcuts and }T_{i-1}\big)
\to 1
\qquad\text{as }n\to\infty.
\]
\end{lemma}

\begin{proof}
Since $r$ and $\overline{H}_*$ are fixed, there exists a
constant $C=C(r,k,\overline{H}_*)$ such that, uniformly in $i,j,n$,
\(
|S_{i,j}|\le C.
\)
Each node contributes at most $4kr+1$ nodes within ring distance $2kr$, so
there exists a constant $C'=C'(r,k,\overline{H}_*)$ such that
\(
|F_{i,j}|\le C'
\)
for all $i,j,n$
We now bound the probability that a revealed shortcut endpoint is \textit{bad}.

\noindent\textbf{Outgoing shortcut:}
On the event that
the corresponding outgoing edge is rewired, its target is uniform on $V(G_n)$.
Hence, we get
\[
\mu_n\big(
\text{the $j$-th outgoing shortcut target lies in }F_{i,j}
\mid v_{i,n}\text{ has }x_i\text{ incoming shortcuts and } T_{i-1}
\big)
\le \frac{|F_{i,j}|}{n}
\le \frac{C'}{n}
= O(n^{-1}).
\]

\noindent\textbf{Incoming shortcut:}
The source is chosen from the directed ring
edges of nodes outside of the ones that are already exposed. At least $k(n-C)$ such edges remain eligible, and at most $k|F_{i,j}|\le
kC'$ of them have source node in $F_{i,j}$. Therefore, we have
\[
\mu_n\big(
\text{the $j$-th incoming shortcut source lies in }F_{i,j}
\mid v_{i,n}\text{ has }x_i\text{ incoming shortcuts and } T_{i-1}
\big)
\le \frac{kC'}{k(n-C)}
= O(n^{-1}).
\]

Since $v_{i,n}$ has at most $k$ outgoing shortcuts and exactly $x_i$ incoming
shortcut endpoints revealed at this stage, only finitely many shortcut endpoints
are considered while exploring $v_{i,n}$. A union bound therefore gives
\[
\mu_n\big(\mathbf{Good}_i^c\mid
v_{i,n}\text{ has }x_i\text{ incoming shortcuts and }T_{i-1}\big)
=O\!\left(n^{-1}\right),
\]
uniformly in $i$. Multiplying over the finitely many nodes of
$\overline{H}_*$ gives the claim.
\end{proof}

Now, assume that $\cap_i\mathbf{Good}_i$ occurs. Then every shortcut
endpoint revealed by the BFS lies at ring distance greater than $2kr$ from the
already exposed set at the moment it is created. Therefore
the radius-$r$ ring neighborhood of a newly created shortcut endpoint is
disjoint from the radius-$r$ ring neighborhoods of all previously exposed
nodes and will be for the rest of the BFS. It remains to prove the convergence of the shortcut marks in
item~\eqref{shortcut-mark}. For each $i$, let
\[
E_i:=
T_{i-1}\cap
\{v_{i,n}\text{ has }x_i\text{ incoming shortcuts}\}\cap
\mathbf{Good}_i,
\]
and let
\[
(M_{i,1},\dots,M_{i,x_i})
\]
denote the ordered marks of the incoming shortcuts of $v_{i,n}$. For $m\in[k]$ and $1\le j\le x_i$, let $N_m^{(i,j)}$ be the number of eligible
directed ring edges of mark $m$ just before the $j$-th incoming shortcut of
$v_{i,n}$ is chosen, and let
\[
N^{(i,j)}:=\sum_{m=1}^k N_m^{(i,j)}.
\]
Initially there are exactly $n$ directed ring edges of each mark $m$ and $nk$
directed ring edges in total.
The only non-eligible edges at stage $(i,j)$ are:
\begin{itemize}
    \item edges already exposed by the BFS;
    \item edges whose source lies in the forbidden set $F_{i,j}$;
    \item the $j-1$ source edges already chosen for the previous incoming
    shortcuts of $v_{i,n}$.
\end{itemize}
Since $r$, $k$, and $\overline{H}_*$ are fixed, there exists a constant
$C=C(r,k,\overline{H}_*)$ such that, uniformly in $i,j,m$,
\begin{equation}\label{eligible-counts}
N_m^{(i,j)} = n - O_C(1),
\qquad
N^{(i,j)} = nk - O_C(1).
\end{equation}

Conditional on $E_i$, the ordered marks of the incoming shortcuts are sampled
without replacement from these eligible shortcuts. Hence, for every ordered
tuple $(m_1,\dots,m_{x_i})\in[k]^{x_i}$,
\begin{align*}
&\mu_n\Big(
(M_{i,1},\dots,M_{i,x_i})=(m_1,\dots,m_{x_i})
\,\Big|\, E_i
\Big) =
\prod_{j=1}^{x_i}
\frac{N_{m_j}^{(i,j)}}{N^{(i,j)}}.
\end{align*}
By \eqref{eligible-counts},
\[
\frac{N_{m_j}^{(i,j)}}{N^{(i,j)}}
=
\frac{1}{k}+O\!\left(\frac1n\right),
\]
uniformly in $i,j$ and in the previously chosen marks. Since $x_i$ is fixed, we
obtain
\[
\mu_n\Big(
(M_{i,1},\dots,M_{i,x_i})=(m_1,\dots,m_{x_i})
\,\Big|\, E_i
\Big)
=
k^{-x_i}+O\!\left(\frac1n\right)
\to
k^{-x_i},
\]
which is exactly the law of the ordered incoming shortcut marks at $v_i$ in the
limit graph $\overline{G}$.  Since $|H_*|<\infty$,
multiplying the conditional probabilities along the BFS exploration yields convergence of the shortcut marks in item~\eqref{shortcut-mark}.
Together with Lemmas~\ref{lm-degdist-WS} and \ref{no cycle}, this proves that
the joint law of the BFS data at each explored node converges to the
corresponding law in the recursive $k$-Fuzz limit.  \qed

\subsection*{Proof of Lemma~\ref{ws-2moment}} \label{secondmoment} 

For the second moment, we are interested in showing
\begin{equation} \label{goal2}
\begin{split}
&\text{Var}\left(\frac{1}{n} \sum_{o_n\in V(G_n)}  \mathbbm{1}_r((G_n,o_n)\simeq H_* )  \right) \\ &=\mathbb{E}\left[\left(\frac{1}{n} \sum_{o_n\in V(G_n)}  \mathbbm{1}_r((G_n,o_n)\simeq H_* ) \right)^2 \right] - \mathbb{E}\left[\frac{1}{n} \sum_{o_n\in V(G_n)}  \mathbbm{1}_r((G_n,o_n)\simeq H_* )  \right]^2 \to 0.
\end{split}
\end{equation}
By symmetry,
\begin{equation} \label{eq-prob}
\begin{split}
   &\lim_{n \to \infty} \mathbb{E}\left[\left(\frac{1}{n} \sum_{o_n\in V(G_n)}  \mathbbm{1}_r((G_n,o_n)\simeq H_* ) \right)^2 \right] \\ &= \lim_{n \to \infty} \frac{1}{n} \mathbb{P}[B_r(G_n, v) \simeq H_*] + \frac{n(n-1)}{n^2} \mathbb{P}_w[B_r(G_n, v) \simeq H_* \cap B_r(G_n, w) \simeq H_*] \\
    &= \lim_{n \to \infty} \mathbb{P}_w[B_r(G_n, v) \simeq H_* \cap B_r(G_n, w) \simeq H_*]
\end{split}
\end{equation}
where $v$ is some arbitrary node and $w$ is picked uniformly at random from all nodes but $v$ ($\mathbb{P}$ accounts for the randomness of $G_n$ while $\mathbb{P}_w$ accounts for the randomness of $G_n$ as well as $w$). Note that we can rewrite $\mathbb{P}_w[B_r(G_n, v) \simeq H_* \cap B_r(G_n, w) \simeq H_*]$ by conditioning on the event that $v$ and $w$ have at least a distance of $2r$, i.e.,
\begin{equation}
\begin{split} \label{prob-partition}
    &\mathbb{P}_w[B_r(G_n, v) \simeq H_* \cap B_r(G_n, w) \simeq H_* | \text{dist}_{G_n}(v, w) \leq 2r] \mathbb{P}_w[\text{dist}_{G_n}(v, w) \leq 2r] \\
    &+ \mathbb{P}_w[ B_r(G_n, w) \simeq H_* | B_r(G_n, v) \simeq H_* \cap \text{dist}_{G_n}(v, w) > 2r] \mathbb{P}_w[B_r(G_n, v) \simeq H_* \cap \text{dist}_{G_n}(v, w) > 2r].
\end{split}
\end{equation} 

Now, let's evaluate the former probability. Because $w$ is picked uniformly at random, by applying Lemma~\ref{prelim-1}, 
\begin{align*} 
\mathbb{P}_w[\text{dist}_{G_n}(v, w) \leq 2r] = \frac{\mathbb{E}[B_{2r}(G_n, v)-1]}{n-1} \leq \frac{(3k)^{2r}}{n-1},
\end{align*}
which converges to $0$ as $n$ grows. This also implies that 
\begin{align*}
    \lim_{n \to \infty} \mathbb{P}_w[B_r(G_n, v) \simeq H_* \cap \text{dist}_{G_n}(v, w) > 2r] &=  \lim_{n \to \infty} \mathbb{P}[B_r(G_n, v) \simeq H_* ]
\end{align*}
which converges to $\mu((G,o)\simeq H_*)$ by Lemma~\ref{ws-1moment}. Next, let's analyze the rest of the terms. We claim  that $$\lim_{n \to \infty} \mathbb{P}_w[ B_r(G_n, w) \simeq H_* | B_r(G_n, v) \simeq H_* \cap \text{dist}_{G_n}(v, w) > 2r] = \lim_{n \to \infty} \mu_{r, n}((G_n,o_n)\simeq H_*).$$
Define $G_n^{v,r}$ to be the graph induced by the nodes in $V(G_n) \backslash V(B_{r}(G_n, v))$. Notice that 
\begin{align*}
    \lim_{n \to \infty} \mathbb{P}_w[ B_r(G_n, w) \simeq H_* | B_r(G_n, v) \simeq H_* \cap \text{dist}_{G_n}(v, w) > 2r] &=  \lim_{n \to \infty} \mathbb{P}_{w\sim V(G_n^{v,2r} )}[ B_r(G_n^{v,r}, w) \simeq H_*].
\end{align*}
where in the second term $w$ is chosen uniformly at random from the nodes of $G_n^{v,2r}$. Now, one can follow very similar steps to the proof of Lemma~\ref{ws-1moment} to show the distribution of $B_r(G_n^{v,r}, w)$ converges to $\mu((G_n,o_n)\simeq H_*)$. This is formalized in the next lemma, whose proof appears in Appendix~\ref{sec:omitted-ws-proof}.

\begin{lemma}\label{omitted-ws-lemma}
Fix any $r>0$ and $H_*$, and define $\mu$ as in Lemma~\ref{main-lemma-ws}. Then,
\begin{equation*} 
    \frac{1}{|G_n^{v,2r}|}\sum_{w\in V(G_n^{v,2r})}  \mathbbm{P} \left[B_r(G_n^{v,r}, w) \simeq H_* \right] \to\mu_r((G,o) =  H_*) \quad \text{ as $n \to \infty$.}
\end{equation*}  
\end{lemma}
Applying Lemma~\ref{omitted-ws-lemma}, we get
\begin{align}
   \lim_{n \to \infty} \left| \mathbb{P}_w[B_r(G_n, w) \simeq H_* | B_r(G_n, v) \simeq H_* \cap \text{dist}_{G_n}(v, w) > 2r] - \mathbb{P}[B_r(G_n, v) \simeq H_*] \right| = 0.
\end{align}
Combining all of these results with \eqref{prob-partition},
\begin{equation*}
    \lim_{n \to \infty} |\mathbb{P}_w [B_r(G_n, v) \simeq H_* \cap B_r(G_n, w) \simeq H_*] -  \mathbb{P}[B_r(G_n, v) \simeq H_*]^2 | \to 0.
\end{equation*}
By \eqref{goal2} and \eqref{eq-prob}, we have
\begin{align*}
    \lim_{n \to \infty} \text{Var}\left(\frac{1}{n} \sum_{o_n\in V(G_n)}  \mathbbm{1}_r((G_n,o_n)\simeq H_* )  \right)  &= \lim_{n \to \infty} \mathbb{P}[B_r(G_n, v) \simeq H_*]^2 - \biggr[ \frac{1}{n} \mathbb{P}[B_r(G_n, v) \simeq H_*]^2 \\
    &\quad + \frac{n(n-1)}{n^2} \mathbb{P}_w[B_r(G_n, v) \simeq H_* \cap B_r(G_n, w) \simeq H_* \biggr] = 0.
\end{align*}
With this, we conclude the proof for Lemma~\ref{ws-2moment} and Theorem \ref{main-thm-ws}. \qed

\begin{figure}
\centering
\begin{tikzpicture}[scale=0.45]
\def\n{9} 
\foreach \i in {1,...,\n} {
    \foreach \j in {1,...,\n} {
        \node [sdot] (\i-\j) at (\i,-\j) {};
    }
}
\node [sdot, red] at (5, -5) {};
\foreach \i in {1,...,\n} {
    \foreach \j [count = \k] in {2, ..., \n} {
            \draw (\i-\k) -- (\i-\j);
            \draw (\k-\i) -- (\j-\i);
        }}
\draw[->, blue-shortcut, thick] (5-5) -- (8-9);
\draw[->, blue-shortcut, thick] (5-5) -- (1-3);
\draw [<->, violet, out = 110, in = 250] (5-5) to (5-3);
\draw [<->, violet, out = 20, in = 160] (5-5) to (7-5);
\draw [<->, violet, out = -70, in = 70] (5-5) to (5-7);
\draw [<->, violet, out = -160, in = -20] (5-5) to (3-5);
\draw [<->, violet] (5-5) to (4-4);
\draw [<->, violet] (5-5) to (6-6);
\draw [<->, violet] (5-5) to (6-4);
\draw [<->, violet] (5-5) to (4-6);
\draw [<->, violet] (5-5) to (5-4);
\draw [<->, violet] (5-5) to (5-4);
\draw [<->, violet] (5-5) to (5-6);
\draw [<->, violet] (5-5) to (4-5);
\draw [<->, violet] (5-5) to (6-5);
\draw [->, teal, thick] (1-7) to (5-5);
\draw [->, teal, thick] (9-8) to (5-5);
\end{tikzpicture}
    \caption{$K(9, 2, 2, \ell)$ for some $\ell$. For simplicity, we plot the directed edges that only belong to the center node and colored edges based on the type: incoming vs. outgoing, lattice edges vs. shortcuts.}
    \label{kleinberg-model}
\end{figure}

\section{Kleinberg Model when $\ell\leq2$: \\ Behavior of Marks and Shortcuts in the Locally Tree-Like Regime} \label{proof-main-thm-kl-1}

Let $G_n$ be a random graph sampled according to $K(n, q, k, \ell)$, $o_n$ be a root selected uniformly at random from $G_n$, and $(G, o)$ be the random graph sampled according to $K_{<}(q, k, \ell)$  with measure $\mu$ defined in Section~\ref{local-limit-kl-1} if $\ell < 2$ and $K_{=}(q, k, \ell)$  with measure $\mu$ defined in Section~\ref{local-limit-kl-critical} if $\ell = 2$ . Define $\mu_n$ to be the measure of $(G_n, o_n) = (K(n, q, k, \ell), o_n)$. Let $V(G_n)$ be the node set of $G_n$ with $n^2$ nodes. 
Define $\mu_{r,n}^{\epsilon}$ and $\mu_r^{\epsilon}$ as in \eqref{exp1} and \eqref{exp2}. Later, we will introduce $\hat{\mu}$ and define $\hat{\mu}_{r, n}^{\epsilon}$ similarly. 

\subsection{Node Marks and Ordering} \label{sec:node-marks-ordering}
We begin by defining marks and the associated metric space. The nodes in \((G, o)\) and \(H_* \simeq B_r(G,v)\) are, by definition, sampled with marks. Recall from Section~\ref{kleinberg-process} that the lattice coordinate of node \(u\) is \((u_x, u_y)\), with the top-left node at \((0,0)\) and each coordinate increasing by \(1\) when moving right and down respectively. For each node \(u \in G_n\), define its mark as
\[
M(u) = \left(\frac{u_x}{n}, \frac{u_y}{n}\right).
\]
We use the \(L_1\) metric $d_\Xi(m_1, m_2) = \|m_1 - m_2\|_1$ for any two marks $m_1$ and $m_2$. Since marks lie in \(\mathbb{R}^2\) with the \(L_1\) metric, the marks belong to a separable space, as required by Definition~\ref{metric}.

Similar to the proof of Theorem~\ref{main-thm-ws}, we will order and mark the nodes of $(G_n,o_n)$, $(G,o)$ and rooted graph $H_*$ following a BFS process that starts from the root and prioritizes lattice edges over incoming shortcuts, and incoming shortcuts over outgoing shortcuts and breaking ties between lattice edges in some deterministic way (see Figure \ref{lattice order} for an example). Let the ordered and marked version of $G_n, G$, and $H_*$ be $\overline{G_n}$, $\overline{G}$ and $\overline{H}_*$ respectively, and let $(\overline{G_n},v) = \overline{H}_*$ denote the order equivalence similar to \eqref{order-equivalence}. For simplicity, when we write $B_r(G,v)\simeq H_*$ (or $B_r(\overline{G_n},v) = \overline{H}_*$ etc.), we will assume $\max_{u\in H_*} d_\Xi(M(u),M_*(\pi(u))) = |M(u),M_*(\pi(u))|_1 \leq \epsilon$ holds for some $\pi$ as in Definition \ref{local-convergence-def} in this section.

\subsection{Absence of Local Cycles and Collapse of Marks at Criticality}

Before proving the two main theorems, we isolate a basic estimate on the distribution of shortcut lengths in
\(
G_n:=K(n,q,k,\ell).
\)
This estimate implies that, when $\ell\le 2$, shortcuts do not appear in any `local' cycle with high probability. In particular, traversing shortcuts leads to fresh $(q,k)$-patch structures in the limit.

\begin{lemma}[Distribution of shortcut lengths]\label{shortcut-length-distribution}
Fix a node $u$ in $G_n$, and let $v$ be the endpoint of one of the $q$ outgoing shortcuts of $u$. Then
\[
\mathbb P[d(u,v)<d]
=
\frac{\sum_{\substack{w\neq u\\ d(u,w)<d}} d(u,w)^{-\ell}}
{\sum_{w\neq u} d(u,w)^{-\ell}}.
\]
Moreover, as $n\to\infty$,
\begin{itemize}
    \item when $\ell<2$,
    \[
    \mathbb P[d(u,v)<d]\to 0
    \qquad\text{for every }d=o(n),
    \]
    \item when $\ell=2$,
    \[
    \mathbb P[d(u,v)<d]\to 0
    \qquad\text{whenever }\frac{\log d}{\log n}\to 0.
    \]
    and more generally, if
    \[
    \frac{\log d}{\log n}\to c\in[0,1],
    \]
    then
    \[
    \frac{c}{4}
    \;\le\;
    \liminf_{n\to\infty}\mathbb P[d(u,v)<d]
    \;\le\;
    \limsup_{n\to\infty}\mathbb P[d(u,v)<d]
    \;\le\;
    4c.
    \]
\end{itemize}
\end{lemma}

\begin{proof}
Using a loose lower bound of $\sum_{j=1}^{n-1} j\cdot j^{-\ell}$ for \eqref{corner} and an upper bound of $\sum_{j=1}^{n-1} 4j\cdot j^{-\ell}$ for \eqref{center}, similarly to \cite{kleinberg}, we get
\begin{align*}
    \mathbb P[d(u,v)<d]
    &= \frac{\sum_{\substack{w\neq u\\ d(u,w)<d}} d(u,w)^{-\ell}}
    {\sum_{w\neq u} d(u,w)^{-\ell}}
    \le \frac{4\sum_{j=1}^{d-1} j^{1-\ell}}{\sum_{j=1}^{n-1} j^{1-\ell}}.
\end{align*}
Bounding the sums by integrals as in \cite{kleinberg}, we obtain
\begin{align*}
    \mathbb P[d(u,v)<d]
    &\le \frac{4+4\int_1^{d} j^{1-\ell}\,dj}{\int_1^{n+1} j^{1-\ell}\,dj}
    = \frac{4(2-\ell)+4(d^{2-\ell}-1)}{(n+1)^{2-\ell}-1}.
\end{align*}
Hence, when $\ell<2$ and $d=o(n)$, the last upper bound goes to zero.

For $\ell=2$, the same estimate gives
\begin{align*}
    \mathbb P[d(u,v)<d]
    &\le \frac{4+4\int_1^{d} j^{-1}\,dj}{\int_1^{n+1} j^{-1}\,dj}
    = \frac{4+4\log d}{\log(n+1)}.
\end{align*}
Thus, if $\log d/\log n\to 0$, then $\mathbb P[d(u,v)<d]\to 0$.
Similarly, using the corresponding lower bound
\(
\mathbb P[d(u,v)<d]
\ge
\frac{\sum_{j=1}^{d-1} j^{1-\ell}}{4\sum_{j=1}^{n-1} j^{1-\ell}},
\)
when $\ell=2$ we get
\begin{align*}
    \mathbb P[d(u,v)<d]
    &\ge \frac{\int_1^{d} j^{-1}\,dj}{4+4\int_1^n j^{-1}\,dj}
    = \frac{\log d}{4+4\log n}.
\end{align*}
Therefore, if $\log d/\log n\to c\in[0,1]$, then
\[
\frac{c}{4}
\;\le\;
\liminf_{n\to\infty}\mathbb P[d(u,v)<d]
\;\le\;
\limsup_{n\to\infty}\mathbb P[d(u,v)<d]
\;\le\;
4c.
\]
\end{proof}

The next corollary is the main structural consequence we will use later.

\begin{lemma}[No local shortcut cycles for $\ell\le 2$]\label{lem:no-local-shortcut-cycle}
Fix $r\ge 1$, and let $o_n$ be a uniformly chosen root in $G_n$. If $\ell\le 2$, then
\[
\mathbb P\bigl[B_r(G_n,o_n)\text{ contains a cycle with a shortcut edge}\bigr]\to 0.
\]
\end{lemma}

\begin{proof}
Fix $r\ge 1$, and let $D_r:=2kr$. Fix $\eta>0$. By Lemma~\ref{prelim-4},
\(
\sup_n \mathbb E|B_r(G_n,o_n)|<\infty.
\)
Hence, by Markov's inequality, there exists $M=M(r,\eta)$ such that
\(
\sup_n \mathbb P\bigl(|B_r(G_n,o_n)|>M\bigr)<\eta.
\)
Let
\(
E_n:=\{|B_r(G_n,o_n)|\le M\}.
\)
Work on the event $E_n$. Run the BFS exploration of $B_r(G_n,o_n)$, and for each stage $t$, let $S_t$ be the set of nodes exposed before stage $t$. Define
\[
N_{D_r}(S_t):=\{u\in V(G_n): \exists v\in S_t \text{ such that } d(u,v)\le D_r\}.
\]
Since $D_r$ is fixed and $|S_t|\le M$, there exists a constant $C=C(r,k,M)$ such that
\(
|N_{D_r}(S_t)|\le C
\)
for every possible $S_t$.
We claim that, if at every stage $t$ neither of the following occurs:
\begin{itemize}
    \item[(i)] an outgoing shortcut is exposed for the first time at stage $t$, its source lies in $S_t$, and its target lies in $N_{D_r}(S_t)$
    \item[(ii)] an incoming shortcut into some node of $S_t$ is exposed for the first time at stage $t$, and its source lies in $N_{D_r}(S_t)$,
\end{itemize}
then every shortcut first exposed by the BFS leads to a fresh patch, and therefore no cycle in $B_r(G_n,o_n)$ can contain a shortcut edge. Indeed, here, the shortcuts jump farther than $2kr$ in lattice distance from the previously exposed set. Since every lattice path of graph length at most $r$ moves by at most $kr$ in lattice distance, the radius-$r$ lattice neighborhood around the new endpoint is disjoint from the previously exposed part.  

It remains to bound the probability of (i) and (ii). For (i), there are at most $q|S_t|\le qM$ outgoing shortcuts whose source lies in $S_t$. For any $v\in S_t$, the probability that one of its outgoing shortcuts lands in $N_{D_r}(S_t)$ is at most
\[
q
\frac{\sum_{u\in N_{D_r}(S_t)} d(v,u)^{-\ell}}
{\sum_{w\neq v} d(v,w)^{-\ell}}
\le
\frac{qC}{\sum_{w\neq v} d(v,w)^{-\ell}}.
\]
By the same estimates as in the proof of Lemma~\ref{shortcut-length-distribution},
\(
\sum_{w\neq v} d(v,w)^{-\ell}
=
\Theta\!\left(\sum_{j=1}^{n-1} j^{1-\ell}\right)
\)
uniformly in $v$. Therefore,
\[
\mathbb P\bigl[(i)\text{ occurs at stage }t\bigr]
=
O\!\left(\frac{1}{\sum_{j=1}^{n-1} j^{1-\ell}}\right).
\]

For (ii), fix $u\in S_t$. An incoming shortcut into $u$ from a source in $N_{D_r}(S_t)$ can only originate from a node $v\in N_{D_r}(S_t)\setminus\{u\}$, and each such $v$ contributes at most $q$ outgoing shortcuts. Hence, we have
\begin{align*}
    \mathbb P\bigl[u\text{ receives such an incoming shortcut at stage }t\bigr]
    &\le
    \sum_{v\in N_{D_r}(S_t)\setminus\{u\}}
    q\cdot \frac{d(v,u)^{-\ell}}{\sum_{w\neq v} d(v,w)^{-\ell}}
    \\
    &\le
    \frac{qC}{\inf_{v\in V(G_n)}\sum_{w\neq v} d(v,w)^{-\ell}}.
\end{align*}
Using, again, the bound
\(
\sum_{w\neq v} d(v,w)^{-\ell}
=
\Theta\!\left(\sum_{j=1}^{n-1} j^{1-\ell}\right)
\)
uniformly in $v$, we obtain
\[
\mathbb P\bigl[u\text{ receives such an incoming shortcut at stage }t\bigr]
=
O\!\left(\frac{1}{\sum_{j=1}^{n-1} j^{1-\ell}}\right).
\]
Since $|S_t|\le M$, a union bound over $u\in S_t$ yields
\[
\mathbb P\bigl[(ii)\text{ occurs at stage }t\bigr]
=
O\!\left(\frac{1}{\sum_{j=1}^{n-1} j^{1-\ell}}\right).
\]
When $\ell<2$, this is $O(n^{-(2-\ell)})\to 0$, and when $\ell=2$, it is $O((\log n)^{-1})\to 0$.

Since on $E_n$ the BFS has at most $M$ stages, a union bound gives
\[
\mathbb P\bigl[(i)\text{ or }(ii)\text{ occurs at some stage} \,\bigm|\, E_n\bigr]\to 0.
\]
Putting everything together, this gives
\[
\mathbb P\bigl[B_r(G_n,o_n)\text{ contains a cycle with a shortcut edge}\bigr]
\le
\mathbb P(E_n^c)
+
\mathbb P\bigl[(i)\text{ or }(ii)\text{ occurs at some stage }t \,\bigm|\, E_n\bigr].
\]
and
\[
\limsup_{n\to\infty}
\mathbb P\bigl[B_r(G_n,o_n)\text{ contains a cycle with a shortcut edge}\bigr]
\le \eta.
\]
Since $\eta>0$ was arbitrary, the claim follows.
\end{proof}

Lemma~\ref{shortcut-length-distribution} and Lemma~\ref{lem:no-local-shortcut-cycle} have the following consequences.

\begin{remark}[Mark separation in the subcritical regime]\label{mark-change}
For $\ell<2$, whenever one traverses a shortcut, the mark (the coordinates normalized by $n$) changes with high probability. Indeed, Lemma~\ref{shortcut-length-distribution} shows that the endpoint of a shortcut is not within distance $o(n)$ of its starting point with high probability.
\end{remark}

Remark~\ref{mark-change} indicates that for $\ell<2$, the marks of the nodes at the two ends of each shortcut in a (reduced) $(q,k)$-patch differ with high probability (see Figure~\ref{q-k patch}). When $\ell=2$, on the other hand, the marks of two ends of a shortcut may remain the same or change (i.e. the distance between them could be sublinear or linear in $n$). For $\ell>2$, shortcuts link nearby nodes, so the endpoints no longer look like separate grids and, hence, every node in the local neighborhood of any node has the same mark. The local limit then remains close to the original model rather than simplifying to a structure like Figure~\ref{q-k patch}.

\section{Local Limit of the Kleinberg Model in Sub-critical Regime: Proof of Theorem~\ref{main-thm-kl-1}} \label{sec:proof-main-thm-kl-1}

To show that $(G_n, o_n)$ converges locally in probability to $(G, o)$, by the criterion given by Definition \ref{local-convergence-def}, it suffices to prove the following lemma.

\begin{lemma}[Local Convergence of the Kleinberg Model] \label{main-lemma-kl-1} For all $r > 0, \epsilon > 0$ and $H_* = B_r(H_*, o_*, M_*) \sim \mu$ that can be sampled with a positive probability, i.e. $\mu_r^{\epsilon}(H_*, o_*, M_*)>0$,
\begin{equation*} 
\frac{1}{n^2}\sum_{o_n\in V(G_n)}  \mathbbm{1}_r^\epsilon\big((G_n,o_n)\simeq H_* \big)\overset{\mathbb P}{\to}\mu\Big(  \mathbbm{1}_r^\epsilon\big((G,o)\simeq H_*)\big)\Big) \quad \text{ as $n \to \infty.$}
\end{equation*}
\end{lemma}

Again, we will use the Second Moment Method to formally prove convergence through the next two lemmas: the proof of the first lemma appears in this section, while the second lemma’s proof is deferred to Appendix~\ref{proof-kl-2moment}.

\begin{lemma}[First Moment of the Kleinberg Model] \label{kl-1moment} Fix any $r > 0$, $\epsilon > 0$, and $H_* = B_r(H_*, o_*, M_*)$ such that $\mu_r^{\epsilon}(H_*, o_*, M_*)>0$. Then,
\begin{equation*} 
\mu_{n,r}^\epsilon\big((G_n, o_n) \simeq H_*\big) \to \mu_{r}^\epsilon\big((G,o) \simeq H_*\big) \quad \text{as $n \to \infty$.}
\end{equation*}
\end{lemma}

\begin{lemma}[Second Moment of the Kleinberg Model] \label{kl-2moment}Fix any $r > 0$, $\epsilon > 0$, and $H_* = B_r(H_*, o_*, M_*)$ such that $\mu_r^{\epsilon}(H_*, o_*, M_*)>0$. Then,
\begin{equation*} 
\text{Var}\left(\frac{1}{n^2}\sum_{o_n\in V(G_n)}  \mathbbm{1}_r^\epsilon((G_n,o_n)\simeq H_*) \right) \to 0 \quad \text{as $n \to \infty$.}
\end{equation*}
\end{lemma}

\begin{proof}[Proof of Lemma~\ref{main-lemma-kl-1}]
    Given Lemma~\ref{kl-1moment} and \ref{kl-2moment}, the proof of  Lemma~\ref{main-lemma-kl-1} is identical to the proof of Lemma~\ref{main-lemma-ws}. 
\end{proof}

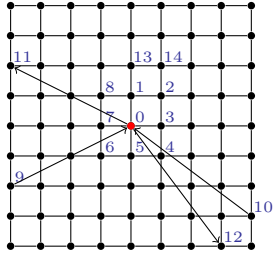
\begin{figure}
    \centering
\begin{tikzpicture}[scale=0.45]
\def\n{9} 
\foreach \i in {1,...,\n} {
    \foreach \j in {1,...,\n} {
        \node [sdot] (\i-\j) at (\i,-\j) {};
    }
}
\node [sdot, red] at (5, -5) {};
\node [blue-shortcut, thick] at (5.3, -4.7) {\tiny $0$};
\node [blue-shortcut, thick] at (5.3, -3.7) {\tiny $1$};
\node [blue-shortcut, thick] at (6.3, -3.7) {\tiny $2$};
\node [blue-shortcut, thick] at (6.3, -4.7) {\tiny $3$};
\node [blue-shortcut, thick] at (6.3, -5.7) {\tiny $4$};
\node [blue-shortcut, thick] at (5.3, -5.7) {\tiny $5$};
\node [blue-shortcut, thick] at (4.3, -5.7) {\tiny $6$};
\node [blue-shortcut, thick] at (4.3, -4.7) {\tiny $7$};
\node [blue-shortcut, thick] at (4.3, -3.7) {\tiny $8$};
\node [blue-shortcut, thick] at (1.3, -6.7) {\tiny $9$};
\node [blue-shortcut, thick] at (9.4, -7.7) {\tiny $10$};
\node [blue-shortcut, thick] at (1.4, -2.7) {\tiny $11$};
\node [blue-shortcut, thick] at (8.4, -8.7) {\tiny $12$};
\node [blue-shortcut, thick] at (5.4, -2.7) {\tiny $13$};
\node [blue-shortcut, thick] at (6.4, -2.7) {\tiny $14$};
\foreach \i in {1,...,\n} {
    \foreach \j [count = \k] in {2, ..., \n} {
            \draw (\i-\k) -- (\i-\j);
            \draw (\k-\i) -- (\j-\i);
        }}
\draw[->, black] (5-5) -- (8-9);
\draw[->, black] (5-5) -- (1-3);
\draw [->, black] (1-7) to (5-5);
\draw [->, black] (9-8) to (5-5);
\end{tikzpicture}
    \caption{Partial ordering of $K(9, 2, 1, \ell)$.}
    \label{lattice order}
\end{figure}

\subsection{Proof of Lemma~\ref{kl-1moment} via Coupling Marks}

Notice that the mark of each node in $G_n$ would be its integer coordinate normalized by $n$. Therefore, the marks would have a discrete distribution, while the marks in the limit have a continuous distribution over $[0, 1] \times [0, 1]$ as defined in Section~\ref{local-limit-kl-1}. We couple the continuous marks to the discrete version as follows: 

\subsection*{Coupling} \label{coupling-def} Divide $[0, 1] \times [0, 1]$ into $n^2$ equal $1/n \times 1/n$ \textit{small squares} and  construct the ordered graph $B_r(\hat{G}_n,\hat{o}_n)$ from $B_r(\overline{G}, o)$ so that its unmarked version is the same as $B_r(\overline{G}, o)$. Additionally the mark of $i$-th node in $\hat{G}_n$ is 
\begin{equation*}
    M(\hat{v}_{i, n}) = \left(\frac{\lfloor xn \rfloor}{n}, \frac{\lfloor yn \rfloor}{n}\right),
\end{equation*}
where $M(v_{i}) = (x, y)$ is the mark of $i$-th node in $\overline{G}$ (roots corresponds to the $0$-th node in both graphs). One can think of this as choosing $M(\hat{v}_{i, n}) $ to be the top left corner of the small square $M(v_{i})
$ falls to in $[0, 1] \times [0, 1]$. Let $\hat{\mu}_n$ be the marginal distribution of $\hat{G}_n$. 

The next lemma is the key result implying that the difference between the (discrete) distribution $\mu_n$  and $\hat{\mu}_n$ obtained from $\mu$ goes to $0$. 

\begin{lemma} \label{coupling} Fix $\epsilon > 0$. For any rooted graph $H_*$ such that $\mu_r^\epsilon((\overline{G},o) = \overline{H_*}) > 0$, 
\begin{equation*}
    \lim_{n\to\infty} \left| \mu_{r, n}^\epsilon((\overline{G_n},o_n) = \overline{H_*}) -\hat{\mu}_{r,n}^\epsilon((\hat{G}_n,\hat{o}_n) = \overline{H_*}) \right| = 0.
\end{equation*}
\end{lemma}

Assuming that Lemma~\ref{coupling} holds, we can complete the proof for Lemma~\ref {kl-1moment} using that the difference between the marginal distribution $\hat{\mu}_n$ and $\mu$ goes to $0$. This intuition is formalized next.
 
\begin{proof}[Proof of Lemma~\ref{kl-1moment}]
Fix some $\epsilon>0$. As in the case of the Watts-Strogatz model, for the first moment, it suffices to show that  $$\mu_{r, n}^\epsilon ((\overline{G_n}, o_n) = \overline{H_*}) \to \mu_r^\epsilon((\overline{G},o)= \overline{H_*})  \quad \text{ as $n \to \infty.$}$$
Let $M(\hat{v}_{i, n})$ and $M(v_{i})$ be the marks of $i$-the node of $\hat{G}_n$ and $G$. Notice that $\|M(\hat{v}_{i, n}) - M(v_{i})||_1 \leq \frac{2}{n}$ for all $i$ by coupling. We also have
\begin{align*}
    \mu_r^{\epsilon - \frac{2}{n}}((\overline{G},o)  = \overline{H_*}) 
 \leq \hat{\mu}_{r, n}^{\epsilon}((\hat{G}_n,\hat{o}_n)  = \overline{H_*}) \leq \mu_r^{\epsilon + \frac{2}{n}}((\overline{G},o)  = \overline{H_*}).
\end{align*}
Noting that $\mu_r^{\epsilon}((\overline{G}, o)  = \overline{H_*})$ is a continuous function of $\epsilon$, we get $$\lim_{n \to \infty} |\hat{\mu}_{r, n}^{\epsilon}((\hat{G}_n,\hat{o}_n)  = \overline{H_*}) - \mu_r^{\epsilon}((\overline{G},o)  = \overline{H_*})| = 0.$$
Together with Lemma~\ref{coupling}, $$\lim_{n\to\infty}\mu_{r, n}^\epsilon((\overline{G}_n,o_n)  = \overline{H_*}) =\lim_{n\to\infty}\hat{\mu}_{r,n}^\epsilon((\hat{G}_n,\hat{o}_n)  = \overline{H_*}) =\mu_r^{\epsilon}((\overline{G},o)  = \overline{H_*}),$$
concluding the proof for the convergence of the first moments. 
\end{proof}

Next, we prove Lemma~\ref{coupling}.

\subsection*{Proof of Lemma~\ref{coupling}}

Take any rooted finite graph $H_*$ such that
\(
\mu_r^\epsilon((\overline{G},o)=\overline{H_*})>0.
\)
Define $v_{i,n}$, $\hat v_{i,n}$, and $v_{i,*}$ to be the $i$-th node in
$\overline{G}_n$, $\hat G_n$, and $\overline{H_*}$, respectively. Consider the
following sets of marked, ordered, and rooted graphs isomorphic to $H_*$ such
that their node marks are within $\epsilon$-neighborhood of the corresponding
node marks in $\overline{H_*}$ with respect to the $L_1$ metric and they can be
sampled with positive probability according to $\mu_n$ and $\hat\mu_n$,
respectively:
\begin{align*}
I_n(\overline{H_*})=\Big\{\,B_r(\overline{G}_n,o_n)\ \Big|\ &
B_r(\overline{G}_n,o_n)=\overline{H_*},\
\max_{v_{i,n}\in V(\overline{G}_n)}
|M(v_{i,*})-M(v_{i,n})|_1<\epsilon \text{ and }
\mu_{r,n}^0((G_n,o_n))>0
\Big\}
\end{align*}
and
\begin{align*}
\hat I(\overline{H_*})=\Big\{\,B_r(\hat G_n,\hat o_n)\ \Big|\ &
B_r(\hat G_n,\hat o_n)=\overline{H_*},\
\max_{\hat v_{i,n}\in V(\hat G_n)}
|M(v_{i,*})-M(\hat v_{i,n})|_1<\epsilon \text{ and } \hat\mu_{r,n}^0((\hat G_n,\hat o_n))>0
\Big\}.
\end{align*}
To finish the proof of Lemma~\ref{coupling}, it suffices to show that
\begin{equation}\label{closed-forms}
\lim_{n\to\infty}\mu_{r,n}^0((\overline{G}_n,o_n)=F_n)
=
\lim_{n\to\infty}\hat\mu_{r,n}^0((\hat G_n,\hat o_n)=\hat F)
\end{equation}
for any $F_n\in I_n(\overline{H_*})$ and $\phi(F_n)=\hat F\in\hat I(\overline{H_*})$ due to the following remark.
\begin{remark}\label{dif-lattice-con}
For any $\epsilon>0$ and all $n$ sufficiently large, there is an injection $\Phi: I_n(\overline{H_*}) \hookrightarrow \hat I(\overline{H_*})$, where the only difference between $F_n$ and $\Phi(F_n)=\hat F$ is that the marks of the ends of lattice edges differ by $1/n$ in $F_n$ and coincide in $\hat F$. The complement of the image is contained in the set of configurations whose root or some shortcut endpoint has mark $(x,y)$ with $\min\{x,y\}<\frac{rk}{n}$ or $\max\{x,y\}\ge\frac{n-rk}{n}$, which has $\hat\mu_{r,n}^0$-mass $O\left(\frac{rk}{n}\right)=o(1)$ by a standard boundary estimate, and hence does not affect \eqref{closed-forms}.
\end{remark}
To do so, we write a closed form for
$\hat\mu_{r,n}^0((\hat G_n,\hat o_n)=\hat F)$ and then for
$\mu_{r,n}^0((\overline{G}_n,o_n)=F_n)$, and finally show that they converge to
one another.
We first introduce some notation. Given $\overline{H_*}$, define $p(i)$ to be the order of the parent of the $i$-th node in $H_*$ (the node from which it is explored in the BFS process), and let $x_i$ be the number of incoming shortcuts of the $i$-th node in $\overline{H_*}$. Also define
\[
\mathbbm{1}(i,\mathrm{out})
=
\mathbbm{1}(v_{i,*}\text{ is explored through an outgoing shortcut})
\]
and
\[
\mathbbm{1}(i,\mathrm{in})
=
\mathbbm{1}(v_{i,*}\text{ is explored through an incoming shortcut}).
\]

The key point is that in the coupled object $\hat G_n$, only the roots of the different $(q,k)$-patches carry independent mark choices. Nodes reached from a patch root using only lattice edges do not contribute new independent $1/n^2$-type factors: their marks are determined by the mark of the patch root together with the deterministic lattice offset.

Accordingly, let $R(\hat F)$ be the set of patch roots of $\hat F$, namely the original root together with all nodes of $\hat F$ that are first reached
through a shortcut. For each $\rho\in \mathbf R(\hat F)$, let
\[
\mathcal M_{\hat F}(\rho)
=
\big(\mathcal M_{\hat F}(\rho)_1,\mathcal M_{\hat F}(\rho)_2\big)\in[0,1]^2
\]
be its mark, and define
\[
Q_\rho
:=
\big[\mathcal M_{\hat F}(\rho)_1,\ \mathcal M_{\hat F}(\rho)_1+\tfrac1n\big]
\times
\big[\mathcal M_{\hat F}(\rho)_2,\ \mathcal M_{\hat F}(\rho)_2+\tfrac1n\big].
\]
For each node $i$, let $\rho(i)\in \mathbf R(\hat F)$ denote the root of the patch containing node $i$. Equivalently, $\rho(0)=0$; if node $i$ is reached through a shortcut, then $\rho(i)=i$; and if node $i$ is reached through a lattice edge, then $\rho(i)=\rho(p(i))$.

With this notation, the right-hand side of \eqref{closed-forms} equals
\begin{equation}\label{eq:nested-corrected}
\int_{\prod_{\rho\in \mathbf R(\hat F)}Q_\rho}
f\big((m_\rho)_{\rho\in \mathbf R(\hat F)}\big)\,
\prod_{\rho\in \mathbf R(\hat F)}dm_\rho,
\end{equation}
where
\[
f\big((m_\rho)_{\rho\in \mathbf R(\hat F)}\big)
=
\prod_{i=0}^{|H_*|-1}
\frac{e^{-\Lambda_{m_{\rho(i)}}}\Lambda_{m_{\rho(i)}}^{x_i}}{x_i!}
\prod_{i:\,\mathbbm{1}(i,\mathrm{out})=1}
p_{\mathrm{out},\,m_{\rho(p(i))}}\!\big(m_{\rho(i)}\big)
\prod_{i:\,\mathbbm{1}(i,\mathrm{in})=1}
p_{\mathrm{in},\,m_{\rho(p(i))}}\!\big(m_{\rho(i)}\big).
\]

Each $Q_\rho$ has area $n^{-2}$ and diameter at most $2/n$. Since all factors
above are continuous in their $[0,1]^2$ arguments, $f$ is uniformly continuous
on the compact set $\prod_{\rho\in \mathbf R(\hat F)}Q_\rho$. Let
\[
\omega_n
:=
\sup\Big\{
|f(\mathbf z)-f(\mathbf z')|:\ 
\|\mathbf z-\mathbf z'\|_\infty\le 2/n
\Big\}.
\]
Then $\omega_n\to0$ as $n\to\infty$, and iterating the one-cell estimate over
the finitely many patch roots yields
\begin{align*}
&\int_{\prod_{\rho\in \mathbf R(\hat F)}Q_\rho}
f\big((m_\rho)_{\rho\in \mathbf R(\hat F)}\big)\,
\prod_{\rho\in \mathbf R(\hat F)}dm_\rho =
\frac{1}{n^{2|\mathbf R(\hat F)|}}
\Big(
f\big((\mathcal M_{\hat F}(\rho))_{\rho\in \mathbf R(\hat F)}\big)
+O(\omega_n)
\Big).
\end{align*}
Therefore
\begin{equation}\label{closed-form-right}
\hat\mu_{r,n}^0((\hat G_n,\hat o_n)=\hat F)
=
\frac{1}{n^{2|\mathbf R(\hat F)|}}
\Big(
f\big((\mathcal M_{\hat F}(\rho))_{\rho\in \mathbf R(\hat F)}\big)
+o(1)
\Big).
\end{equation}

The closed form of $\mu_{r,n}^0((\overline{G}_n,o_n)=F_n)$ is more complicated,
since we no longer have independence and need to work with conditional
probabilities. Define $T_i$ to be the event that $\overline{G}_n$ and $F_n$
agree up to the $(i+1)$-th step of the BFS process as before, and let
$\mathcal M_{F_n}(i)$ be the mark of the $i$-th node in $F_n$. Then the
left-hand side of \eqref{closed-forms} is equal to the limit of
\begin{equation}\label{closed-form-left}
\prod_{i=0}^{|H_*|-1}
\mathbb P\big[\text{$v_{i,n}$ has $x_i$ incoming shortcuts}\mid
M(v_{i,n})=\mathcal M_{F_n}(i),\,T_{i-1}\big]\,
\mathbb P\big[M(v_{i,n})=\mathcal M_{F_n}(i)\mid T_{i-1}\big].
\end{equation}

Notice that each patch root contributes one $n^{-2}$-type factor: the original
root because $o_n$ is chosen uniformly at random, and each shortcut-created
patch root because its mark is specified by one $1/n\times1/n$ cell. By
contrast, nodes reached only through lattice edges contribute no new
$n^{-2}$-factor, since their marks are determined by the mark of their patch
root up to the deterministic lattice offset.

What remains to show for Lemma~\ref{coupling} is that the limit of
\eqref{closed-form-left} is the same as \eqref{closed-form-right}. Since
$\overline{H_*}$ is finite, it is enough to prove the following two key lemmas whose proof appears in Appendix~\ref{proof-kl-1moment}. At a high level, the proof of the first lemma follows from the bijection established in Remark~\ref{dif-lattice-con} and from the fact that the \(1/n\) offset between the shortcut marks in \(\overline{G}_n\) and \(F_n\) tends to \(0\). Consequently, the discrepancy between the corresponding distributions of shortcut marks also vanishes.

\begin{lemma}\label{mark-conv}
Let $M(v_{i,n})$ and $\mathcal M_{F_n}(i)$ be the marks of the $i$-th node in
$\overline{G}_n$ and $F_n$, respectively, and let $T_i$ be the event that
$\overline{G}_n$ and $F_n$ agree up to the $(i+1)$-th step of the BFS process.
Define
\[
Q_i := \big[\mathcal{M}_{F_n}(i)_1,\ \mathcal{M}_{F_n}(i)_1+\tfrac{1}{n}\big]
       \times
       \big[\mathcal{M}_{F_n}(i)_2,\ \mathcal{M}_{F_n}(i)_2+\tfrac{1}{n}\big].
\]

Then,
    \[
    \mathbb P\big[M(v_{i,n})=\mathcal M_{F_n}(i)\mid T_{i-1}\big]=1
    \]
    if $v_{i,n}$ is explored through a lattice connection;
    \[
    \mathbb P\big[M(v_{i,n})=\mathcal M_{F_n}(i)\mid T_{i-1}\big]
    =
    \int_{Q_i} p_{\mathrm{out},\,\mathcal M_{F_n}(p(i))}(z)\,dz
    + o(n^{-2})
    \]
    if $v_{i,n}$ is explored through an outgoing shortcut; and
    \[
    \mathbb P\big[M(v_{i,n})=\mathcal M_{F_n}(i)\mid T_{i-1}\big]
    =
    \int_{Q_i} p_{\mathrm{in},\,\mathcal M_{F_n}(p(i))}(z)\,dz
    + o(n^{-2})
    \]
    if $v_{i,n}$ is explored through an incoming shortcut.
\end{lemma}

Intuitively, the next lemma follows from the same Poisson approximation as Lemma~\ref{lm-degdist-WS}, except that here the parameter depends on the node’s location in the grid.
\begin{lemma}\label{incom-pois-conv}
Take $M(v_{i,n})$, $\mathcal M_{F_n}(i)$, and $T_i$ as in
Lemma~\ref{mark-conv}. Then
\begin{equation}\label{lemma-2}
\lim_{n\to\infty}
\mathbb P\big[\text{$v_{i,n}$ has $x_i$ incoming shortcuts}\mid
M(v_{i,n})=\mathcal M_{F_n}(i)\text{ and }T_{i-1}\big]
=
\lim_{n\to\infty}
\frac{e^{-\Lambda_{\mathcal M_{F_n}(i)}}\Lambda_{\mathcal M_{F_n}(i)}^{x_i}}{x_i!}.
\end{equation}
\end{lemma}

Now, using Lemmas~\ref{mark-conv} and \ref{incom-pois-conv}, we finish the proof
of Lemma~\ref{coupling}. By Remark~\ref{dif-lattice-con}, if
$\mathcal M_{F_n}(i)\neq \mathcal M_F(i)$, then $v_{i,*}$ is explored through a
lattice connection. Since we explore only nodes at graph distance at most $rk$
from the root, we have
\[
|\mathcal M_{F_n}(i)-\mathcal M_F(i)|_1\le \frac{rk}{n}\to 0.
\]
The densities $p_{\mathrm{out},\,\mathcal M_F(p(i))}$ and
$p_{\mathrm{in},\,\mathcal M_F(p(i))}$ are continuous on a neighborhood of
$\mathcal M_F(i)$. Because $\mathcal M_{F_n}(i)\to \mathcal M_F(i)$ and
$\mathcal M_{F_n}(p(i))\to \mathcal M_F(p(i))$, for all large $n$ we have
\[
\int_{Q_i} p_{\mathrm{out},\,\mathcal M_{F_n}(p(i))}(z)\,dz
=
n^{-2}\bigl(p_{\mathrm{out},\,\mathcal M_F(p(i))}(\mathcal M_F(i))+o(1)\bigr),
\]
\[
\int_{Q_i} p_{\mathrm{in},\,\mathcal M_{F_n}(p(i))}(z)\,dz
=
n^{-2}\bigl(p_{\mathrm{in},\,\mathcal M_F(p(i))}(\mathcal M_F(i))+o(1)\bigr).
\]
Also, by continuity of $\Lambda_m$ in $m$,
\[
\frac{e^{-\Lambda_{\mathcal M_{F_n}(i)}}\Lambda_{\mathcal M_{F_n}(i)}^{x_i}}{x_i!}
\to
\frac{e^{-\Lambda_{\mathcal M_F(i)}}\Lambda_{\mathcal M_F(i)}^{x_i}}{x_i!}.
\]
Thus the limit of \eqref{closed-form-left} agrees with
\eqref{closed-form-right}, concluding the proof. \qed

\section{Local Limit of the Kleinberg Model in Super-critical Regime: Proof of Theorem~\ref{main-thm-kl-2}} \label{proof-main-thm-kl-2}

Let $(G, o) \sim \mu$ be the random graph sampled according to $K_>(q, k, \ell)$  with measure $\mu$ defined in Section~\ref{local-limit-kl-2}, and let $(G_n, o_n) = (K(n, q, k, \ell) \sim \mu_n$, $V(G_n)$, $\mu_{r,n}$ and $\mu_r$ be defined as in the proof of Theorem~\ref{main-thm-kl-1}. (Notice that we omitted $\epsilon$ here since the limit graph $G$ is unmarked, so we take $G_n$ as an unmarked graph, and mark equivalences are not considered.) Again, as in the proof of Lemma~\ref{main-lemma-kl-1}, it suffices to prove 
\begin{equation} \label{main-lemma-kl-2} 
\frac{1}{n^2}\sum_{o_n\in V(G_n)}  \mathbbm{1}_r\big((G_n,o_n)\simeq H_* \big)\overset{\mathbb P}{\to}\mu\Big(  \mathbbm{1}_r\big((G,o)\simeq H_*)\big)\Big) \quad \text{ as $n \to \infty.$}
\end{equation}
for all $r > 0$ and $H_* = B_r(H_*, o_*)$ such that $\mu_r(H_*, o_*)>0$.

We will start with some definitions and preliminaries. 
Recall that $d(u, v)$ is the lattice distance, i.e. $d(u, v) = |x_u - x_v| + |y_u - y_v|$, and $d_G(u, v)$ is the graph distance between the nodes in graph $G$.
Armed with these two definitions, we can work with two types of local neighborhoods: graph neighborhoods given by $B_r(G, o)$ for a graph $G$ with root $o$ and lattice neighborhoods defined next.

\begin{definition} [Lattice Neighborhood] Let $L_r(G, o)$ denote the subgraph of $(G, o)$ induced by all the nodes $v \in V(G)$ whose lattice distance to $o$ is less than or equal to $r$: $$L_r(G, o) = G[\{v \in V(G) | d(v, o) \leq r\}].$$
\end{definition}

The following lemma shows that all shortcuts of a node span a short range, explaining why, in the limit, traversing through shortcuts does not change the marks.

\begin{lemma} \label{incoming shortcut} Let $\psi:\mathbb{N}\to\mathbb R$ be any such function that $\lim_{n \to \infty} \psi(n) \to \infty$ and define  $G_n = K(n, q, k, \ell)$, where $\ell>2$.Then, for a given node $u \in V(G_n)$,
$$\mathbb{P}[\text{There exists a shortcut between $u$ and a node $v$ s.t. } d(u, v) > \psi(n)] \to 0\qquad \text{ as }n \to \infty.$$
\end{lemma} 

\begin{proof} 
First, we bound the probability that there exists an \textit{incoming} shortcut of distance at most $\psi(n)$. Let $S_u$ be the event that there exists an incoming shortcut to $u$ from a node $v$ such that $d(u, v) > \psi(n)$. There are at most $4j$ nodes at lattice distance $j$ from $u$ and each has $q$ outgoing shortcuts. Since $\ell>2$, there exists $c_\ell>0$ such that $D_n(v)\ge c_\ell $uniformly in $v$ and $n$ large enough. Therefore
\[
\mathbb P[S_u]
\le \frac{4q}{c_\ell}\sum_{j=\psi(n)+1}^{2n-2} j^{1-\ell},
\] using union bound over $4qj$ potential edges,
\begin{align*}
\mathbb P[S_u]
 \le
\frac{4q}{c_\ell}\int_{\psi(n)}^{2n-1} x^{1-\ell}\,dx =
\frac{4q}{c_\ell(\ell-2)}
\bigl(\psi(n)^{2-\ell}-(2n-1)^{2-\ell}\bigr) \to 0
\qquad \text{as } n\to\infty.
\end{align*}

Now turning our attention to outgoing shortcuts, we have the following inequality from \cite[Proof of Theorem 3b]{kleinberg}: Let $u$ and $v$ be the two ends of a shortcut. Then,
\begin{equation} \label{klein}
    \mathbb{P}[d(u, v) > d] \leq (\ell -2)^{-1} d^{2-\ell}.
\end{equation}
This gives that 
\begin{equation*} \label{klein2}\mathbb{P}[\text{an outgoing shortcut of $u$ is to a node $v$ s.t. } d(u, v) > \psi(n)] \leq q (\ell -2)^{-1} \psi(n) ^{2-\ell} \to 0\end{equation*} for any function $\psi(n) \to \infty$, proving the desired result.
\end{proof}
Lemma~\ref{incoming shortcut} shows that the shortcuts of a given node are contained within the $o(n)$ lattice neighborhood of it. Therefore, the marks of all nodes in  $B_r(G_n, o_n)$––the coordinates normalized by $n$––will have the same mark as the mark of node $o_n$ for any finite $r$ in the limit. 

We will later show a stronger result:  for any $r$, there exists a constant $R$ such that, with high probability, the $R$-lattice neighborhood contains the $r$-graph neighborhood, i.e., $B_r(G_n, v) \subseteq L_R(G_n, v)$. This result is particularly useful because the $R$-lattice neighborhoods of two nodes in $G_n$ are isomorphic if their distance to the boundary of the grid is larger than $R$.  Consequently, for two uniformly chosen nodes $o_1$ and $o_2$ in $G_n$, their $R$-lattice neighborhoods are isomorphic with high probability, as the probability of sampling a node within the $R$ distance of the boundary tends to zero. As a result, the neighborhood structures of $B_r(G_n, o_1)$ and $B_r(G_n, o_2)$ have the same distribution (ignoring the marks).
This conveniently eliminates the need for marks in the limiting graph. With this, we are ready to prove \eqref{main-lemma-kl-2} formally by proving the following two steps:
\begin{itemize}
     \item The lattice neighborhood of $K(n, q, k, \ell)$ converges in probability to the lattice neighborhood of $K(q, k, \ell)$. We will use a Second-Moment argument here similar to the prior proofs.
     \item The lattice neighborhood convergence implies graph neighborhood convergence. For this purpose, we will use a careful path counting and coupling argument here inspired by \cite{van2021local}.
\end{itemize}
To do so, we introduce the following notation:
\begin{align*}
    \mathbbm{1}_r^L\big((G_n,o_n)\simeq (H_*,o_*)\big)=  \mathbbm{1}\Big( L_r(G_n,o_n)\simeq (H_*,o_*)\Big).
\end{align*}
Similar to the definition of expectation over $\mu$ and $\mu_n$, we will define $\mu_r^L((G,o) \simeq  (H_*,o_*))$ to be the probability of the event $L_r(G,o)\simeq (H_*,o_*)$ under measure $\mu$ and $\mu_{n,r}^L((G_n,o_n) \simeq  (H_*,o_*))$ to be the probability of the event $L_r(G_n,o_n)\simeq (H_*,o_*)$ under measure $\mu_n$.

\subsection{Local Lattice Convergence}

\begin{lemma} [Lattice Neighborhood Convergence of the Kleinberg Model] \label{local-lattice-cv} Let \(r > 0\) be a constant. 
Then for any rooted graph $(H_*,o_*)$,
 \begin{equation*} 
    \frac{1}{n^2}\sum_{o_n\in V(G_n)}  \mathbbm{1}_r^L((G_n,o_n) \simeq (H_*,o_*) )\overset{\mathbb P}{\to}\mu_r^L((G,o) \simeq (H_*,o_*)) \quad \text{ as $n \to \infty.$}
\end{equation*} 
\end{lemma}

Again, we will use the Second Moment Method to prove the convergence formally via the next two lemmas, with their proofs provided in Appendix \ref{moment-kl}. 

\begin{lemma}[First Moment of Lattice Convergence] \label{lattice-cv-1} 

Let \(r > 0\) be a constant. For any  rooted graph $(H_*,o_*)$,
\begin{equation*} 
\mu_{r, n}^L((G_n,o_n) \simeq (H_*,o_*)) \to \mu_{r}^L((G,o) \simeq (H_*,o_*)) \quad \text{ as $n \to \infty.$}
\end{equation*}
\end{lemma}

\begin{lemma}[Second Moment of Lattice Convergence] \label{lattice-cv-2} Let \(r > 0\) be a constant.  For any given rooted graph $(H_*,o_*)$,
\begin{equation*} 
\text{Var}\left(\frac{1}{n^2}\sum_{o_n\in V(G_n)}  \mathbbm{1}_r^L((G_n,o_n)\simeq (H_*,o_*)) \right) \to 0 \quad \text{ as $n \to \infty.$}
\end{equation*}
\end{lemma}

\begin{proof}
    Given Lemma~\ref{lattice-cv-1} and \ref{lattice-cv-2}, the proof of  Lemma~\ref{local-lattice-cv} is identical to the proof of Lemma~\ref{main-lemma-ws}. 
\end{proof}

\subsection{Lattice Neighborhood Convergence implies Graph Neighborhood Convergence}\label{sec: lattic-to-graph}
It is enough to demonstrate that for any constant \(r > 0\), there exists a corresponding constant \(R > 0\) such that the lattice neighborhood of radius \(R\) fully contains the graph neighborhood of radius \(r\). The main idea behind the proof is to show that any path of a given length remains confined within a bounded distance in the lattice. To establish this result, we adapt a technical lemma from \cite{van2021local}, originally developed to analyze local limits of geometric random graphs. 

For this purpose, we need the following notation. Given a path of shortcuts  $\mathcal{U}_n = (o_n = u_0, u_1, u_2, \cdots, u_{j-1})$ in $G_n$, and  a path $\mathcal{U} = (o = u_0, u_1, u_2, \cdots, u_{j-1})$ in $G$, let $P(G_n, \mathcal{U}_n)$ and $P(G, \mathcal{U})$ be the probabilities that path $\mathcal{U}_n$ exists in graph $G_n$ and path $\mathcal{U}$ exists in $G$, respectively. Also define
\begin{align*}
    \mathcal{I}_{n,j} :&= \{  \mathcal{U}_n =  (u_1, \dots, u_j) | (u_0, u_1, \dots, u_j) \text{ is a } j\text{-path in } (G_n, o_n) \text{ starting from } u_0 = o_n, \\
     &\forall 0\leq i \leq j-2, d(u_i, u_{i+1}) \leq a^{b^{i+1}}, d(u_{j-1}, u_{j}) > a^{b^{j}}\}
\end{align*}
 and $\mathcal{I}_{j}$ is the corresponding set in $(G, o)$:
 \begin{align*}
    \mathcal{I}_{j} :&= \{  \mathcal{U} =  (u_1, \dots, u_j) | (u_0, u_1, \dots, u_j) \text{ is a } j\text{-path in } (G, o) \text{ starting from } u_0 = o, \\
     &\forall 0\leq i \leq j-2, d(u_i, u_{i+1}) \leq a^{b^{i+1}}, d(u_{j-1}, u_{j}) > a^{b^{j}}\}.
\end{align*}

\begin{lemma} [Path Counting] \label{path counting}

Given $\ell>2$, let $G_n=K(n,q,k,\ell)$ and $(G,o)\sim K_>(q,k,\ell)$, for any $j \geq 1$, and $a > 1$,
\begin{align}
\lim_{b \to \infty} \limsup_{n \to \infty} \frac{1}{n^2} \sum_{\mathcal{U}_n \in \mathcal{I}_{n,j}} P(G_n, \mathcal{U}_n) \left( \prod_{i=0}^{j-2} \mathbbm{1} \left( d( u_{i+1}, u_i) \leq a^{b^{i+1}} \right) \right)
\mathbbm{1} \left( d(u_{j-1}, u_j, ) > a^{b^j} \right) = 0,
\end{align}
and
\begin{align} \label{second identity}
\lim_{b \to \infty}   \sum_{\mathcal{U} \in \mathcal{I}_{j}} P(G, \mathcal{U}) \left( \prod_{i=0}^{j-2} \mathbbm{1} \left( d(u_{i+1}, u_i)  \leq a^{b^{i+1}} \right) \right)
\mathbbm{1} \left( d(u_{j-1}, u_j) > a^{b^j} \right) = 0.
\end{align}
\end{lemma}

As a corollary of this lemma, we will show for any fixed $K \in \mathbb{N}$ and $a,b > 1$, if we choose
$$r := r(a,b,K) = a^b + a^{b^2} + a^{b^3} + \cdots + a^{b^K}$$
then the $K$-neighborhood of the rooted graph $L_r(G_n, o_n)$, and  $B_K(G_n, o_n)$ will be the same in the limit as first $n$ and then $b$ tends to infinity, and the same result would hold for $(G, o)$. This is formalized in the following lemma, whose proof appears in the Appendix~\ref{moment-kl} together with the proof of Lemma~\ref{path counting}. 

\begin{lemma}  [Coupling Lattice Neighborhood with Graph Neighborhood] \label{lattice-to-graph} Take any $K > 0$ and $a, b > 1$, and let $r := r(a,b,K) = a^b + a^{b^2} + a^{b^3} + \cdots + a^{b^K}$. Then,
\begin{equation}
    \lim_{b \to \infty} \limsup_{n \to \infty} \mu_{n, K}((L_r(G_n, o_n), o_n) \not\simeq B_K(G_n, o_n)) = 0,
\end{equation}
and
\begin{equation}
    \lim_{b \to \infty} \mu_K((L_r(G, o), o) \not\simeq B_K(G, o)) = 0.
\end{equation}    
\end{lemma}
Now, using Lemma~\ref{local-lattice-cv} and Lemma~\ref{lattice-to-graph}, we can prove the main local convergence result.

\begin{proof}[Proof of Theorem~\ref{main-thm-kl-2}] 
The proof closely follows the approach used for \cite[Theorem 1.11]{van2021local}. We provide a brief sketch:
Fix some $a, b, K > 1$ and let $r = r(a,b,K) = a^b + a^{b^2} + a^{b^3} + \cdots + a^{b^K}$. Note that $$\frac{1}{n^2}\sum_{o_n\in V(G_n)}  \mathbbm{1}_K ((G_n,o_n) \simeq H_*)$$ can be rewritten as 
\begin{equation}
\begin{split}
    & \frac{1}{n^2}\sum_{o_n\in V(G_n)} \mathbbm{1}_K((G_n,o_n) \simeq H_* ) \mathbbm{1}_K((L_r(G_n, o_n),o_n) \simeq B_K(G_n, o_n))  \label{firstline} \\
    & + \frac{1}{n^2}\sum_{o_n\in V(G_n)}\mathbbm{1}_K((G_n,o_n) \simeq H_*) \mathbbm{1}_K((L_r(G_n, o_n),o_n) \not\simeq B_K(G_n, o_n)). 
\end{split}
\end{equation}

Notice that the first term in \eqref{firstline} is
$$\frac{1}{n^2}\sum_{o_n\in V(G_n)} \mathbbm{1}_K((L_r(G_n, o_n),o_n) \simeq H_*) - \mathbbm{1}_K((L_r(G_n, o_n),o_n) \simeq H_*)  \mathbbm{1}_K((L_r(G_n, o_n),o_n) \not\simeq B_K(G_n, o_n)).$$
Let $\epsilon_{r, n}$ be 
\begin{equation} \label{eps-def}
    \frac{1}{n^2}\sum_{o_n\in V(G_n)} \mathbbm{1}_K((L_r(G_n, o_n),o_n) \not\simeq B_K(G_n, o_n))\left[\mathbbm{1}_K((G_n,o_n) \simeq H_*)  - \mathbbm{1}_K((L_r(G_n, o_n),o_n) \simeq H_*)\right]
\end{equation}
giving
$$\frac{1}{n^2}\sum_{o_n\in V(G_n)}  \mathbbm{1}_K ((G_n,o_n) \simeq H_*) = \epsilon_{r, n} + \frac{1}{n^2}\sum_{o_n\in V(G_n)}\mathbbm{1}_K((L_r(G_n, o_n),o_n) \simeq H_*).$$

We know that $$ 0 \leq \lim_{b \to \infty} \lim_{n \to \infty} \mathbb{E}[|\epsilon_{n, K}|] \leq \lim_{b \to \infty} \lim_{n \to \infty} 2 \mu_{n, K}((L_r(G_n, o_n), o_n) \not\simeq B_K(G_n, o_n)) = 0$$
by \eqref{eps-def} and Lemma~\ref{lattice-to-graph}. Then, for any $\epsilon > 0$, $K > 0$, and $H_* = B_K(H_*, o_*)$, there exists large enough $b$ and $n$ such that
 $$ \mathbb{P} \left[ \left| \frac{1}{n^2} \sum_{o_n\in V(G_n)}  \mathbbm{1}_K ((G_n,o_n) \simeq H_*) -\mathbbm{1}_K((L_r(G_n, o_n),o_n) \simeq H_*)\right| > \epsilon \right] = \mathbb{P} [|\epsilon_{n, K}| > \epsilon] < \epsilon.$$
 
This gives $\frac{1}{n^2} \sum_{o_n\in V(G_n)}  \mathbbm{1}_K ((G_n,o_n) \simeq H_*) \overset{\mathbb P}{\to} \frac{1}{n^2} \sum_{o_n\in V(G_n)} \mathbbm{1}_K((L_r(G_n, o_n),o_n) \simeq H_*)$. By a similar argument, for any $\epsilon > 0$, $K > 0$, and $H_* = B_K(H_*, o_*)$, there exists large enough $b$ such that
$$\mathbb{P} \left[ \left| \mu_K ((G,o) \simeq H_*) -\mu_K((L_r(G, o),o) \simeq H_*)\right| > \epsilon \right] < \epsilon$$ and $\mu_K ((G,o) \simeq H_*) \to  \mu_K((L_r(G, o),o) \simeq H_*)$. Together with Lemma~\ref{local-lattice-cv}, this completes the proof \eqref{main-lemma-kl-2} and Theorem~\ref{main-thm-kl-2}. 
\end{proof}

\section{Local Limit of the Kleinberg Model at Criticality: Proof of Theorem~\ref{main-thm-kl-critical}} \label{proof-main-thm-kl-critical}

Since the critical limit is unmarked, our main convergence statement is for the underlying unmarked rooted graph. The role of the marks is only auxiliary, and we will show below that they collapse in every fixed-radius neighborhood. As before, it is enough to prove that for every \(r\ge 1\) and every finite rooted graph \(H_* = B_r(H_*,o_*)\) with \(\mu_r(H_*,o_*)>0\),
\begin{equation*}
\frac{1}{n^2}\sum_{u\in V(G_n)} \mathbbm{1}_r\big[(G_n,u)\simeq H_*\big]
\overset{\mathbb P}{\to}
\mu\Big(\mathbbm{1}_r\big[(G,o)\simeq H_*\big]\Big)
\qquad \text{as } n\to\infty.
\end{equation*}
As before, the proof will follow first and second moment bounds. Since the argument is parallel to the proofs in the subcritical and supercritical regimes, we defer the full details to Appendix~\ref{proof-main-thm-kl-critical-appendix} and include here only the main ingredients that explain why the critical case simplifies.

For \(u\in V(G_n)\), write
\[
D_n(u):=\sum_{v\neq u} d(u,v)^{-2}.
\]
This quantity is the normalization constant for the shortcut distribution in the finite graph: a shortcut from \(u\) lands at \(v\) with probability
\(
\frac{d(u,v)^{-2}}{D_n(u)}.
\)
It is therefore the critical analogue of the normalizing terms that appear in the other two regimes: For \(\ell<2\), the corresponding marked limit is governed by the kernel \(\|m-m'\|_1^{-\ell}\), normalized by
\(
\int_{[0,1]^2}\|m-x\|_1^{-\ell}\,dx,
\)
while for \(\ell>2\) the normalization converges to the finite constant \(4\zeta(\ell-1)\), yielding the limiting kernel
\(
\kappa(u,v)=\frac{d(u,v)^{-\ell}}{4\zeta(\ell-1)}.
\)
At criticality, however, this normalization diverges with \(n\).

The key new fact at the critical point \(\ell=2\) is that, although \(D_n(u)\) diverges, it does so at the same rate for every node that is not too close to the boundary.

\begin{lemma}[Uniform normalization at criticality]\label{lem:uniform-normalization}
For every \(u\in V(G_n)\),
\[
D_n(u)\ge \log n+O(1).
\]
Moreover, if \(u\in V(G_n)\) has lattice distance at least
\[
d_n:=\left\lfloor \frac{n}{\log n}\right\rfloor
\]
from the boundary of the \(n\times n\) grid, then
\[
D_n(u)=4\log n+O(\log\log n),
\]
uniformly in \(u\).
\end{lemma}

\begin{proof}
For the first claim, the lower bound used in the proof of Lemma~\ref{shortcut-length-distribution} gives
\[
D_n(u)>
\sum_{j=1}^{n-1} j^{1-2}
=\sum_{j=1}^{n-1}\frac1j
=\log n+O(1),
\]
uniformly in \(u\). If \(u\) is at lattice distance at least \(d_n\) from the boundary, then for every \(1\le j\le d_n\) there are exactly \(4j\) nodes at lattice distance \(j\) from \(u\). Hence
\[
D_n(u)\ge \sum_{j=1}^{d_n} 4j\cdot j^{-2}
=4\sum_{j=1}^{d_n}\frac1j
=4\log d_n+O(1)
=4\log n-O(\log\log n).
\]
On the other hand, for every \(u\), there are at most \(4j\) nodes at lattice distance \(j\), so
\[
D_n(u)\le \sum_{j=1}^{n-1} 4j\cdot j^{-2}
=4\sum_{j=1}^{n-1}\frac1j
=4\log n+O(1).
\]
Combining the lower and upper bounds yields the claim.
\end{proof}

Let \(G_{n-d_n}\) denote the induced subgraph of the \(n\times n\) grid on the nodes whose lattice distance from the boundary is at least
\(
d_n:=\left\lfloor \frac{n}{\log n}\right\rfloor .
\)
Lemma~\ref{lem:uniform-normalization} implies that all nodes in \(G_{n-d_n}\) have asymptotically the same normalization. Since
\[
d_n=o(n),
\]
the boundary layer is negligible. Accordingly, the next lemma, proved in Appendix~\ref{proof-main-thm-kl-critical-appendix}, shows that the bounded neighborhood of a uniformly chosen root lies in \(G_{n-d_n}\) with high probability.

\begin{lemma}[Boundary effects vanish]
\label{lem:boundary-avoidance}
Let \(d_n:=\lfloor n/\log n\rfloor\) as before, and let
\(G_n=K(n,q,k,\ell)\) with \(\ell\le 2\). If \(o_n\) is uniformly chosen from
\(V(G_n)\), then for every fixed \(r\ge 1\),
\[
    \mathbb P\bigl[B_r(G_n,o_n)\subseteq G_{n-d_n}\bigr]\to 1 .
\]
\end{lemma}

Once the exploration is confined to \(G_{n-d_n}\), the only remaining local randomness is the number of incoming shortcuts seen at each newly discovered node. The next lemma is the critical-case analogue of the previous Poisson approximation lemmas: the same mechanism applies, but here the mark dependence disappears, so the limiting parameter is simply \(q\). Its proof appears in Appendix~\ref{proof-main-thm-kl-critical-appendix}.

\begin{lemma}[Critical limit distribution for incoming shortcuts is Poisson]\label{lem:critical-poisson-input}
Fix \(i\ge 1\). Let the nodes revealed before the \(i\)-th step of the BFS exploration started from some root, \(S_i\), be a fixed finite set of nodes,
with \(|S_i|=O(1)\), and let \(u\notin S_i\) satisfy
\(
S_i\cup\{u\}\subseteq G_{n-d_n}.
\) Define
\(
\mathrm{In}(u,V(G_n)\setminus S_i)
\)
to be the number of incoming shortcuts to \(u\) from nodes from \(V(G_n)\setminus S_i\). Then, 
\[
\mathrm{In}(u,V(G_n)\setminus S_i)\xrightarrow{d}\mathrm{Poi}(q).
\]
\end{lemma}

Together, Lemmas~\ref{lem:uniform-normalization}, \ref{lem:boundary-avoidance}, and \ref{lem:critical-poisson-input} identify the local weak limit at criticality: Since the lattice neighborhood is deterministic, local shortcut cycles are negligible by Lemma~\ref{lem:no-local-shortcut-cycle}, and the number of incoming shortcuts to each newly revealed node converges in distribution to \(\mathrm{Poi}(q)\), it follows that the breadth-first exploration of \(B_r(G_n,o_n)\) converges weakly to the corresponding exploration of \(K_{=}(q,k,2)\).

\begin{remark}[Mark Collapse]\label{rem:critical-mark-collapse}
Let \(m\sim \mathrm{Unif}([0,1]^2)\) be independent of \(K_{=}(q,k,2)\), and assign the same mark \(m\) to every node of \(K_{=}(q,k,2)\). Then, this implies that \(G_n\) converges weakly in the marked local topology to this constant-marked version of \(K_{=}(q,k,2)\). Equivalently, one may disregard the marks and work with the underlying unmarked rooted graph.
\end{remark}

With this reformulation, the first-moment and second-moment arguments are the same in spirit as in the proofs of Theorem~\ref{main-thm-kl-1} and Theorem~\ref{main-thm-kl-2}: the first moment is obtained by multiplying the finitely many conditional probabilities along the breadth-first exploration while the second moment follows from the asymptotic independence of two bounded-radius explorations rooted at uniformly chosen nodes. This is carried out in Appendix~\ref{proof-main-thm-kl-critical-appendix}, and completes the proof of Theorem~\ref{main-thm-kl-critical}.

\section{The Giant Component of the Kleinberg Model under Bond Percolation: \\Proof of Theorem~\ref{thm:kleinberg-perc-giant}} \label{proof-applications-epidemic}

We distinguish two regimes.
When $\ell<2$, the Kleinberg model exhibits expansion properties with high probability \cite{flaxman}.
In this case, the conclusion follows directly from existing results showing that, for expander
sequences, the giant component under bond percolation is local and unique
~\cite{yeganeh-digraph}.

When $\ell > 2$,
we apply the ``giant is almost local'' theorem of van der Hofstad~\cite{vanderhofstad2023giant}.
Since bond percolation is a local operation, the local convergence of $(G_n,o_n)$ to
$(G,o)$ implies local convergence of the percolated graphs
$(G_n^{(p)},o_n)$ to $(G^{(p)},o)$ (see also \cite[lemma 3.1]{yeganeh-digraph}).
Thus, by~\cite[Theorem~2.2]{vanderhofstad2023giant}, it suffices to verify the coalescence
condition
\begin{equation}\label{eq:vdh-condition-proof}
\lim_{k\to\infty}\ \limsup_{n\to\infty}\ \frac{1}{|V(G_n)|^2}\,
\mathbb E\Big[\#\big\{(u,v)\in V(G_n)^2:\ |\mathcal C_n^{(p)}(u)|,|\mathcal C_n^{(p)}(v)|\ge k,\
u\not\leftrightarrow v\big\}\Big]\ =\ 0,
\end{equation}
where $\mathcal C_n^{(p)}(v)$ denotes the connected component of $v$ in $G_n^{(p)}$ and
$u\not\leftrightarrow v$ means that $u$ and $v$ are disconnected in $G_n^{(p)}$. 
Condition~\eqref{eq:vdh-condition-proof} rules out the presence of two disjoint percolation
clusters of large size.

In this regime, we verify the coalescence condition by comparison with bond percolation on
$\mathbb Z^2$. Define
\[
p_c((G, o)):=\inf\{p\in[0,1]:\mathbb P_p(o\leftrightarrow\infty)>0\},
\]
where $\mathbb P_p$ denotes the law of bond percolation with probability $p$, and
$\{o\leftrightarrow\infty\}$ is the event that the origin belongs to an infinite cluster after percolation,
equivalently, that there exists an infinite path starting from $o$ that has survived percolation.
It is known that, in the supercritical regime $p>p_c(\mathbb Z^2)=1/2$, bond percolation on
$\mathbb Z^2$ produces a unique giant component whose size is determined by the local
limit \cite{borgs2000lattice}. 
Since coalescence is known for supercritical bond percolation on $\mathbb{Z}^2$, it remains to show that the additional shortcut edges do not create two large disconnected components.

We verify the (stronger) condition
\begin{equation}\label{eq:vdH-condition-2.42}
\lim_{k\to\infty}\ \limsup_{n\to\infty}\ \mathbb P\left[|\mathcal C_n^{(p)}(o_n)|\ge k,\ o_n\notin \mathcal C_{\max}(G_n^{(p)})\right]=0,
\end{equation}
where \(\mathcal C_{\max}(G_n^{(p)})\) denotes the largest component of \(G_n^{(p)}\). Let $H_n\subseteq G_n$ denote the $\mathbb{Z}^2$ lattice subgraph (the deterministic lattice coming from the lattice edges of the Kleinberg model).  Let $H_n^{(p)}$ be its bond percolation with parameter $p$.
Let $\mathcal G_n^{\mathrm{lat}}$ denote the largest component of $H_n^{(p)}$.

Since $p>1/2=p_c(\mathbb Z^2)$, the percolation on the lattice is supercritical. 
Theorem~3.2 of~\cite{borgs2000lattice} yields
\[
\frac{|\mathcal G_{n}^{\mathrm{lat}}|}{n^2}
\;\xrightarrow{\mathbb P}\;
\theta_{\mathrm{lat}}(p),
\qquad
\theta_{\mathrm{lat}}(p):=\mathbb P_p(0\leftrightarrow\infty)>0.
\]
% In particular, for every $\alpha<\theta_{\mathrm{lat}}(p)$,
% \[
% \mathbb P\!\left(|\mathcal G_{n}^{\mathrm{lat}}|\ge \alpha n^2\right)\to 1.
% \]
Thus the lattice percolation already contributes a unique linear-size component.
Define $\mathcal G_n$ to be the component of $G_n^{(p)}$ that contains $\mathcal G_n^{\mathrm{lat}}$.
Then $\mathcal G_n$ is a candidate for the (unique) giant component of $G_n^{(p)}$.

For each node $v$, let $\vec e_v$ denote the \emph{first} long-edge connection (shortcut) from $v$
(as generated in the Kleinberg model).  This half-edge chooses an endpoint $Y_v$ with probability proportional
to $d(v,\cdot)^{-\ell}$ and then survives percolation with probability $p$ (independently of all other randomness). Because $\ell>2$, the normalization constant is upper bounded by
$4\zeta(\ell-1) < \infty,$
and hence uniformly in $n$ and $v$ the endpoint satisfies\[
\mathbb P[Y_v=u]\ \ge\ \frac{d(v,u)^{-\ell}}{4\zeta(\ell-1)}.
\]
Fix $R\ge1$, and call a node $v$ \emph{$R$-near} if $d(v,u)\le R$ (the distance in the underlying lattice) for some $u \in \mathcal G_n^{\mathrm{lat}}$. 
Then, for any \emph{$R$-near} node $v$,
\begin{equation}\label{eq:deltaR}
\mathbb P\bigl[\vec e_v \text{ survives percolation and connects to }\mathcal G_n^{\mathrm{lat}}\bigr]
\ \ge\
p\mathbb P[Y_v=u]
\ \ge\
\frac{p}{4\zeta(\ell-1) R^\ell} =: \delta_R.
\end{equation}
Moreover, the events in \eqref{eq:deltaR} are independent across distinct nodes $v$
because the first half-edges $(\vec e_v)_v$ are generated independently across $v$ and percolation is independent. Let $V_{R\mathrm{-far}}$ denote the set of nodes at lattice distance more than $R$ from $\mathcal G_n^{\mathrm{lat}}$:
\[
V_{R\mathrm{-far}}:=\{v\in V(G_n):\ d(v,u)>R \quad \forall u \in V(\mathcal G_n^{\mathrm{lat}})\}.
\]
For $p>1/2$, the dual percolation parameter $1-p<1/2$ is subcritical, hence, by \cite[ Theorem 1]{vanneuville2025exponential} there exists $c=c(p)>0$ such that
\begin{equation}\label{eq:hole-tail}
\limsup_{n\to\infty}\mathbb P[o_n\in V_{R\mathrm{-far}}]\ \le\ e^{-cR}
\qquad\text{for all }R\ge1.
\end{equation}

Fix $k > 1$, fix any $\gamma>0$, and set $R:=\lceil \gamma\log k\rceil$.
Consider the ``bad'' event
\[E_{n,k}:=\Bigl\{|\mathcal C_n^{(p)}(o_n)|\ge k\ \text{ and }\ o_n\notin \mathcal G_n\Bigr\}.\]
On $E_{n,k}$, the component $\mathcal C_n^{(p)}(o_n)$ is disjoint from $\mathcal G_n^{\mathrm{lat}}$ (since $\mathcal G_n$ contains $\mathcal G_n^{\mathrm{lat}}$).
Let
\[
V_{R\mathrm{-near}}:=\#\bigl(\mathcal C_n^{(p)}(o_n)\cap (G_n\setminus V_{R\mathrm{-far}})\bigr),
\]
be the number of $R$-near nodes in $\mathcal C_n^{(p)}(o_n)$.

\noindent\textbf{Case 1: $V_{R\mathrm{-near}}<k/2$.}
Then, at most $k/2$ nodes of $\mathcal C_n^{(p)}(o_n)$ do not lie in $V_{R\mathrm{-far}}$.
Since $o_n$ is uniform on $V(G_n)$, conditional on $C_n^{(p)}(o_n)$ the root is uniform on its own component; hence
$$\mathbb P[o_n\in V_{R\mathrm{-far}}\mid E_{n,k},V_{R\mathrm{-near}}<k/2]\ge 1/2.$$ Therefore, we have
\[
\mathbb P[E_{n,k},V_{R\mathrm{-near}}<k/2]\ \le\ 2\,\mathbb P[o_n\in V_{R\mathrm{-far}}].
\]
\noindent\textbf{Case 2: $V_{R\mathrm{-near}}\ge k/2$.}
On $E_{n,k}$, \emph{no} node of $\mathcal C_n^{(p)}(o_n)$ could have outgoing shortcut to $\mathcal G_n^{\mathrm{lat}}$,
in particular no $R$-near node in the component can have its first surviving long-rang connection to $\mathcal G_n^{\mathrm{lat}}$.
By \eqref{eq:deltaR} and independence across nodes,
\[
\mathbb P[E_{n,k},V_{R\mathrm{-near}}\ge k/2]\ \le\ (1-\delta_R)^{k/2}
\ \le\ \exp\Bigl(-\frac{\delta_R k}{2}\Bigr)
=\exp\Bigl(-\frac{p}{8\zeta(\ell-1)}\cdot \frac{k}{R^\ell}\Bigr).
\]

Combining both cases yields
\begin{equation}\label{eq:Ek-bound}
\mathbb P[E_{n,k}]
\ \le\
2\,\mathbb P(o_n\in V_{R\mathrm{-far}})
\ +\
\exp\Bigl(-\frac{p}{8\zeta(\ell-1)}\cdot \frac{k}{R^\ell}\Bigr).
\end{equation}
Using \eqref{eq:hole-tail} and $R=\lceil\gamma\log k\rceil$, for some constant $c'$,
\[
\limsup_{n\to\infty}\mathbb P[E_{n,k}]
\ \le\
2e^{-c\gamma\log k}
+\exp\Bigl(-c'\frac{k}{(\log k)^\ell}\Bigr)
=
2k^{-c\gamma}
+\exp\Bigl(-c'\frac{k}{(\log k)^\ell}\Bigr),
\]
which tends to $0$ as $k\to\infty$. Hence \eqref{eq:vdH-condition-2.42} holds for $\ell>2$ and $p>1/2$. The conclusion of Theorem~\ref{thm:kleinberg-perc-giant} then follows directly from
\cite[Theorem~2.2]{vanderhofstad2023giant}. \qed

\section{Concluding Remarks}

In this work, we established the local convergence of two foundational small-world network models—the Watts-Strogatz model and the Kleinberg model—providing theoretical insights into their asymptotic structures. For the Watts-Strogatz model, we characterized the local limit as the \(k\)-fuzz structure, demonstrating how rewiring induces a recursive process that preserves small-world properties while creating locally tree-like global structures. For the Kleinberg model, we showed a phase transition in the local limit at \(\ell = 2\), where the graph exhibits distinct behaviors depending on the decay of shortcut probabilities. 

Our results unify the analysis of these models through the framework of local convergence and highlight their implications for global network properties, such as clustering coefficients,  greedy maximal independent
set, number of spanning trees, PageRank distributions, and the performance of decentralized search algorithms. By leveraging local limits, we provide a comprehensive foundation for studying small-world networks, offering tools for analyzing both graph properties and dynamic processes on these graphs. 

\section*{Acknowledgments}

The authors thank Remco van der Hofstad for pointing out relevant literature on the Watts-Strogatz model, Christian Borgs and Thanawat Sornwanee for helpful discussions on earlier versions of this work, and Tselil Schramm for her guidance in navigating the random graph theory literature when this project began during the undergraduate years of Senem I\c{s}\i{}k.

Yeganeh Alimohammadi and Amin Saberi are supported by the Air Force Office of Scientific Research (AFOSR) under award number FA9550-23-1-0251. Senem I\c{s}\i{}k is supported by Stanford Management Science and Engineering departmental fellowship.

\bibliographystyle{plain}
\bibliography{main}

\appendix

\section{Metric on Rooted Marked Graphs: Omitted Proofs from Section~\ref{theory}} \label{metric-proof}
\begin{definition}[Metric on marked rooted graphs $\mathscr{G}_*$] \label{metric} 
Let $(G_1,o_1, M(G_1))$ and $(G_2,o_2, M(G_2))$ be two rooted marked graphs. Let $d_\Xi$ be a metric on the space of marks $\Xi$.  The distance metric that allows $\mathscr{G}_*$ to be interpreted as a metric space is as follows:
\begin{equation*}
    d_{\mathscr{G}_*}((G_1,o_1, M_1), (G_2,o_2, M_2)) = \frac{1}{1 + R^*}
\end{equation*}
where 
\begin{align*}
    R^* &= sup\{r: B_r(G_1, o_1) \simeq B_r(G_2, o_2), \\ &\exists \text{ } \pi \text{ such that } d_\Xi(M_1(u),M_2(\pi(u))\leq 1/r \text{ for all } u \in V(B_r(G_1, o_1)), \\ &d_\Xi(M_1((u, v)), M_2(\pi((u, v)))\leq 1/r \text{ for all } (u, v) \in E(B_r(G_1, o_1))\},
\end{align*}
where $\pi: V(B_r(G_1, o_1))\to V(B_r(G_2, o_2))$ running over all isomorphisms between $B_r(G_1, o_1)$ and $B_r(G_2, o_2)$. 
\end{definition}

We provide a short proof that $d_{\mathscr{G}_*}$ defines a metric on rooted marked graphs. 

\begin{lemma}
The function $d_{\mathscr{G}_*}$ defines a metric on rooted marked graphs.
\end{lemma}

\begin{proof}
Symmetry follows immediately from the definition, and
$d_{\mathscr G_*}(\mathbf G, \mathbf G)=0$ holds since identical rooted graphs agree on all radii.
If $\mathbf G_1 := (G_2,o_2,M_2) \neq (G_2,o_2,M_2) := \mathbf G_2$, then there exists some finite radius at
which they differ, so $d_{\mathscr G_*}(\mathbf G_1, \mathbf G_2) >0$ and positivity holds.

For the triangle inequality, let $R_{ij}$ be the maximum matching radius between
$\mathbf G_i := (G_i,o_i,M_i)$ and $\mathbf G_j := (G_j,o_j,M_j)$. Assume $\min\{R_{12}, R_{23}\} > 0$. Set
\[
h:=\frac{R_{12} R_{23}}{R_{12}+R_{23}}.
\]
Since $h\le \min\{R_{12}, R_{23}\}$, we may restrict a radius-$R_{12}$ matching between
$(G_1,o_1,M_1)$ and $(G_2,o_2,M_2)$, and a radius-$R_{23}$ matching
between $(G_2,o_2,M_2)$ and $(G_3,o_3,M_3)$, to the radius-$h$ balls.
On these radius-$h$ balls, the corresponding mark errors are at most $1/R12$ and $1/R23$, respectively. By the triangle inequality in $\Xi$, after composing
the two matchings the mark error is at most
\[
\frac{1}{R_{12}}+\frac{1}{R_{23}}=\frac{1}{h}.
\]
Hence $(G_1,o_1,M_1)$ and $(G_3,o_3,M_3)$ match for radius of at least $h$.
Therefore
\[
R_{13}\ge \frac{R_{12}R_{23}}{R_{12}+R_{23}}.
\]
Consequently,
\[
d_{\mathscr G_*}((G_1,o_1,M_1),(G_3,o_3,M_3))
= \frac{1}{1+R_{13}}
\le \frac{1}{1+\frac{R_{12}R_{23}}{R_{12}+R_{23}}}
\le \frac{1}{1+R_{12}}+\frac{1}{1+R_{23}}.
\]
Thus
\[
d_{\mathscr G_*}((G_1,o_1,M_1),(G_3,o_3,M_3))
\le
d_{\mathscr G_*}((G_1,o_1,M_1),(G_2,o_2,M_2))
+
d_{\mathscr G_*}((G_2,o_2,M_2),(G_3,o_3,M_3)).
\]
\end{proof}

\section{Second Moment of the WS Model: Proof of Lemma~\ref{omitted-ws-lemma}} \label{sec:omitted-ws-proof}

We closely follow the proof of Lemma~\ref{ws-1moment}, keeping track only of
the vertices deleted around \(v\). Fix \(\eta>0\). By Lemma~\ref{prelim-4},
there is \(M=M(r,\eta)\) such that
\[
    \sup_n \mathbb P\bigl[|B_{2r}(G_n,v)|>M\bigr]<\eta .
\]
Let \(\mathcal M_n:=\{|B_{2r}(G_n,v)|\le M\}\). On \(\mathcal M_n\), we have
\(|B_r(G_n,v)|\le M\) and \(|V(G_n^{v,2r})|\ge n-M\). Let
\(R_n:=\{u:d_{\mathrm{ring}}(u,B_r(G_n,v))\le 2kr\}\). Then on
\(\mathcal M_n\), \(|R_n|\le (4kr+1)M\), so only an \(O_M(1/n)\) fraction of
vertices \(w\in V(G_n^{v,2r})\) belong to \(R_n\). 

For \(w\notin R_n\), the radius-\(r\) ring neighborhood of \(w\) is
disjoint from \(B_r(G_n,v)\), so the deterministic ring part is the same as in
Lemma~\ref{ws-1moment}. Run the same BFS exploration from \(w\) in \(G_n^{v,r}\). The outgoing rewiring
decisions are unchanged, except that a visible shortcut endpoint is chosen from
\(n-c_n\) vertices, where \(c_n=O_{r,k,H_*,M}(1)\). This correction does not
affect the limit. For incoming shortcuts, conditional on the previously exposed
information \(T_{i-1}\), the number of eligible directed edges that may rewire
to the \(i\)-th explored vertex lies between \(nk-O_{r,k,H_*,M}(1)\) and
\(nk\), and each such edge chooses this vertex with probability
\(\phi/(n-c_n)\). Hence, for each fixed \(x_i\),
\[
\mathbb P\!\left[
v_{i,n}\text{ has }x_i\text{ incoming shortcuts}\mid T_{i-1},\mathcal M_n
\right]
\to
\frac{e^{-\phi k}(\phi k)^{x_i}}{x_i!}.
\]
We now define the good event for this exploration. Just before the \(j\)-th
shortcut associated with \(v_{i,n}\) is revealed, let \(S_{i,j}\) be the set of
vertices already exposed by the BFS and set
\(F_{i,j}:=\{u:d_{\mathrm{ring}}(u,S_{i,j})\le 2kr\}\cup R_n\). Let
\(\mathbf{Good}_i\) be the event that every outgoing shortcut target and every
incoming shortcut source revealed while exploring \(v_{i,n}\) lies outside its
corresponding \(F_{i,j}\). A shortcut violating this condition is called bad.
On \(\mathcal M_n\), \(|F_{i,j}|=O_{r,k,H_*,M}(1)\). Thus, conditional on
\(T_{i-1}\), \(\mathcal M_n\), and the event that \(v_{i,n}\) has \(x_i\)
incoming shortcuts, each revealed shortcut is bad with probability
\(O_{r,k,H_*,M}(1/n)\). Since at this step at most \(k+x_i\) shortcut endpoints
are revealed,
\[
    \mathbb P\!\left[
    \mathbf{Good}_i^c
    \,\middle|\,
    T_{i-1},\mathcal M_n,
    \{v_{i,n}\text{ has }x_i\text{ incoming shortcuts}\}
    \right]\to0 .
\]
It remains to check the incoming shortcut marks. Define
\[
    \mathcal G_i:=
    T_{i-1}\cap
    \{v_{i,n}\text{ has }x_i\text{ incoming shortcuts}\}\cap
    \mathbf{Good}_i\cap \mathcal M_n .
\]
Let \((M_{i,1},\dots,M_{i,x_i})\) be the ordered marks of these incoming
shortcuts. If \(N_m^{(i,j)}\) and \(N^{(i,j)}\) denote the eligible mark-\(m\)
and total eligible directed edges before the \(j\)-th incoming shortcut is
chosen, then on \(\mathcal M_n\), \(N_m^{(i,j)}=n-O_{r,k,H_*,M}(1)\) and
\(N^{(i,j)}=nk-O_{r,k,H_*,M}(1)\). Therefore, for every
\((m_1,\dots,m_{x_i})\in[k]^{x_i}\),
\[
\mathbb P\!\left[
(M_{i,1},\dots,M_{i,x_i})=(m_1,\dots,m_{x_i})
\,\middle|\,\mathcal G_i
\right]
=
\prod_{j=1}^{x_i}\frac{N_{m_j}^{(i,j)}}{N^{(i,j)}}
=
k^{-x_i}+O_{r,k,H_*,M}(1/n).
\]
Combining the convergence of incoming shortcut counts, the estimate for
\(\mathbf{Good}_i\), and the mark calculation, we obtain
\[
    \mathbb P_{w\sim V(G_n^{v,2r})}\!\left[
    B_r(G_n^{v,r},w)\simeq H_* \,\middle|\,  w \notin R_n, \mathcal M_n, 
    \right]
    \to \mu_r((G,o)=H_*).
\]

Finally, by the law of total probability, $\left|
\mathbb P_{w\sim V(G_n^{v,2r})}\!\left[
B_r(G_n^{v,r},w)\simeq H_*
\right]
-
\mu_r((G,o)=H_*)
\right| $ is at most 
\begin{align*}
\left|
\mathbb P_{w\sim V(G_n^{v,2r})}\!\left[
B_r(G_n^{v,r},w)\simeq H_* \,\middle|\,  w \notin R_n, \mathcal M_n
\right]
-
\mu_r((G,o)=H_*)
\right|
 + \mathbb P[ w \in R_n \mid \mathcal M_n] +
\mathbb P[\mathcal M_n^c].
\end{align*}
Since $\mathbb P[ w \in R_n \mid \mathcal M_n] = O_M(1/n)$, we have
\[
\limsup_{n\to\infty}
\left|
\mathbb P_{w\sim V(G_n^{v,2r})}\!\left[
B_r(G_n^{v,r},w)\simeq H_*
\right]
-
\mu_r((G,o)=H_*)
\right|
\le \eta .
\]
Since \(\eta>0\) was arbitrary, letting \(\eta\downarrow0\) proves the lemma.
\qed

\section{First Moment of the Kleinberg Model in Sub-critical Regime: \\ Omitted Proofs from Section~\ref{sec:proof-main-thm-kl-1}} \label{proof-kl-1moment}

\subsubsection*{Proof of Lemma~\ref{mark-conv}}

When $v_{i,n}$ is explored through a lattice connection, its mark is
deterministic given the mark of its patch root and the lattice offset. Hence
\[
\mathbb P\big[M(v_{i,n})=\mathcal M_{F_n}(i)\mid T_{i-1}\big]=1.
\]
It therefore remains to consider the shortcut cases.

\noindent\textbf{Outgoing shortcut: } $\mathbbm{1}(i,\mathrm{out})=1$.
Let $S_{i-1}$ be the finite set of nodes that should be excluded as targets for the shortcuts explored when or after $i$-th node is revealed conditioned on the exploration $T_{i-1}$. Then, we have
\[
\mathbb P\big[M(v_{i,n})=\mathcal M_{F_n}(i)\mid T_{i-1}\big]
=
\frac{d(v_{p(i),n},v_{i,n})^{-\ell}}
{\sum_{u\notin S_{i-1},\,u\neq v_{p(i),n}} d(v_{p(i),n},u)^{-\ell}}.
\]
Since $|S_{i-1}|=O(1)$, the denominator differs from the full normalization by
$o(n^{2-\ell})$. 
Since the lattice coordinates of \(v_{p(i),n}\) and \(v_{i,n}\) lie in the cells
corresponding to \(\mathcal M_{F_n}(p(i))\) and \(\mathcal M_{F_n}(i)\), we have
\[
\frac{d(v_{p(i),n},v_{i,n})}{n}
=
\|\mathcal M_{F_n}(p(i))-\mathcal M_{F_n}(i)\|_1 + O(n^{-1}).
\]
Then,
\[
\sum_{u\neq v_{p(i),n}} d(v_{p(i),n},u)^{-\ell}
=
n^{2-\ell}
\left(
\int_{[0,1]^2}\|\mathcal M_{F_n}(p(i))-x\|_1^{-\ell}\,dx
+o(1)
\right),
\]
while
\[
d(v_{p(i),n},v_{i,n})^{-\ell}
=
n^{-\ell}
\left(
\|\mathcal M_{F_n}(p(i))-\mathcal M_{F_n}(i)\|_1^{-\ell}
+o(1)
\right).
\]
Therefore
\[
n^2\,\mathbb P\big[M(v_{i,n})=\mathcal M_{F_n}(i)\mid T_{i-1}\big]
=
p_{\mathrm{out},\,\mathcal M_{F_n}(p(i))}\big(\mathcal M_{F_n}(i)\big)+o(1).
\]
Since $p_{\mathrm{out},\,\mathcal M_{F_n}(p(i))}$ is continuous on a neighborhood of \(\mathcal M_{F}(i)\) and $Q_i$ has
area $n^{-2}$, this is equivalent to
\[
\mathbb P\big[M(v_{i,n})=\mathcal M_{F_n}(i)\mid T_{i-1}\big]
=
\int_{Q_i} p_{\mathrm{out},\,\mathcal M_{F_n}(p(i))}(z)\,dz + o(n^{-2}).
\]
Now, for each $u\in V(G_n)$, let \(\mathrm{out}_{u,i}\) be the number of outgoing shortcuts of \(u\) whose destinations have already been revealed before the \(i\)-th BFS step.

\noindent\textbf{Incoming shortcut: } $\mathbbm{1}(i,\mathrm{in})=1$.
Here, we have 
\[
\mathbb P\big[M(v_{i,n})=\mathcal M_{F_n}(i)\mid T_{i-1}\big]
=
\frac{
(q-\mathrm{out}_{v_{i,n},i})
\dfrac{d(v_{i,n},v_{p(i),n})^{-\ell}}
{\sum_{w\neq v_{i,n}} d(v_{i,n},w)^{-\ell}}
}
{
\sum_{u \notin S_{i-1}}
(q-\mathrm{out}_{u,i})
\dfrac{d(u,v_{p(i),n})^{-\ell}}
{\sum_{w\neq u} d(u,w)^{-\ell}}
}.
\]
Since \(S_{i-1}\) is finite and only finitely many nodes \(u\) satisfy
\(\mathrm{out}_{u,i}>0\), replacing the exact denominator by
\[
q\sum_{u\neq v_{p(i),n}}
\frac{d(u,v_{p(i),n})^{-\ell}}{\sum_{w\neq u} d(u,w)^{-\ell}}
\]
changes only finitely many summands. Each such summand is \(O(n^{\ell-2})\),
because \(\sum_{w\neq u} d(u,w)^{-\ell}=\Theta(n^{2-\ell})\) for \(\ell<2\).
Hence the total error is \(o(1)\). The resulting sum is a Riemann sum for
\(\Lambda_{\mathcal M_{F_n}(p(i))}\), so
\[
q\sum_{u \neq v_{p(i),n}}
\frac{d(u, v_{p(i),n})^{-\ell}}
{\sum_{w\neq u} d(u,w)^{-\ell}}
+o(1)
=
\Lambda_{\mathcal M_{F_n}(p(i))} + o(1).
\]
Since $v_{i,n}$ has not yet been exposed, $\mathrm{out}_{v_{i,n},i}=0$. Therefore, by definition of
$p_{\mathrm{out},m}(m_{\mathrm{out}})$ and the definition of
$p_{\mathrm{in},m}(m_{\mathrm{in}})$, we have
\[
n^2\,\mathbb P\big[M(v_{i,n})=\mathcal M_{F_n}(i)\mid T_{i-1}\big]
=
\frac{q}{\Lambda_{\mathcal M_{F_n}(p(i))}}
\,p_{\mathrm{out},\,\mathcal M_{F_n}(i)}\big(\mathcal M_{F_n}(p(i))\big)
+o(1)
=
p_{\mathrm{in},\,\mathcal M_{F_n}(p(i))}\big(\mathcal M_{F_n}(i)\big)
+o(1).
\]
Because
$p_{\mathrm{in},\,\mathcal M_{F_n}(p(i))}$ is continuous on a neighborhood of \(\mathcal M_{F}(i)\) and $Q_i$ has area
$n^{-2}$, this is equivalent to
\[
\mathbb P\big[M(v_{i,n})=\mathcal M_{F_n}(i)\mid T_{i-1}\big]
=
\int_{Q_i} p_{\mathrm{in},\,\mathcal M_{F_n}(p(i))}(z)\,dz + o(n^{-2}).
\]
\qed

\subsubsection*{Proof of Lemma~\ref{incom-pois-conv}}

Let \(X_{i,n}\) be the number of incoming shortcuts of \(v_{i,n}\). We condition throughout on
\(
M(v_{i,n})=\mathcal M_{F_n}(i) \) and \( T_{i-1}.\)
For each \(u\in V(G_n)\setminus\{v_{i,n}\}\), let \(S_{u,i}\) be the set of destinations ruled out for a still-unrevealed outgoing shortcut from \(u\) by the event \(T_{i-1}\). Since \(T_{i-1}\) reveals only finitely many nodes and shortcuts, \(|S_{u,i}|=O(1)\) uniformly in \(u\). Define
\[
\tilde p_{u,i}^{(n)}:=\frac{d(u,v_{i,n})^{-\ell}}{\sum_{w\notin S_{u,i},\,w\neq u} d(u,w)^{-\ell}}.
\]
Then \(q-\mathrm{out}_{u,i}\) is the number of outgoing shortcuts from \(u\) whose destinations have not yet been revealed by \(T_{i-1}\), and each such shortcut lands at \(v_{i,n}\) with probability \(\tilde p_{u,i}^{(n)}\). Hence, if \(Y_{u,i}\) denotes the number of still-unrevealed outgoing shortcuts from \(u\) that land at \(v_{i,n}\), then
\[
Y_{u,i}\sim \mathrm{Bin}\!\left(q-\mathrm{out}_{u,i},\,\tilde p_{u,i}^{(n)}\right),
\]
and the family \(\{Y_{u,i}\}_{u\neq v_{i,n}}\) is independent across \(u\). Therefore
\[
X_{i,n}=\sum_{u\neq v_{i,n}}Y_{u,i}.
\]

Let
\[
\lambda_{i,n}:=\mathbb E\!\left[X_{i,n}\mid M(v_{i,n})=\mathcal M_{F_n}(i),\,T_{i-1}\right]
=\sum_{u\neq v_{i,n}}(q-\mathrm{out}_{u,i})\,\tilde p_{u,i}^{(n)}.
\]
Also define
\[
p_{u,i}^{(n)}:=\frac{d(u,v_{i,n})^{-\ell}}{\sum_{w\neq u} d(u,w)^{-\ell}}.
\]
Since \(|S_{u,i}|=O(1)\) and \(\sum_{w\neq u} d(u,w)^{-\ell}=\Theta(n^{2-\ell})\) uniformly in \(u\) for \(\ell<2\), we have
\[
\left|
\sum_{w\notin S_{u,i},\,w\neq u} d(u,w)^{-\ell}
-
\sum_{w\neq u} d(u,w)^{-\ell}
\right|=O(1),
\]
and therefore
\[
\tilde p_{u,i}^{(n)}=p_{u,i}^{(n)}\bigl(1+o(1)\bigr)
\]
uniformly in $u$. Moreover,
\[
\sum_{u\neq v_{i,n}} p_{u,i}^{(n)}
=
\frac{1}{q}\Lambda_{\mathcal M_{F_n}(i)}+o(1)=O(1).
\]
Since \(T_{i-1}\) reveals only finitely many shortcuts, only finitely many nodes \(u\) satisfy \(\mathrm{out}_{u,i}>0\). Hence
\[
\lambda_{i,n}
=
q\sum_{u\neq v_{i,n}} p_{u,i}^{(n)}+o(1).
\]
The sum on the right is the finite-\(n\) analogue of \(\Lambda_m\), so by the same Riemann-sum approximation used earlier,
\[
q\sum_{u\neq v_{i,n}} p_{u,i}^{(n)}
=
\Lambda_{\mathcal M_{F_n}(i)}+o(1).
\]
Thus, we have
\(
\lambda_{i,n}=\Lambda_{\mathcal M_{F_n}(i)}+o(1).
\)

Next, by Le Cam’s theorem,
\[
d_{\mathrm{TV}}\!\left(\mathcal L(X_{i,n}\mid M(v_{i,n})=\mathcal M_{F_n}(i),\,T_{i-1}),\,\mathrm{Poi}(\lambda_{i,n})\right)
\le
2\sum_{u\neq v_{i,n}}(q-\mathrm{out}_{u,i})(\tilde p_{u,i}^{(n)})^2.
\]
Since \(\tilde p_{u,i}^{(n)}=p_{u,i}^{(n)}(1+o(1))\) uniformly in \(u\), it suffices to show
\[
\sum_{u\neq v_{i,n}}(q-\mathrm{out}_{u,i})(p_{u,i}^{(n)})^2=o(1).
\]
Using \(\sum_{w\neq u} d(u,w)^{-\ell}=\Theta(n^{2-\ell})\) uniformly in \(u\), we get
\[
(p_{u,i}^{(n)})^2
=
O\!\left(\frac{d(u,v_{i,n})^{-2\ell}}{n^{2(2-\ell)}}\right).
\]
Since there are at most \(4m\) nodes at lattice distance \(m\) from \(v_{i,n}\),
\[
\sum_{u\neq v_{i,n}}(p_{u,i}^{(n)})^2
=
O\!\left(n^{-2(2-\ell)}\sum_{j=1}^{2n} j^{1-2\ell}\right).
\]
Now, notice that
\[
\sum_{j=1}^{2n} j^{1-2\ell}
=
\begin{cases}
O(n^{2-2\ell}),&\ell<1,\\
O(\log n),&\ell=1,\\
O(1),&1<\ell<2,
\end{cases}
\]
so, in all cases, we have
\[
\sum_{u\neq v_{i,n}}(p_{u,i}^{(n)})^2=o(1).
\]
Since \(q-\mathrm{out}_{u,i}\le q\), it follows that
\[
\sum_{u\neq v_{i,n}}(q-\mathrm{out}_{u,i})(p_{u,i}^{(n)})^2=o(1).
\]
Thus, we can conclude
\[
d_{\mathrm{TV}}\!\left(\mathcal L(X_{i,n}\mid M(v_{i,n})=\mathcal M_{F_n}(i),\,T_{i-1}),\,\mathrm{Poi}(\lambda_{i,n})\right)\to 0,
\]
where \(\mathcal L (\mathbf X )\) gives the distribution of random variable \(\mathbf X\).
Finally, since \(\lambda_{i,n}=\Lambda_{\mathcal M_{F_n}(i)}+o(1)\) and, for fixed \(x_i\), the map
\(
\lambda\mapsto e^{-\lambda}\frac{\lambda^{x_i}}{x_i!}
\)
is continuous on \([0,\infty)\), we obtain
\[
\mathbb P\!\left[X_{i,n}=x_i\mid M(v_{i,n})=\mathcal M_{F_n}(i),\,T_{i-1}\right]
=
\frac{e^{-\Lambda_{\mathcal M_{F_n}(i)}}\Lambda_{\mathcal M_{F_n}(i)}^{x_i}}{x_i!}+o(1).
\]
\qed

\section{Second Moment of the Kleinberg Model in Sub-critical Regime: \\ Proof of Lemma~\ref{kl-2moment}} \label{proof-kl-2moment}

For the second moment, we would like to show that 
\begin{equation} \label{goal-3}
    \begin{split}
        &\lim_{n \to \infty}\text{Var}\left(\frac{1}{n^2} \sum_{o_n\in V(G_n)}  \mathbbm{1}_r^\epsilon((G_n,o_n)\simeq H_* )  \right) \\ &=\mathbb{E}\left[\left(\frac{1}{n^2} \sum_{o_n\in V(G_n)}  \mathbbm{1}_r^\epsilon((G_n,o_n)\simeq H_*) \right)^2 \right] - \mathbb{E}\left[\frac{1}{n^2} \sum_{o_n\in V(G_n)}  \mathbbm{1}_r^\epsilon((G_n,o_n)\simeq H_*)  \right]^2 \to 0.
    \end{split}
\end{equation}

We will follow the same steps as the proof for the second moment for the Watts-Strogatz model (Lemma~\ref{ws-2moment}). The only difference is that instead of summing over all nodes, we need to sum over a subset of nodes because we need mark differences to be at most $\epsilon$. To do so, let's define 
$W_{\epsilon, m} = \left\{w \in V(G_n) \big| \left\|m - M(w) \right\|_1 < \epsilon \right\}$, where $M(w)$ is mark of node $w$, for any $\epsilon > 0$ and $m \in M$. Let the mark of the root of $H_*$ be $M(v_{0, *})$. By linearity of expectation,
\begin{equation} \label{goal-4}
    \begin{split}
&\lim_{n \to \infty} \mathbb{E}\left[\left(\frac{1}{n^2} \sum_{o_n\in V(G_n)}  \mathbbm{1}_r^\epsilon((G_n,o_n)\simeq H_*) \right)^2 \right] \\ 
    &= \lim_{n \to \infty} \frac{1}{n^4} \sum_{v_0 \in W_{\epsilon, m_*} } \sum_{\substack{w_0 \in W_{\epsilon, M(v_{0, *})} \\ w_0 \neq v_0}} \mathbb{P}\left[B_r(G_n,v_0)\simeq H_*  \right] \mathbb{P}\left[B_r(G_n,w_0)\simeq H_*  | B_r(G_n,v_0)\simeq H_* \right] \\
    &\to \lim_{n \to \infty} \frac{1}{n^2} \sum_{v_0 \in W_{\epsilon,m_*}} \frac{|W_{\epsilon,M(v_{0, *})}|-1}{n^2} \quad \mathbb{P}\left[B_r(G_n,v_0)\simeq H_*  \right] \mathbb{P}_{w_0}\left[B_r(G_n,w_0)\simeq H_*  | B_r(G_n,v_0)\simeq H_* \right]
    \end{split}
\end{equation}
where the first probability $\mathbb{P}$ is over the randomness of $G_n$ but $\mathbb{P}_{w_0}$ is \textit{also} over the randomness of root $w_0$ that is sampled uniformly at random from $W_{\epsilon, M(v_{0, *})} \backslash \{v_0\}$. Since $|W_{\epsilon,m}| \geq \sum_{j=1}^{\lfloor \epsilon n - 1 \rfloor} j = \frac{1}{2} \epsilon n ( \epsilon n - 1)$,
$$\mathbb{P}_{w_0}[w_0 \notin B_{2r}(G_n,v_0)] \geq 1- \frac{\mathbb{E}[|B_{2r}(G_n,v_0)]|}{\frac{1}{2}\epsilon^2 n^2} \to 1 \quad \text{ as $n \to \infty$}$$ 
by Lemma~\ref{prelim-4}. Recall $B_r(G_n,w_0)$ has finitely many shortcuts and $B_{2r}(G_n,v_0)$ is a finite set with probability tending to one. Lemma~\ref{lem:no-local-shortcut-cycle} implies that these shortcuts will not connect back to the finite set $B_r(G_n,w_0)$ to create a cycle with high probability. Similarly, none of them will connect to another finite and even remoter set $B_{2r}(G_n,v_0)$ in lattice distance with high probability. This implies $B_r(G_n,w_0) \cap B_{2r}(G_n,v_0) = \emptyset$ holds with probability tending to one. Then, $$\lim_{n \to \infty}  
 \mathbb{P}_{w_0}\left[B_r(G_n,w_0)\simeq H_*  | B_r(G_n,v_0)\simeq H_* \right]$$ is the same as
$$\lim_{n \to \infty} \mathbb{P}_{w_0}\left[B_r(G_n,w_0) \simeq H_* \big| B_r(G_n,v_0) \simeq H_*  \text{ and } 
B_r(G_n,w_0) \cap B_{r}(G_n,v_0) = \emptyset \right].$$

Finally, we have the following lemma giving the conditional independence of the events.
\begin{lemma} \label{conditional-independence-lemma}
    Let $\text{\textbf{B}}$ be the event that $\big(B_r(G_n,v_0) \simeq H_*\big)  \cap\big( B_r(G_n,w_0) \cap B_{r}(G_n,v_0) = \emptyset \big)$. Then, 
    \begin{equation} \label{conditional-independence}
        \lim_{n \to \infty} \mathbb{P}_{w_0}\left[B_r(G_n,w_0) \simeq H_* \big| \text{\textbf{B}} \right] = \lim_{n \to \infty} \mathbb{P}_{w_0}\left[B_r(G_n,w_0) \simeq H_* \right]
    \end{equation}
\end{lemma}

Assuming Lemma~\ref{conditional-independence-lemma}, the proof of Lemma~\ref{kl-2moment} follows from \eqref{goal-3} and \eqref{goal-4}. \qed

\subsection*{Proof of Lemma~\ref{conditional-independence-lemma}}

Let $v_{i,n}$ and $v_{i,*}$ be the $i$-th node in $(\overline{G}_n, w_0)$ and $\overline{H_*}$ respectively, let $M(v_{i,n})$ and $M(v_{i,*})$ be their corresponding marks, let $x_i$ be the number of incoming shortcuts of $v_{i,*}$, and let $T_i$ be the event that $(\overline{G}_n, w_0)$ and $H_*$ look identical up until the $i$-th node in the BFS. Then, similar to \eqref{closed-form-left}, we can write the right-hand side of \eqref{conditional-independence} as the limit of
\begin{equation*} 
     \prod_{i =0}^{|H_*|-1} \mathbb{P}[\text{$v_{i,n}$ has $x_i$ incoming shortcuts} \mid M(v_{i,n}) = M(v_{i,*}), T_{i-1}] \, \mathbb{P}[M(v_{i,n}) = M(v_{i,*}) \mid T_{i-1}],
\end{equation*}
and the left-hand side similarly, except that each probability is additionally conditioned on \textbf{B}.

Since $\overline{H_*}$ is a finite graph, to prove Lemma~\ref{conditional-independence-lemma}, it suffices to prove the following two statements:
\begin{equation} \label{statement-3}
    \lim_{n \to \infty} 
 \mathbb{P}[M(v_{i,n}) = M(v_{i,*}) \mid T_{i-1} \text{ and \textbf{B}}] = \lim_{n \to \infty}  \mathbb{P}[M(v_{i,n}) = M(v_{i,*}) \mid T_{i-1}],
\end{equation}
and
\begin{equation} \label{statement-4}
\begin{split}
    & \lim_{n \to \infty}  \mathbb{P}[\text{$v_{i,n}$ has $x_i$ incoming shortcuts} \mid M(v_{i,n}) = M(v_{i,*}), T_{i-1}, \text{ and \textbf{B}}]\\
    &= \lim_{n \to \infty} \mathbb{P}[\text{$v_{i,n}$ has $x_i$ incoming shortcuts} \mid M(v_{i,n}) = M(v_{i,*}), T_{i-1}].
\end{split}
\end{equation}

Before doing so, note that, since \textbf{B} includes \(B_r(G_n,v_0)\simeq H^*\), the set \(B_r(G_n,v_0)\) has exactly \(|V(H^*)|\) nodes on \textbf{B}. In particular, \(|B_r(G_n,v_0)|=O(1)\) on \textbf{B}.

Let's start with \eqref{statement-3}. We are going to rely on parts of the proof of Lemma~\ref{mark-conv}, so we only mention the important differences. We have two types of shortcuts, giving two cases.

\noindent\textbf{Outgoing shortcut.}
Let $S_{i-1}$ be the finite set of nodes that should be excluded as targets for the shortcuts explored when or after the $i$-th node is revealed, conditioned on the exploration $T_{i-1}$. Since \textbf{B} also includes \(B_r(G_n,w_0)\cap B_r(G_n,v_0)=\varnothing\), the prescribed target \(v_{i,n}\) is not excluded by \(B_r(G_n,v_0)\). Thus only the denominator is modified, and we have
\[
\mathbb P\big[M(v_{i,n})=\mathcal M_{F_n}(i)\mid T_{i-1} \text{ and \textbf{B}}\big]
=
\frac{d(v_{p(i),n},v_{i,n})^{-\ell}}
{\sum_{u\notin S_{i-1}\cup B_r(G_n, v_0),\,u\neq v_{p(i),n}} d(v_{p(i),n},u)^{-\ell}}.
\]
Since \(|S_{i-1}\cup B_r(G_n,v_0)|=O(1)\) on \textbf{B}, the denominator differs from the full normalization by \(o(n^{2-\ell})\), and the rest of the argument can be carried out analogously to Lemma~\ref{mark-conv}.

\noindent\textbf{Incoming shortcut.} 
Here again the prescribed source \(v_{i,n}\) is not removed by conditioning on \textbf{B}, so only the denominator changes. We have
\[
\mathbb P\big[M(v_{i,n})=\mathcal M_{F_n}(i)\mid T_{i-1} \text{ and \textbf{B}}\big]
=
\frac{
(q-\mathrm{out}_{v_{i,n},i})
\dfrac{d(v_{i,n},v_{p(i),n})^{-\ell}}
{\sum_{w\neq v_{i,n}} d(v_{i,n},w)^{-\ell}}
}
{
\sum_{u \notin S_{i-1} \cup B_r(G_n, v_0)}
(q-\mathrm{out}_{u,i})
\dfrac{d(u,v_{p(i),n})^{-\ell}}
{\sum_{w\neq u} d(u,w)^{-\ell}}
}.
\]
Since \(|S_{i-1}\cup B_r(G_n,v_0)|=O(1)\) on \textbf{B}, the same argument as in Lemma~\ref{mark-conv} applies.

Next, we move to identity \eqref{statement-4}. Again, we are going to rely on parts of the proof of Lemma~\ref{incom-pois-conv}, so we only mention the important differences. Similar to Lemma~\ref{incom-pois-conv}, let \(S_{u,i}\) be the set of destinations ruled out for a still-unrevealed outgoing shortcut from \(u\) by the event \(T_{i-1}\). Now define
\[
\tilde p_{u,i}^{(n),B}:=\frac{d(u,v_{i,n})^{-\ell}}{\sum_{w\notin S_{u,i}\cup B_r(G_n, v_0),\,w\neq u} d(u,w)^{-\ell}}.
\]
In other words, \(\tilde p_{u,i}^{(n),B}\) is the probability that one of the still-unrevealed outgoing shortcuts from \(u\) lands at \(v_{i,n}\), once we also exclude the nodes in \(B_r(G_n,v_0)\). Since \(|S_{u,i}\cup B_r(G_n,v_0)|=O(1)\) uniformly in \(u\), we have
\[
\tilde p_{u,i}^{(n),B}=\tilde p_{u,i}^{(n)}(1+o(1))
\]
uniformly in \(u\). Therefore, the same mean estimate and Le Cam bound as in Lemma~\ref{incom-pois-conv} continue to hold, and the same argument proves \eqref{statement-4}.
\qed

\section{Local Limit of the Kleinberg Model in Super-critical Regime: \\ Omitted Proofs from Section~\ref{proof-main-thm-kl-2}} \label{moment-kl}

\subsection*{Proof of Lemma
\ref{lattice-cv-1}}
Suppose \((H_*, o_*)\) is a sample from \(\mu_r^L\) such that \(\mu_r^L(H_*, o_*) > 0\), i.e., it can be sampled with positive probability. 
We will order the nodes in $L_r(G_n, o_n)$, $L_r(G, o)$ and $L_r(H_*,o_*)$ such that the nodes with smaller $y$ coordinates take precedence over others and ties are broken by giving precedence to nodes with smaller $x$ coordinates. The first node in this ordering would be the top-left node. Notice that this naturally gives an ordering of each shortcut: the shortcut that connects nodes with a smaller sum of orders takes precedence. With slight abuse of notation, let $L_r(\overline{G}_n, o_n)$,  $L_r(\overline{G}, o)$, and $\overline{H_*} = L_r(\overline{H_*},o_*)$ be the ordered versions. Note that this ordering changes depending on $r$ for $G$ and $H_*$ and is based on both $o_n$ and $r$ for $G_n$. As in the case of the Watts-Strogatz model, for the first moment, it suffices to prove that 
\[
\mu_{r, n}^L((\overline{G_n},o_n) = \overline{H_*}) \to \mu_{r}^L((\overline{G},o) = \overline{H_*}).
\]

The lattice edges in $L_r(G, o)$ and $L_r(H_*,o_*)$ are formed following the same procedure by definition. 
For the finite case, as long as the distance of $o_n$ to the boundary is at least $r$, the lattice edges observed in $L_r(G_n, o_n)$ and $L_r(H_*,o_*)$ are the same too. This happens with high probability since
\[
\frac{(n-2r)^2}{n^2}=1-\frac{4r}{n} + O(n^{-2}).
\]
Henceforth, we will assume that all lattice edges are the same. Given that $o_n$ is distant enough for this to hold, we will now explore the outgoing shortcuts of each node in $L_r(G_n, o_n)$ following the ordering to give a closed form for $\mu_{r, n}^L((\overline{G_n},o_n) = \overline{H_*})$. 

Let $v_i$ be the $i$-th node in $L_r(\overline{G}_n, o_n)$. For each $i$, let
\(
u_{i,1},\dots,u_{i,s_i}
\)
be the distinct nodes of $L_r(H_*,o_*)$ that receive an outgoing shortcut from the $i$-th node of $L_r(\overline{H_*},o_*)$, and let
\(
a_{i,1},\dots,a_{i,s_i}
\)
be the corresponding multiplicities. Thus
\(
x_i:=\sum_{m=1}^{s_i} a_{i,m}
\)
is the total number of outgoing shortcuts of the $i$-th node in $L_r(\overline{H_*},o_*)$ that connect to nodes in $L_r(H_*,o_*)$. Equivalently, the multiset of distances
\(
\{d(v_i,u_{i,m}) : 1\le m\le s_i\}
\)
with multiplicities \(a_{i,m}\) is the same data as the multiset \(\{d_{i,j}\}_{j=1}^{x_i}\).

Recall that the probability that an outgoing shortcut of $v$ connects to $w$ is $\frac{[d(v, w)]^{-\ell}}{4\zeta(\ell-1)}$ in $G$ and $\frac{[d(v, w)]^{-\ell}}{D_{v}^{(n)}}$ in $G_n$, where
\(
D_{v}^{(n)}=\sum_{\substack{u\neq v \\ u \in V(G_n)}} d(v,u)^{-\ell}.
\)
Since the $q$ outgoing shortcuts of each node are independent, the probability that the $q$ outgoing shortcuts of $v_i$ realize exactly the internal shortcut pattern prescribed by \(\overline{H_*}\) is given by the corresponding multinomial probability. Therefore,
\begin{align}
\mu_{r, n}^L((\overline{G_n},o_n) = \overline{H_*} \mid d(o_n, \partial G_n) \geq r)
= \prod_{i =0}^{|H_*|-1}
&\frac{q!}{(q-x_i)!\prod_{m=1}^{s_i} a_{i,m}!}
\left(\frac{\sum_{\substack{w \neq v_i \\ w \in V(L_r(G_n, o_n))}} [d(v_i, w)]^{-\ell}}{D_{v_i}^{(n)}}\right)^{x_i} \label{line-1-new} \\
&\prod_{m=1}^{s_i}
\left(
\frac{d(v_i,u_{i,m})^{-\ell}}
{\sum_{\substack{w \neq v_i \\ w \in V(L_r(G_n, o_n))}} [d(v_i, w)]^{-\ell}}
\right)^{a_{i,m}} \label{line-2-new} \\
&\left(1-\frac{\sum_{\substack{w \neq v_i \\ w \in V(L_r(G_n, o_n))}} [d(v_i, w)]^{-\ell}}{D_{v_i}^{(n)}}\right)^{q-x_i}. \label{line-3-new}
\end{align}
Here, the multinomial coefficient counts the number of ways to assign \(x_i\) of the \(q\) outgoing shortcuts of \(v_i\) to the prescribed internal targets with multiplicities \(a_{i,1},\dots,a_{i,s_i}\); the factor in \eqref{line-1-new} gives the probability that exactly those \(x_i\) shortcuts connect to nodes in \(L_r(G_n,o_n)\); the product in \eqref{line-2-new} gives the conditional probability that those \(x_i\) shortcuts connect to the specific internal target nodes with the prescribed multiplicities; and \eqref{line-3-new} gives the probability that the remaining \(q-x_i\) shortcuts connect outside \(L_r(G_n,o_n)\).

On the other hand, the closed form for $\mu_{r}^L((\overline{G},o) = (\overline{H_*},o_*))$ is the same as above except $L_r(G_n, o_n)$ is replaced with $L_r(G, o)$ and all $D_{v_i}^{(n)}$'s are replaced with $4\zeta(\ell -1)$. Since for fixed $r$ the lattice neighborhoods involve only finitely many nodes, it suffices to show the following lemma to conclude the proof. 

\begin{lemma} Fix any $r > 0$ and $\epsilon > 0$. For a given node $v_i$, let $D_{v_i}^{(n)}=\sum_{\substack{u\neq v_i \\ u \in V(G_n)}} d(v_i,u)^{-\ell}$.
Then, 
$$\mathbb{P}\left[\left|\frac{1}{D_{v_i}^{(n)}} - \frac{1}{4\zeta(\ell-1)}\right| \leq \epsilon \text{ for all } v_i \in V(L_r(G_n, o_n))\right] \to 1$$
as $n \to \infty$ where the probability is over the uniform distribution of $o_n$ over all the nodes with a distance at least $r$ to the boundary of $G_n$.
\end{lemma}
\begin{proof}
    Take any $\epsilon > 0 $ and choose some $\delta < \frac{16\epsilon \zeta(\ell-1)^2}{4 + 16 \epsilon  \zeta(\ell-1)}$. We know  for $\ell > 2$, $\sum_{d=1}^\infty d^{1-\ell}=\zeta(\ell-1)$. Then, there is $R_\delta$ such that $\zeta(\ell-1)-\sum_{d=1}^{R_\delta} d^{1-\ell} \leq \delta.$ If a node $v$'s distance from the boundary is at least $R_\delta$, then 
    \begin{align*}
    \sum_{\substack{v\neq u \\ v \in V(G_n)}} d(v, u)^{-\ell}&\geq \sum_{\substack{v\neq u \\ v\in B_{R_\delta}(u)}} d(v, u)^{-\ell} = 4\sum_{d=1}^{ R_\delta} d^{1-\ell}\geq 4\zeta(\ell-1)-4\delta.
    \end{align*}
    On the other hand,
    \begin{align*}
    \sum_{\substack{u\neq v \\ v \in V(G_n)}} d(u,v)^{-\ell}&\leq 4\sum_{d=1}^n d^{1-\ell}  \leq 4\zeta(\ell-1).
    \end{align*}
    We chose $\delta < \frac{16\epsilon \zeta(\ell-1)^2}{4 + 16 \epsilon  \zeta(\ell-1)}$. Assume $v_i$'s distance to the boundary is greater than $R_\delta$, we have
    \begin{align*}
        \left|\frac{1}{D_{v_i}^{(n)}} -\frac{1}{4\zeta(\ell-1)}\right| \leq \frac{4\delta}{4\zeta(\ell-1)(4\zeta(\ell-1) -4\delta)} \leq \epsilon.
    \end{align*}
    Notice $$\left|\frac{1}{D_{u_i}^{(n)}} - \frac{1}{4\zeta(\ell-1)}\right| \leq \epsilon$$ is guaranteed to happen for all  $v_i \in V(L_r(G_n, o_n))$ if $o_n$ is at least $r+R_\delta$ far in lattice distance from the boundary. The probability of this event is  at least $$\frac{(n-2(r+R_\delta))^2}{n^2} \to 1.$$
    concluding the proof of this lemma.
\end{proof}
By taking a union bound over all nodes in $L_r(G_n,o_n)$ and applying the lemma above, we get the proof of the first-moment for the $r$-lattice neighborhoods. \qed

\subsection*{Proof of Lemma
\ref{lattice-cv-2}} 

As the argument follows a similar structure to the second-moment proofs presented earlier, we provide a brief outline here, focusing on the key steps.  We are interested in showing that 
\begin{align*}
    & \lim_{n \to \infty} \Bigg(\mathbb{E}\Bigg[\Big(\frac{1}{n^2} \sum_{o_n\in V(G_n)}  \mathbbm{1}_r^L((G_n,o_n)\simeq H_*) \Big)^2\Bigg] \\
    &\quad - \frac{1}{n^4} \sum_{v_0,w_0\in V(G_n), d(v_0,w_0)>2r} \mathbb{P}\left[L_r(G_n,v_0)\simeq H_* \right] \mathbb{P}%_{w_0}
    \left[L_r(G_n,w_0)\simeq H_* %| L_r(G_n,v_0) \simeq H_*
    \right]\Bigg) = 0.
\end{align*}
First, note that when $L_r(G_n,v_0) \cap L_r(G_n,w_0) = \emptyset $, the events $\{L_r(G_n,v_0) \simeq H_*\}$ and  $\{L_r(G_n,w_0) \simeq H_*\}$ are independent. 
So to prove the result, it is enough to show that 
\begin{align}\label{goal-lattice-1m}
  \lim_{n\to\infty}\frac{1}{n^4}   \sum_{v_0,w_0\in V(G_n), d(v_0,w_0)<2r} \mathbb{P}\left[L_r(G_n,v_0)\simeq H_*, L_r(G_n,w_0)\simeq H_* %| L_r(G_n,v_0) \simeq H_*
    \right]\Big) = 0
\end{align}
Note that for any finite $r$, we know that $|L_{2r}(G_n,v_0)|$ is bounded by a constant $C$ independent of $n$. As a result, number of pairs whose $r$-lattice neighborhood overlap (i.e., $L_r(G_n,v_0) \cap L_r(G_n,w_0) \neq \emptyset$) is upper bounded by $ C n^2$. 
Since each probability is bounded by $1$, the sum of terms that appear in the left-hand side of \eqref{goal-lattice-1m} is at most $\frac{Cn^2}{n^4}$, which converges to zero as desired.

Now, one can follow the same steps as the ones in Section~\ref{secondmoment} to conclude the proof.  \qed

Recall, the definition
\begin{align*}
    \mathcal{I}_{n,j} :&= \{  \mathcal{U}_n =  (u_1, \dots, u_j) | (u_0, u_1, \dots, u_j) \text{ is a } j\text{-path in } (G_n, o_n) \text{ starting from } u_0 = o_n, \\
     &\forall 0\leq i \leq j-2, d(u_i, u_{i+1}) \leq a^{b^{i+1}}, d(u_{j-1}, u_{j}) > a^{b^{j}}\}
\end{align*}
 and $\mathcal{I}_{j}$ is the corresponding set in $(G, o)$. 

\subsection*{Proof of Lemma~\ref{path counting}} \label{path counting proof}

Since the proof is very similar to the proof of \cite[Lemma 3.3]{van2021local}, we will be brief. By Fubini's theorem for infinite series,
\begin{align*}
&\lim_{b \to \infty} \lim_{n \to \infty} \sum_{u \in \mathcal{I}_{n,j}} P(G_n, u) \left( \prod_{i=1}^{j-1} \mathbbm{1} \left( d(u_i, u_{i+1}) \leq a^{b^{i+1}} \right) \right)
\mathbbm{1} \left( d(u_{j-1}, u_j) > a^{m^{j-1}} \right) \\
& \leq  \lim_{b \to \infty} C_1 (a^b)^2 \cdots (a^{b^{j}})^2 \mathbb{E}_{(u_{j-1}, u_j) \in I_n^{2}} \left[\mathbbm{1} \left( d(u_{j-1}, u_j) > a^{b^j} \right) \right]\\
\intertext{for some constant $C_1$, and}
&\leq \lim_{b \to \infty} C_1 (a^b)^2 \cdots (a^{b^{j-1}})^2  8q (\ell-2)^{-1} a^{b^j (2-\ell)} \\
\intertext{by substituting $d = a^{b^j}$ to inequality \eqref{klein} from \cite{kleinberg}, and taking union bound over at most $2q$ shortcuts between any two nodes. Then, for some constant $C_2$, }
&\leq \lim_{b \to \infty} C_2 a^{2b + \cdots + 2b^{j-1} - b^j(\ell -2)} = 0
\end{align*}
for any $a > 1$ and $\ell > 2$. The proof for \eqref{second identity} is identical since the inequality given by \eqref{klein} also works for the limiting graph $(G, o)$. \qed

\subsection*{Proof of Lemma~\ref{lattice-to-graph}} 

\label{coupling proof}
Since the proof is very similar to the proof of Corollary 3.5 in \cite{van2021local}, we will be brief. Take $r = r(a,b,K) = a^b + a^{b^2} + a^{b^3} + \cdots + a^{b^K}$ for some $a, b > 1$. The event $\{B_K(L_r(G_n, o_n), o_n) \neq B_K(G_n, o_n)\}$ implies that there is some edge $(u_i,  u_{i+1})$ in $B_K(G_n, o_n)$ such that the lattice distance $d(u_i,  u_{i+1}) > a^{b^{L+1}}$. Let's call this event $\text{\textbf{Bad}}_{r, n}$ for $G_n$ and the corresponding event $\text{\textbf{Bad}}_{r}$ for $G$. Then, by union bound, we have
 \begin{equation} \label{prob-1}
     \mu_n(B_K(L_r(G_n, o_n), o_n) \neq B_K(G_n, o_n)) \leq \mu_n(\text{\textbf{Bad}}_{r, n}) \leq \sum_{j=1}^K \mathbb{P}[\mathcal{I}_{n,j}]
 \end{equation} and 
 \begin{equation} \label{prob-2}
     \mu(B_K(L_r(G, o), o) \neq B_K(G, o)) \leq \mu(\text{\textbf{Bad}}_{r}) \leq \sum_{j=1}^K \mathbb{P}[\mathcal{I}_{j}].
 \end{equation}
Since $K$ is fixed, the probabilities on the right-hand side of \eqref{prob-1} and \eqref{prob-2} tend to $0$ by Lemma~\ref{path counting}. \qed

\section{Local limit of the Kleinberg model at criticality: proof of Theorem~\ref{main-thm-kl-critical}}
\label{proof-main-thm-kl-critical-appendix}

\begin{proof}[Proof of Lemma~\ref{lem:boundary-avoidance}]
   Fix $r\ge 1$, and set the safe region
  $G_n^{\mathrm{safe}}:=G_{n-2(d_n+rk)}$ and the boundary strip
  \[
  S_n:=V(G_n)\setminus V(G_{n-d_n-rk}).
  \]
  Note that, by construction, every $v\in G_n^{\mathrm{safe}}$ has
  lattice distance at least $2(d_n+rk)$ from the boundary, and every
  $u\in S_n$ has lattice distance less than $d_n+rk$ from the boundary;
  hence
  \[
  d(v,u)\ \ge\ 2(d_n+rk)-(d_n+rk)\ =\ d_n+rk
  \qquad\text{for every }v\in G_n^{\mathrm{safe}},\ u\in S_n.
  \]
  Since $o_n$ is uniform on the $n\times n$ grid,
  \[
  \mathbb P[o_n\notin G_n^{\mathrm{safe}}]
  = 1-\frac{(n-4(d_n+rk))^2}{n^2}
  = O\!\left(\frac{d_n}{n}\right)
  = O\!\left(\frac1{\log n}\right)\to 0.
  \]
  Thus, with high probability, the root lies in $G_n^{\mathrm{safe}}$.

% Fix \(r\ge 1\), and define the boundary strip
% \[
% S_n:=V(G_n)\setminus V(G_{n-d_n-rk}),
% \]
% that is, the set of nodes whose lattice distance from the boundary is \(<d_n+rk\). Since \(o_n\) is uniform on the \(n\times n\) grid,
% \[
% \mathbb P[o_n\in S_n]
% =1-\frac{(n-2(d_n+rk))^2}{n^2}
% =O\!\left(\frac{d_n}{n}\right)
% =O\!\left(\frac1{\log n}\right)\to 0.
% \]
% Thus, with high probability, the root lies in \(G_{n-d_n-rk}\).

Fix \(\eta>0\). By Lemma~\ref{prelim-4}, there exists \(M=M(r,\eta)\) such that
\[
\sup_n \mathbb P\bigl(|B_r(G_n,o_n)|>M\bigr)<\eta.
\]
We will work on the event
\(
E_n:=\{|B_r(G_n,o_n)|\le M\}.
\)
We now show that, uniformly for nodes \(v\in G_{n-d_n-rk}\), the probability that a shortcut incident to \(v\) touches the strip \(S_n\) is \(o(1)\). For a single outgoing shortcut \(s\) from \(v\),
\[
\mathbb P[\text{$s$ lands in }S_n]
=\sum_{u\in S_n}\frac{d(v,u)^{-2}}{D_n(v)}.
\]
By Lemma~\ref{lem:uniform-normalization}, \(D_n(v)\ge \log n+O(1)\) uniformly in \(v\). 
% Also, since \(v\in G_{n-d_n-rk}\), every \(u\in S_n\) satisfies \(d(v,u)\ge d_n\). 
By the bound above, every $v\in G_n^{\mathrm{safe}}$ and every $u\in S_n$
  satisfy $d(v,u)\ge d_n+rk\ge d_n$.
Using that there are at most \(4j\) nodes at lattice distance \(j\), we obtain
\[
\sum_{u\in S_n} d(v,u)^{-2}
\le \sum_{j=d_n}^{n} 4j\cdot j^{-2}
=4\sum_{j=d_n}^{n}\frac1j
=O(\log(n/d_n))
=O(\log\log n).
\]
Since \(v\) has exactly \(q\) outgoing shortcuts, a union bound gives
\[
\mathbb P[\text{some outgoing shortcut from }v\text{ lands in }S_n]
\le \frac{q\,O(\log\log n)}{\log n+O(1)}
=o(1).
\]
Similarly, the expected number of incoming shortcuts from \(S_n\) to \(v\) is at most
\[
\sum_{u\in S_n} q\,\frac{d(u,v)^{-2}}{D_n(u)}
\le \frac{q}{\log n+O(1)}\sum_{u\in S_n} d(u,v)^{-2}
=O\!\left(\frac{\log\log n}{\log n}\right)
=o(1),
\]
again uniformly in \(v\). By Markov's inequality,
\[
\mathbb P[\exists \text{ an incoming shortcut from }S_n\text{ to }v]=o(1),
\]
uniformly in \(v\in G_n^{\mathrm{safe}}\). Hence, uniformly in \(v\in G_n^{\mathrm{safe}}\),
\[
\mathbb P[\text{some shortcut incident to }v\text{ touches }S_n]=o(1).
\]

On \(\{o_n\in G_n^{\mathrm{safe}}\}\cap E_n\), the ball \(B_r(G_n,o_n)\) contains at most \(M\) nodes. Therefore, by a union bound,
\[
\mathbb P\bigl[\exists v\in B_r(G_n,o_n)\cap G_n^{\mathrm{safe}}
\text{ such that some shortcut incident to }v\text{ touches }S_n \,\bigm|\, \{o_n\in G_n^{\mathrm{safe}}\}\cap E_n\bigr]
=o(1).
\]

If no shortcut incident to any node of \(B_r(G_n,o_n)\cap G_n^{\mathrm{safe}}\) touches \(S_n\), then every shortcut seen during the exploration from such nodes stays inside \(G_n^{\mathrm{safe}}\). Moreover, starting from a node in \(G_n^{\mathrm{safe}}\), each lattice step can decrease the distance to the boundary by at most \(k\). Thus, on \(\{o_n\in G_n^{\mathrm{safe}}\}\cap E_n\) and on the complement of the event above, we can conclude that
\(
B_r(G_n,o_n)\subseteq G_{n-d_n}.
\)
Consequently,
\[
\mathbb P\bigl[B_r(G_n,o_n)\subseteq G_{n-d_n}\bigr]
\ge \mathbb P(\{o_n\in G_n^{\mathrm{safe}}\}\cap E_n)-o(1)
\ge 1-\mathbb P(o_n\in S_n)-\mathbb P(E_n^c)-o(1).
\]
Taking \(\liminf_{n\to\infty}\) and using \(\mathbb P(o_n\in S_n)\to 0\), we obtain
\[
\liminf_{n\to\infty}\mathbb P\bigl[B_r(G_n,o_n)\subseteq G_{n-d_n}\bigr]\ge 1-\eta.
\]
Since \(\eta>0\) was arbitrary, the claim follows.
\end{proof}

\begin{proof}[Proof of Lemma~\ref{lem:critical-poisson-input}]
For each \(v\in V(G_n)\), let \(\mathrm{out}_{v,i}\in\{0,\dots,q\}\) denote the
number of outgoing shortcuts from \(v\) that have already been revealed before
step \(i\). For each \(v\notin S_i\cup\{u\}\), let \(Y_{v,i}\) denote the number
of still-unrevealed outgoing shortcuts from \(v\) that land at \(u\). Then, we have
\[
Y_{v,i}\sim \mathrm{Bin}\!\left(q-\mathrm{out}_{v,i},\,p_{v,i}\right),
\qquad
p_{v,i}:=\frac{d(v,u)^{-2}}{D_n(v)},
\]
and the family \(\{Y_{v,i}\}_{v\notin S_i\cup\{u\}}\) is independent. Hence, the number of incoming shortcuts satsify
\[
\mathrm{In}(u,V(G_n)\setminus S_i)
=\sum_{v\notin S_i\cup\{u\}} Y_{v,i}.
\]
Define
\[
\lambda_n:=\sum_{v\notin S_i\cup\{u\}} (q-\mathrm{out}_{v,i})p_{v,i},
\]
giving the mean of \(\mathrm{In}(u,V(G_n)\setminus S_i)\). We
first show that \(\lambda_n=q+o(1)\).
Set
\[
b_n:=\Big\lfloor \frac{d_n}{2}\Big\rfloor \text{ for } 
d_n:=\left\lfloor \frac{n}{\log n}\right\rfloor,
\qquad
N_n:=\{w\in V(G_n): d(w,u)\le b_n\}.
\]
Since \(u\) is at lattice distance at least \(d_n\) from the boundary, every
\(w\in N_n\) is at lattice distance at least \(d_n-b_n\ge b_n\) from the
boundary. Repeating the proof of Lemma~\ref{lem:uniform-normalization} with
\(b_n\) in place of \(d_n\), we obtain
\[
D_n(w)=4\log n+O(\log\log n)
\qquad\text{uniformly for }w\in N_n.
\]
Therefore, we get
\[
\sum_{w\in N_n} q\,p_{w,i}
=\frac{q}{4\log n+O(\log\log n)}\sum_{j=1}^{b_n}4j\cdot j^{-2}
=q\,\frac{4\log b_n+O(1)}{4\log n+O(\log\log n)}
=q+o(1),
\]
since \(\log b_n/\log n\to 1\).

For the complement \(N_n^c\), Lemma~\ref{lem:uniform-normalization} gives the
uniform lower bound
\[
D_n(w)\ge \log n+O(1)
\qquad\text{for all }w\in V(G_n),
\]
giving
\[
\sum_{w\notin N_n} q\,p_{w,i}
\le \frac{q}{\log n+O(1)}\sum_{j>b_n}4j\cdot j^{-2}
=O\!\left(\frac{\log(n/b_n)}{\log n}\right)
=o(1).
\]

It remains to control the correction coming from the factors
\(\mathrm{out}_{v,i}\). Because $|S_i| = O(1)$ by assumption, we have
\(
\sum_{v\in V(G_n)} \mathrm{out}_{v,i} = C,
\)
for some constant $C$.
Using again \(D_n(v)\ge \log n+O(1)\) and \(d(v,u)\ge 1\), we get
\[
p_{v,i}\le \frac{1}{\log n+O(1)},
\qquad
\sum_{v\notin S_i\cup\{u\}} \mathrm{out}_{v,i}p_{v,i}
\le \frac{C}{\log n+O(1)}
=o(1).
\]
Combining the near contribution, the far contribution, and this correction term
yields
\(
\lambda_n=q+o(1).
\)
We now identify the full conditional law. Here and below,
\(\mathcal L(X)\) denotes the law of \(X\). Since
\(\mathrm{In}(u,V(G_n)\setminus S_i)\) is a sum of independent Bernoulli random
variables, Le Cam's inequality gives
\[
d_{\mathrm{TV}}\!\left(
\mathcal L\bigl(\mathrm{In}(u,V(G_n)\setminus S_i)\bigr),
\,\mathrm{Poi}(\lambda_n)\right)
\le
2\sum_{v\notin S_i\cup\{u\}} (q-\mathrm{out}_{v,i})p_{v,i}^2.
\]
Using again \(D_n(v)\ge \log n+O(1)\), we obtain
\(
p_{v,i}^2\le \frac{d(v,u)^{-4}}{(\log n+O(1))^2}.
\)
Therefore, we have
\[
\sum_{v\notin S_i\cup\{u\}} (q-\mathrm{out}_{v,i})p_{v,i}^2
\le
\frac{q}{(\log n+O(1))^2}\sum_{j\ge 1}4j\cdot j^{-4}
=O\!\left(\frac{1}{\log^2 n}\right)\to 0,
\]
and
\[
d_{\mathrm{TV}}\!\left(
\mathcal L\bigl(\mathrm{In}(u,V(G_n)\setminus S_i)\bigr).
\,\mathrm{Poi}(\lambda_n)\right)\to 0
\]
Since \(\lambda_n\to q\), we
also have
\(
d_{\mathrm{TV}}\bigl(\mathrm{Poi}(\lambda_n),\mathrm{Poi}(q)\bigr)\to 0.
\)
By the triangle inequality,
\[
d_{\mathrm{TV}}\!\left(
\mathcal L\bigl(\mathrm{In}(u,V(G_n)\setminus S_i) \bigr),
\,\mathrm{Poi}(q)\right)\to 0.
\]
\end{proof}

Next, we move on to the proof of the first moment.

\begin{lemma}[First moment at criticality]
For every fixed \(r\ge 1\) and every finite rooted graph \(H^*=B_r(H^*,o^*)\) with
\(\mu_r\left[(G,o)=H^*\right]>0\),
\[
\mu_{r,n}\left[(G_n,o_n)=H^*\right]\to \mu_r\left[(G,o)=H^*\right].
\]
\end{lemma}

\begin{proof}
Fix \(r\) and \(H^*\), and let
\(
V(H^*)=\{v_0^*,v_1^*,\dots,v_{m-1}^*\}
\)
be the BFS ordering from Section~\ref{sec:node-marks-ordering}. For each \(i\), let \(x_i\) be the number of incoming shortcuts incident
to \(v_i^*\) that are discovered at the \(i\)-th step of the breadth-first exploration. Let \(T_i\) be the event that the first \(i+1\) revealed nodes in the BFS exploration of \(B_r(G_n,o_n)\)
agree with \(H^*\) up to that stage. Also let \(A_n\) be the event that this exploration stays inside \(G_{n-d_n}\)
and contains no cycle with a shortcut edge. By Lemma~\ref{lem:boundary-avoidance} and Lemma~\ref{lem:no-local-shortcut-cycle},
\(
\mathbb P\left[A_n\right]\to 1.
\)
On \(A_n\), each shortcut revealed by the BFS leads to a new patch. Moreover, conditional on
\(T_{i-1}\cap A_n\), Lemma~\ref{lem:critical-poisson-input} gives, uniformly over all admissible histories,
\[
\mathbb P\left[v_{i,n}\text{ has }x_i\text{ incoming shortcuts}\mid T_{i-1},A_n\right]
\to e^{-q}\frac{q^{x_i}}{x_i!}.
\]
The lattice edges are deterministic, every node has exactly \(q\) outgoing shortcuts, and once shortcut
cycles are excluded, an outgoing shortcut leads to an independent copy of \(K_{=}(q,k,2)\) while an incoming
shortcut leads to an independent reduced copy of \(K_{=}(q,k,2)\). Hence, the joint law of the BFS data at
each explored node converges to the corresponding law in the critical limit object.

Since \(|H^*|<\infty\), multiplying the finitely many conditional probabilities along the breadth-first
exploration yields
\[
\mu_{r,n}\left[(G_n,o_n)=H^*\mid A_n\right]\to \mu_r\left[(G,o)=H^*\right].
\]
Since \(\mathbb P\left[A_n\right]\to 1\), the same convergence holds unconditionally.
\end{proof}

Next, the proof of second moment: The argument here is the same as Lemma~\ref{ws-2moment} or Appendix~\ref{proof-kl-2moment}, but simpler because marks are no longer relevant.

\begin{lemma}[Second moment at criticality]
For every fixed \(r\ge 1\) and every finite rooted graph \(H^*=B_r(H^*,o^*)\),
\[
\operatorname{Var}\!\left(\frac1{n^2}\sum_{u\in V(G_n)} \mathbf 1_r\big((G_n,u)=H^*\big)\right)\to 0.
\]
\end{lemma}

\begin{proof}
By symmetry, it is enough to show that for an arbitrary node \(v\in V(G_n)\) and a uniformly chosen
\(w\in V(G_n)\setminus\{v\}\),
\[
\mathbb P_w\bigl[B_r(G_n,v)=H^*,\,B_r(G_n,w)=H^*\bigr]
=
\mathbb P\bigl[B_r(G_n,v)=H^*\bigr]^2+o(1).
\]

For a fixed realization of \(G_n\) and fixed \(v\in V(G_n)\), by Lemma~\ref{prelim-4},
\[
\mathbb P_w\bigl[d_{G_n}(v,w)\le 2r\bigr]
\le
\frac{\mathbb E|B_{2r}(G_n,v)|}{n^2-1}
=
O\!\left(\frac1{n^2}\right)\to 0.
\]
Hence, it suffices to work on the event \( \{d_{G_n}(v,w)>2r\} \). Let \(G_n^{v,2r}:=G_n[V(G_n)\setminus V(B_{2r}(G_n,v))]\). Then,
\[
\mathbb P_w\bigl[B_r(G_n,w)=H^* \mid B_r(G_n,v)=H^*,\,d_{G_n}(v,w)>2r\bigr]
=
\mathbb P_{w\sim V(G_n^{v,2r})}\bigl[B_r(G_n^{v,2r},w)=H^* \mid B_r(G_n,v)=H^*\bigr].
\]

Fix \(\eta>0\). By Lemma~\ref{prelim-4}, there exists \(M=M(2r,\eta)\) such that
\(
\sup_n \mathbb P\bigl[|B_{2r}(G_n,v)|>M\bigr]<\eta.
\)
Recall \(G_{n-d_n}\) denotes the induced subgraph of the \(n\times n\) grid on the nodes whose lattice distance from the boundary is at least \(d_n := \left\lfloor \frac{n}{\log n}\right\rfloor
\).  Let \(A_v\) be the event
\[
\bigl\{ |B_{2r}(G_n,v)|\le M,\;
B_{2r}(G_n,v)\subseteq G_{n-d_n},
\ \text{and}\
B_{2r}(G_n,v)\text{ contains no cycle with a shortcut edge}\bigr\}.
\]
By the choice of \(M\), Lemma~\ref{lem:boundary-avoidance}, and Lemma~\ref{lem:no-local-shortcut-cycle},
\(
\limsup_{n\to\infty}\mathbb P(A_v^c)\le \eta .
\)

We now condition on \(A_v\). Run the breadth-first exploration of \(B_r(G_n^{v,2r},w)\). Let
\(V(H^*)=\{v_0^*,v_1^*,\dots,v_{m-1}^*\}\)
be the BFS ordering of \(H^*\), and let \(x_i\) be the number of incoming shortcuts incident to \(v_i^*\) discovered
at the \(i\)-th step. Let \(T_i^v\) be the event that the first \(i+1\) revealed nodes in the BFS exploration of
\(B_r(G_n^{v,2r},w)\) agree with \(H^*\) up to that stage. Also let \(A_{v,w}\) be the event that this \(w\)-exploration
stays inside \(G_{n-d_n}\) and contains no cycle with a shortcut edge.

Since on \(A_v\) we have \(|V(B_{2r}(G_n,v))|\le M\), the graph \(G_n^{v,2r}\) is obtained from \(G_n\) by deleting at most \(M\) nodes.
Thus, the proofs of Lemma~\ref{lem:boundary-avoidance} and Lemma~\ref{lem:no-local-shortcut-cycle} apply in the same way to the
\(w\)-exploration in \(G_n^{v,2r}\), and therefore,
\[
\mathbb P(A_{v,w}\mid A_v)\to 1.
\]

On \(A_v\cap A_{v,w}\), the exploration is identical to the first-moment exploration except that the nodes of
\(V(B_{2r}(G_n,v)\) are no longer available as shortcut endpoints. Since \(|V(B_{2r}(G_n,v))|\le M\), following the proof of
Lemma~\ref{lem:uniform-normalization}, removing \(V(B_{2r}(G_n,v)\) changes each outgoing shortcut probability by \(o(1)\),
uniformly over all admissible histories. Likewise, following the proof of
Lemma~\ref{lem:critical-poisson-input}, the conditional law of the incoming shortcuts at each revealed node
is still \(\mathrm{Poi}(q)+o(1)\), uniformly over all admissible histories in the \(w\)-exploration. Therefore, for each
\(i=0,\dots,m-1\), the conditional probability of the \(i\)-th BFS step in \(G_n^{v,2r}\) differs from the corresponding
conditional probability in a new exploration of \(G_n\) by \(o(1)\), uniformly over all realizations of the
\(v\)-exploration on \(A_v\).

Since \(m=|V(H^*)|<\infty\), multiplying these finitely many conditional probabilities along the breadth-first
exploration yields
\[
\mathbb P_{w\sim V(G_n^{v,2r})}\bigl[B_r(G_n^{v,2r},w)=H^* \mid B_r(G_n,v)=H^*,\,A_v\bigr]
=
\mathbb P\bigl[B_r(G_n,w)=H^*\bigr]+o(1).
\]
Therefore, we have
\begin{align*}
&\mathbb P_w\bigl[B_r(G_n,v)=H^*,\,B_r(G_n,w)=H^*,\,d_{G_n}(v,w)>2r\bigr] \\
&\qquad=
\mathbb P\bigl[B_r(G_n,v)=H^*,\,A_v\bigr]
\Bigl(\mathbb P\bigl[B_r(G_n,w)=H^*\bigr]+o(1)\Bigr)
+O\bigl(\mathbb P(A_v^c)\bigr).
\end{align*}
Since
\[
\mathbb P\bigl[B_r(G_n,v)=H^*,\,A_v\bigr]
=
\mathbb P\bigl[B_r(G_n,v)=H^*\bigr]+O\bigl(\mathbb P(A_v^c)\bigr),
\]
we obtain
\[
\mathbb P_w\bigl[B_r(G_n,v)=H^*,\,B_r(G_n,w)=H^*,\,d_{G_n}(v,w)>2r\bigr]
=
\mathbb P\bigl[B_r(G_n,v)=H^*\bigr]^2+O(\eta)+o(1).
\]
Combining this with \(\mathbb P_w[d_{G_n}(v,w)\le 2r]\to 0\), we get
\[
\mathbb P_w\bigl[B_r(G_n,v)=H^*,\,B_r(G_n,w)=H^*\bigr]
=
\mathbb P\bigl[B_r(G_n,v)=H^*\bigr]^2+O(\eta)+o(1).
\]
Noting the choice of \(\eta>0\) was arbitrary concludes the proof.
\end{proof}

Finally, by the two lemmas above and the same Chebyshev argument as in the proof of Lemma~\ref{main-lemma-ws},
\[
\frac1{n^2}\sum_{u\in V(G_n)} \mathbf 1_r\big((G_n,u)=H^*\big)
\xrightarrow{\mathbb P}
\mu_r\left[(G,o)=H^*\right]
\]
for every fixed \(r\) and every finite rooted graph \(H^*\) with \(\mu_r\left[(G,o)=H^*\right]>0\). This is precisely
local convergence in probability of \(K(n,q,k,2)\) to \(K_{=}(q,k,2)\). Hence, Theorem~\ref{main-thm-kl-critical}
follows. \qed

\section{Proof of Applications: Clustering Coefficient} \label{proof-applications-cluster}

\subsection*{Proof of Corollary \ref{cor1}}

Let \((G,o)\sim \mu\) be the local limit from Theorem~\ref{main-thm-ws}.  \cite[Theorem~2.23]{remco} yields
\[
C_{\mathrm{local}}
=\frac{1}{|V|}\sum_{u\in V} C(u)
\xrightarrow{\mathbb P}
\mathbb E_{\mu}\left[\frac{\Delta_G(o)}{d_o(d_o-1)}\right].
\]
For the global clustering coefficient,  \cite[Theorem~2.22]{remco} gives
\[
C_{\mathrm{global}} \xrightarrow{\mathbb P} \frac{\mathbb E_{\mu}[\Delta_G(o)]}{\mathbb E_{\mu}[d_o(d_o - 1)]}.
\]
Thus, it remains to compute the numerator and denominator.

Recall that the outgoing and incoming shortcuts in the limiting graph connect to different full or reduced \(k\)-paths and, hence, cannot be part of a triangle (or any cycle). Thus, the triangles that contain the root can only be created by the non-rewired ring edges. The number of ordered pairs of adjacent neighbors of the root in a full \(k\)-path is \(3k(k-1)\), so twice the number of triangles containing the root is \(3k(k-1)\). Each such triangle survives the rewiring procedure exactly when its three constituent ring edges are not rewired, which occurs with probability \((1-\phi)^3\). Therefore,
\[
\mathbb E_{\mu}[\Delta_G(o)] = 3k(k-1)(1-\phi)^3.
\]
Next, the degree of the root under \(\mu\) is distributed as
\[
k+ \mathrm{Bin}(k,1-\phi) + \mathrm{Poi}(\phi k),
\]
corresponding to the outgoing ring edges, incoming ring edges that were not rewired, and incoming shortcuts, respectively. Hence,
\[
\mathbb{E}[d_o^2 - d_o]
= \mathrm{Var}(d_o) + \mathbb{E}[d_o]^2 - \mathbb{E}[d_o]
= k(1-\phi)\phi + \phi k + 4k^2 - 2k.
\]
Therefore,
\[
\mathbb E_{\mu}[d_o(d_o-1)] = k(1-\phi)\phi + \phi k + 4k^2 - 2k.
\]
Substituting these expressions gives
\[
\frac{\mathbb E_{\mu}[\Delta_G(o)]}{\mathbb E_{\mu}[d_o(d_o-1)]}
=
\frac{3k(k-1)(1-\phi)^3}{k(1-\phi)\phi + \phi k + 4k^2 - 2k}
=
\frac{3(k-1)}{2(2k-1) + \phi (2 - \phi)} (1-\phi)^3.
\]
\qed

\subsection*{Proof of Corollary \ref{cor2}}

For \(\ell\le 2\), each shortcut in the local limit connects to a different \(k\)-lattice and therefore cannot be part of a triangle; hence triangles arise only from lattice edges. Consequently,
\[
\Delta_G(o)=\Delta_{L_k}
\]
with high probability.
First, suppose \(\ell<2\). Under the local limit \(\mu\) from Theorem~\ref{main-thm-kl-1}, the degree of the root decomposes as
\[
d_o = A + X_m,
\qquad A:=2k(k+1)+q,
\]
where, conditional on the mark \(m\), the random variable \(X_m\) is Poisson with mean \(\Lambda_m\). Then, Theorem~2.23 in~\cite{remco} yields
\[
C_{\mathrm{local}}
=\frac{1}{|V|}\sum_{u\in V} C(u)
\xrightarrow{\mathbb P}
\mathbb E\!\left[\frac{\Delta_{L_k}}{(A+X_m)(A+X_m-1)}\right].
\]
For the global clustering coefficient, \cite[Theorem~2.22]{remco} yields
\[
C_{\mathrm{global}}
\xrightarrow{\mathbb P}
\frac{\mathbb E_{\mu}[\Delta_G(o)]}{\mathbb E_{\mu}[d_o(d_o-1)]}.
\]
Since \(\Delta_G(o)=\Delta_{L_k}\), it remains to compute the denominator. Conditioning on \(m\),
\[
\mathbb E[d_o(d_o-1)\mid m]
=\mathrm{Var}(d_o\mid m)+\mathbb E[d_o\mid m]^2-\mathbb E[d_o\mid m]
=\Lambda_m+(A+\Lambda_m)^2-(A+\Lambda_m)
=(A+\Lambda_m)^2-A.
\]
Averaging over \(m\) gives
\[
\mathbb E_{\mu}[d_o(d_o-1)]
=\mathbb E_m\!\big[(A+\Lambda_m)^2\big]-A
= A^2 - A + 2A\,\mathbb E[\Lambda_m] + \mathbb E[\Lambda_m^2].
\]
Substituting \(\Delta_G(o)=\Delta_{L_k}\) proves the formula for \(C_{\mathrm{global}}\).

For \(\ell=2\), the same argument applies, except that \(\Lambda_m=q\) for all \(m\). Thus, if \(X\sim \mathrm{Poi}(q)\), then
\(
d_o = A+X,
\)
and
\[
C_{\mathrm{local}}
=\frac{1}{|V|}\sum_{u\in V} C(u)
\xrightarrow{\mathbb P}
\mathbb E\!\left[\frac{\Delta_{L_k}}{(A+X)(A+X-1)}\right].
\]
Moreover,
\[
\mathbb E[d_o(d_o-1)]
=\mathbb E[(A+X)(A+X-1)]
=(A+q)^2-A,
\]
so
\[
C_{\mathrm{global}}\xrightarrow{\mathbb P}\frac{\Delta_{L_k}}{(A+q)^2-A}.
\]
This completes the proof.
\qed

\end{document}